%% file: main.tex
\documentclass{article}
\usepackage{graphicx} % Required for inserting images
\input{preamble}

\title{Spectral Intertwining Operators}
\author{Qiyuan Chen}
\date{}

\begin{document}

\maketitle

\begin{abstract}
    We study spectral intertwining operators between spectral Eisenstein series $\Eis_{P^\vee}$, $\Eis_{Q^\vee}$ for two parabolic subgroups $P, Q$ of a $p$-adic reductive group $G$ with the same Levi subgroup $M$, inspired by the analogy with the classical intertwining operators between parabolic induced representations of $p$-adic reductive groups. 
    In particular, we construct the normalized (canonical) intertwining operator that satisfies the transitivity and an unnormalized intertwining operator that is adjoint to a rational section of an analog of the Bruhat--Mackey's filtration. Moreover, the normalized and unnormalized intertwining operators differ by a ratio of L-functions, analogously to the Langlands conjecture about classical ones up to units.
    Finally we prove that the spectral intertwining operators correspond to classical ones up to units
    under any conjectural categorical local Langlands correspondence.
\end{abstract}

\tableofcontents

\input{1}

\input{2}

\input{3}

\input{4}

\bibliographystyle{alpha}
\bibliography{biblio}

\end{document}

%% file: preamble.tex
\usepackage[a4paper, portrait, margin = .75in ]{geometry}
\usepackage[utf8]{inputenc}
\usepackage{enumitem}
\usepackage{hyperref}
\usepackage{amsmath}
\usepackage{amssymb}
\usepackage{amsthm}
\usepackage{stmaryrd}
\usepackage{tikz-cd}
\usepackage{ bbold }
\usepackage{tikz}
\usetikzlibrary{matrix, calc}

\newtheorem{theorem}{Theorem}[subsection]
\newtheorem{proposition}[theorem]{Proposition}%[section]
\newtheorem{lemma}[theorem]{Lemma}
\newtheorem{corollary}[theorem]{Corollary}
\newtheorem{conjecture}[theorem]{Conjecture}

\theoremstyle{definition}
\newtheorem{definition}[theorem]{Definition}

\theoremstyle{remark}
\newtheorem{remark}[theorem]{Remark}

\newcommand{\Hom}{\operatorname{Hom}}

\newcommand{\BL}{{\mathbb {L}}}
 
\newcommand{\BN}{{\mathbb {N}}}

\newcommand{\CC}{{\mathcal {C}}} 
\newcommand{\CD}{{\mathcal {D}}}
 
\newcommand{\CF}{{\mathcal {F}}}
\newcommand{\CG}{{\mathcal {G}}} 
\newcommand{\CH}{{\mathcal {H}}}

\newcommand{\CN}{{\mathcal {N}}}
\newcommand{\CO}{{\mathcal {O}}}

\newcommand{\CX}{{\mathcal {X}}}
 
\newcommand{\CZ}{{\mathcal {Z}}}
\newcommand{\fn}{{\mathfrak {n}}}
\newcommand{\Par}{\operatorname{Par}}
\newcommand{\Coh}{\operatorname{Coh}}
\newcommand{\Eis}{\operatorname{Eis}}
\newcommand{\eis}{\operatorname{eis}}
\newcommand{\CT}{\operatorname{CT}}
\newcommand{\ct}{\operatorname{ct}}
\newcommand{\ad}{\mathrm{Ad}}
\newcommand{\End}{\operatorname{End}}
\newcommand{\ft}{\mathfrak{t}}
\newcommand{\Fr}{\mathrm{Fr}}

\newcommand{\res}{\operatorname{res}}
\newcommand{\fg}{\mathfrak{g}}
\newcommand{\im}{\operatorname{im}}
\newcommand{\Spf}{\operatorname{Spf}}
\newcommand{\Gal}{\operatorname{Gal}(G_0/G)}
\newcommand{\simto}{\operatorname{\xrightarrow{\sim}}}
\newcommand{\irr}{\operatorname{Irr}}

%% file: 1.tex
\section{Introduction}

\subsection{Motivation}

This article is concerned with the intertwining operators on the spectral side of the Langlands correspondence. Firstly, we briefly recall the classical intertwining operators. Let $F$ be a $p$-adic local field and $G$ be a reductive group over $F$. Suppose $G$ has a Levi subgroup $M$ and two parabolic subgroups $(P, Q)$ with Levi component $M$. There is a family of (rational) homomorphisms between parabolic inductions $i_P^G$ and $i_Q^G$, called unnormalized intertwining operators:
\[
J_{Q|P}(\pi) \colon i_P^G(\pi) \to i_Q^G(\pi)
\]
where $\pi$ is a irreducible representation of $M(F)$ over some appropriate field, such that the composition 
\[
\pi \to r_{Q}^G \circ i_P^G (\pi)/F^{>1} \to \pi
\]
of the inclusion (given by the Mackey's formula) and the adjunction of $J_{Q|P}$ is the identity, where $F^{>1}$ is a submodule of $r_{Q}^G \circ i_P^G (\pi)$ defined by the Bruhat--Mackey's filtration. Note that this construction of intertwining operators is due to Dat (\cite{Dat2005NuTempered}, (2.11)) and is compatible with the classical definition via integrals by \textit{loc. cit.}, Lemma 7.12. The intertwining operator is a classical object in the representation theory of $p$-adic reductive groups, which will reflect the semisimplicity of representations such as $i_P^G \pi$ and $r_Q^G \circ i_P^G \pi$.

Though the definition of the intertwining operators is natural, they usually do not satisfy the transitivity $J_{R|Q} \circ J_{Q|P} = J_{R|P}$ for another parabolic $R$ with Levi component $M$,
which is useful to decompose parabolically induced representations.
The Langlands conjecture (\cite{langlands1976functional}, Appendix II) measures this failure: it is necessary to compatibly multiplying the unnormalized intertwining operators by some rational functions to obtain the transitivity. More precisely:
\footnote{As Jean-Loup Waldspurger indicated to the author, the original version of the Langlands conjecture has a sign issue, which can be corrected by some $\epsilon$ factor at $1/2$. The following statement is suggested by Arthur in \cite{arthur2013endoscopic}, section 2.3.}

\begin{conjecture}[Langlands]
\label{conjecture-Langlands}
    For irreducible representations $\pi$,
    the normalized intertwining operators 
    \[
    J_{Q|P}^\circ (\pi)
    := \frac{\epsilon(0, \pi, \ad_{\fn_{Q^\vee}/\fn_{P^\vee}})}{\epsilon(1/2, \pi, \ad_{\fn_{Q^\vee}/\fn_{P^\vee}})}\frac{L(1, \pi, \ad_{\fn_{Q^\vee}/\fn_{P^\vee}})}{L(0, \pi, \ad_{\fn_{Q^\vee}/\fn_{P^\vee}})}J_{Q|P}(\pi),
    \]
    where the $\epsilon$ and $L$'s denote the conjectural $\epsilon$-factor and L-functions for automorphic representations,
    satisfy the transitivity 
    \[J^\circ_{R|Q} \circ J^\circ_{Q|P} = J^\circ_{R|P}.\]
\end{conjecture}

In previous results, such as in \cite{arthur2013endoscopic}, Proposition 2.3.1 and
\cite{atobe2026localintertwiningrelationscotempered}, Proposition 1.7.2, the Langlands conjecture is proved for quasi-split classical groups $G$,
where the $\epsilon$-factors and $L$-functions for automorphic representations are constructed purely on the automorphic side. A natural question is to ask whether under a proposed construction of local Langlands correspondence $\pi \mapsto \phi_\pi$, such as the construction of Fargues--Scholze (\cite{fargues2024}), the L-functions defined by $L(s, \pi, V):=L(s, \phi_\pi, V)$ via the Galois representations and similarly the $\epsilon$-factors will satisfy the Langlands conjecture. Note that the L-parameter $\phi_\pi$ by Fargues--Scholze is the ``semisimplified'' parameter and thus the $L(s,\phi_\pi,V)$ is not the classical $L(s,\pi,V)$ in general. 

The main goal of this article is to construct similar intertwining operators on the spectral side of the local Langlands correspondence, 
and then verify the Langlands conjecture under the conjectural categorical local Langlands correspondence.
For this purpose, we firstly introduce the corresponding objects on the spectral side that are concerned in the theory of intertwining operators.
 
Denote by $G^*$ the quasisplit inner form of $G$. We will always assume that $G$ is an extended pure inner form: $G=(G^*)_b$ for an element $b$ in the Kottwitz set $B(G^*)_{bas}$.
Recall that the categorical Langlands correspondence, conjectured by Fargues--Scholze (\cite{fargues2024}, Conjecture X.1.4),
(which is compatible with the set-theoretical correspondence $\pi\mapsto \phi_\pi$ in loc. cit.)
says that over $\overline{\mathbb{Q}_\ell}$, there should be an equivalence of $\infty$-categories:
\[
D_{\text{lis}}(\operatorname{Bun}_{G^*})^\omega \simeq D\operatorname{Coh}^{\text{qc}}(\operatorname{Par}_{G^\vee}),
\]
where:
\begin{itemize}
    \item $D_{\text{lis}}(\operatorname{Bun}_{G^*})^\omega$ is the subcategory of the compact objects in the category of lisse \'etale sheaves over the moduli space of principal $G^*$-bundles over the Fargues--Fontaine curve.
    \item $D\operatorname{Coh}^{\text{qc}}(\operatorname{Par}_{G^\vee})$
    is the category of derived coherent sheaves with quasi-compact support over the moduli space of $L$-parameters of the Langlands dual of $G^*$, ${G^*}^\vee \simeq G^\vee$.
\end{itemize}

The category $D(\mathsf{Rep}(G(F))$ of representations of $G(F)$ can be viewed as a full subcategory of $D(\operatorname{Bun}_{G^*,b})$, for which Hamann, Hansen and Scholze (\cite{hamann}) constructed the ``geometric Eisenstein operators'' $\Eis_{P!}$ that extend the parabolic induction functors $i_P^G$.
Moreover, it is conjectured (\cite{hansen_cllc}, Conjecture 5.4.1) that
under the conjectural categorical local Langlands correspondence,
the geometric Eisenstein operator $\Eis_{\bar{P}!}$ (where bar denotes the opposite parabolic subgroup) corresponds to the following spectral Eisenstein series functor
\[
\Eis_{P^\vee}=p_{P*} q^*_P \colon D\operatorname{Coh}(\Par_{M^\vee}) \to D\operatorname{Coh}(\Par_{G^\vee})
\]
where $p_P \colon \Par_{P^\vee} \to \Par_{G^\vee}$ and $q_P \colon \Par_{P^\vee} \to \Par_{M^\vee}$ are induced by the natural inclusion $P^\vee\to G^\vee$ and projection $P^\vee \to M^\vee$ respectively. Thus it is natural to hope that there are natural ``rational'' transformations
\[
J_{Q^\vee|P^\vee}^\circ, J_{Q^\vee|P^\vee} \colon \Eis_{P^\vee} \to \Eis_{Q^\vee}
\]
serving as the normalized and unnormalized ``spectral intertwining operator'' respectively and satisfying an analog of the Langlands conjecture.

\subsection{Main results}

Now we state the main theorems of the article. Our first surprising result is that unlike the classical case, the normalized intertwining operator is the one that appears the most naturally on the spectral side, as a canonical isomorphism between Eisenstein operators on an open substack of the moduli of parameters.

\begin{theorem}[Intertwining over the regular semisimple and generic locus]
\label{intertwining, rs,g}
    After restricting to the regular semisimple and generic loci 
    $\Par_{M^\vee}^{rs,g}$, 
    $\Par_{P^\vee}^{rs,g}$, $\Par_{Q^\vee}^{rs,g}$ and $\Par_{G^\vee}^{rs,g}$ 
    (which come from dense open subsets in corresponding coarse moduli spaces),
    we obtain operators
    \[
    \Eis_{P^\vee}^{rs,g}, \Eis_{Q^\vee}^{rs,g} \colon D\operatorname{Coh}(\Par_{M^\vee}^{rs,g}) \to D\operatorname{Coh}(\Par_{G^\vee}^{rs,g}).
    \]
    Then there is a canonical isomorphism, called the normalized intertwining operator,
    \[
    J^\circ_{Q^\vee|P^\vee} = J^{rs,g}_{Q^\vee|P^\vee} \colon 
    \Eis^{rs,g}_{P^\vee} \simto \Eis^{rs,g}_{Q^\vee}
    \]
    satisfying the transitivity for another parabolic $R$
    \[J^\circ_{R^\vee|Q^\vee}\circ J^\circ_{Q^\vee|P^\vee} = J^\circ_{R^\vee|P^\vee}.\]
    Moreover, $J^\circ_{Q^\vee|P^\vee}$ generates the free $\CZ(D\Coh(\Par_{M^\vee}^{rs,g}))$-module 
    $\Hom(\Eis^{rs,g}_{P^\vee}, \Eis^{rs,g}_{Q^\vee})$.
\end{theorem}

The idea for proving this theorem is that $\Eis_{P^\vee}^{rs,g}$ is canonically isomorphic to $\iota_{M^\vee,G^\vee*}$, where 
\[
\iota_{M^\vee,G^\vee} \colon \Par_{M^\vee}^{rs,g} \to \Par_{G^\vee}^{rs,g}
\]
is induced by the inclusion $M^\vee\to G^\vee$.

In this theorem, the regular semisimple locus and the generic locus are defined in Definition \ref{Regular semisimple locus} and \ref{Generic locus} in the unipotent case and Definition \ref{rs,g,intrinsic} in general. We remark that similar loci are defined in \cite{hamann2026geometriceisensteinseriesintertwining},
\cite{hansen_beijing} and \cite{zou2025categoricallocallanglandsmathrmgln}.
The isomorphism $J^\circ$ is defined in Definition \ref{rs, g} in the unipotent case and Definition \ref{rs, g, general} in general. The property about generating is proved in Proposition \ref{rs,g,gen} in the unipotent case and Proposition \ref{rs,g,general,gen} in general.

Our second result is that after multiplying by the inverse of some L-function at $1$, seen as a regular function on the moduli space of parameters $\Par_{M^\vee}$ (see Definition \ref{L-function,general}) and restricting to some blocks, the normalized intertwining operator can be uniquely extended to the whole space and thus provides an `integral intertwining operator''. Here, the inverse of the L-function at $1$ appears naturally, as the function of definition (up to a unit) of some components of the ``non-generic'' locus of $\Par_{M^\vee}$.

More precisely, let $\phi$ be an elliptic semisimple parameter in $\Par_{M^\vee}$ for a Levi subgroup $M^\vee \subset G^\vee$. 
Consider the Eisenstein operator on the connected component containing $\phi$:
\[
    \Eis_{P^\vee} \colon D\operatorname{Coh}(\Par_{M^\vee,\phi}) \to D\operatorname{Coh}(\Par_{G^\vee,\phi})
\]
Since $\Par_{M^\vee,\phi}$ is a $C_{M^\vee}(\phi)$-gerbe over its coarse moduli space (see Lemma \ref{gerbe}), we have a block decomposition of $D\operatorname{Coh}(\Par_{M^\vee,\phi})$ by isotypic components. Let $V$ be an irreducible representation of $C_{M^\vee}(\phi)$ that it is trivial on $C_{M^\vee}(\phi)^\circ$.
    
\begin{theorem}[Integral intertwining]
    \label{intertwining}
    After restricting to the $V$-isotypic block 
    $D\operatorname{Coh}(\Par_{M^\vee,\phi})_V$,
    \begin{itemize}
        \item The $\CO(\Par_{M^\vee, \phi})$-module
        $\Hom(\Eis_{P^\vee}, \Eis_{Q^\vee})$ of all transformations between Eisenstein operators is free and of rank $1$.
        \item There is a unique generator of this module, called the integral intertwining operator
        \[
        j_{Q^\vee|P^\vee} \colon \Eis_{P^\vee} \to \Eis_{Q^\vee}
        \]
        which restricts to $L(1, \ad_{\fn_{Q^\vee}/\fn_{P^\vee}})^{-1}$ times $J^\circ_{Q^\vee|P^\vee}$ on the regular semisimple and generic locus.
    \end{itemize}
\end{theorem}

This theorem is proved in Proposition \ref{all} in the unipotent case and Theorem \ref{integral, general} in general.

Our third theorem states that after dividing by the L-function at $0$, the integral intertwining operator will satisfy similar properties as the classical unnormalized intertwining operators, i.e., its adjoint induces a rational section of an analog of the Mackey's filtration. 

\begin{theorem}[Constant term and unnormalized intertwining]
\label{Adjunction}
    After dividing by the L-function $L(0, \ad_{\fn_{Q^\vee}/\fn_{P^\vee}})^{-1}$, the transformation $j_{Q^\vee|P^\vee}$ becomes a rational operator $J_{Q^\vee|P^\vee}$,
    called the unnormalized spectral intertwining operator,
    satisfying that the composition
    \[
    \mathrm{id} \to \CT_{Q^\vee} \circ \Eis_{P^\vee}/F^{>1} \to \mathrm{id}
    \]
    is an isomorphism, where 
    \begin{itemize}
        \item $\CT$ is the (partially defined) left adjoint of $\Eis$;
        \item $F^{>1}$ is defined by some filtration of $\CT\circ \Eis$;
        \item The $\CO(\Par_{M,\phi})$-module $\Hom(\mathrm{id}, \CT_{Q^\vee}\circ\Eis_{P^\vee}/F^{>1})$ is free and of rank $1$ 
        and the first transformation is a generator of it;
        \item The second transformation is induced by the adjunction of $J_{Q^\vee|P^\vee}$.
    \end{itemize}
\end{theorem}

\begin{remark}
    Recently, Hansen--Mann (\cite{hansen_cllc}, Definition 3.5.1 and Theorem 3.6.3) defined a functor $\CT$ and a filtration on $\CT\circ\Eis$. Their $\CT$ is the right adjoint of $\Eis$, which can be related to ours by Grothendieck--Serre duality (explained in Remark \ref{CT_B*}). Moreover, the subquotient indexed by $w \in W^\sigma$ of the filtration in our case is $w^*$, which is simpler than the subquotient in their case, as we are restricted to some special blocks.
\end{remark}

The inverse of the L-function at $0$ appears naturally here, as the function of definition (up to a unit) of some components of the ``non-regular semisimple'' locus of $\Par_{M^\vee,\phi}$.
In this theorem, the adjoint functor $\CT$ is defined in Proposition \ref{ct_circ_eis}, \ref{CT_chi} in the unipotent case and \ref{ct-general} in general; the filtration is defined in \ref{filtration} in the unipotent case and \ref{filtration-general} in general.
This theorem is proved in Theorem \ref{mul-sc} in the unipotent case and \ref{mul-general} in general.

By construction, 
the following analog of the Langlands conjecture holds:
the operator
\[
    \frac{L(1, \ad_{\fn_{Q^\vee}/\fn_{P^\vee}})}{L(0,\ad_{\fn_{Q^\vee}/\fn_{P^\vee}})} J_{Q^\vee|P^\vee}
\]
satisfies the transitivity (as it is identified with $J^\circ_{Q^\vee|P^\vee}$).
In particular, the composition $J_{P^\vee|Q^\vee}\circ J_{Q^\vee|P^\vee}$ is given by the multiplication of the function
\[
    \phi \mapsto \frac{L(0, \phi, \ad_{\fn_{P^\vee}/\fn_{Q^\vee}})L(0, \phi, \ad_{\fn_{Q^\vee}/\fn_{P^\vee}})}{L(1, \phi, \ad_{\fn_{P^\vee}/\fn_{Q^\vee}})L(1, \phi, \ad_{\fn_{Q^\vee}/\fn_{P^\vee}})}
\]
on $\Par_{M^\vee,\phi}$. We remark that these theorems are purely on the spectral side.
Finally, we prove that (Theorem \ref{compatibility}) under the categorical Langlands correspondence, the spectral intertwining operator $J_{Q^\vee|P^\vee}$ corresponds to the classical intertwining operator $J_{\bar{Q}|\bar{P}}$ of representations \textbf{up to a unit} in $\CO(\Par_{M^\vee,\phi})$, verifying the Langlands conjecture \ref{conjecture-Langlands} up to a unit for supercuspidal representations with elliptic L-parameter and for $G = G^*_b$ an extended pure inner form.

\begin{remark}
\label{convention}
    Though there is an opposition of parabolic subgroup $\Eis_{P!} \longleftrightarrow \Eis_{\bar{P}^\vee}$ under the conjectural categorical local Langlands correspondence, our result is still compatible with the Langlands conjecture. In the convention in \cite{hansen_cllc}, the categorical local Langlands correspondence is compatible with the local class field theory $F^\times \simeq W_F^{ab}$ which sends the uniformizer $\varpi$ of $F^\times$ to the \textit{geometric} Frobenius. 
    Hence we should use the geometric Frobenius in the expected formula for the automorphic L-function
    \[
    L(s, \pi, V) = L(s, \phi_\pi, V) = \det(1 - q^{-s}\phi_\pi(\Fr), V^I)^{-1}. 
    \]
    However, in our conventions in section 2 and section 3, the Frobenius in $W_F$ is taken to be the \textit{arithmetic} Frobenius, which is the inverse of the geometric Frobenius. 
    By definition, 
    \[
    L(s, \phi, \ad_{\fn_{P^\vee}}) = \det(1-q^{-s}\phi(\mathrm{Fr}), \ad_{\fn_{P^\vee}^I})^{-1} = \det(1-q^{-s}\phi(\mathrm{Fr}^{-1}), \ad_{\fn_{\bar{P}^\vee}^I})^{-1}.
    \]
    Thus the expected automorphic L-function $L(s, \pi, \ad_{\fn_{\bar{P}^\vee}})$ is $L(s, \phi_\pi, \ad_{\fn_{P^\vee}})$ in our convention, explaining the compatibility. 
\end{remark}

\begin{remark}
    Note that our construction of spectral intertwining operators $j_{Q|P}$ and $J_{Q|P}$ depends on the arbitrary choices of a generator of the corresponding $\Hom$-modules, which are unique up to a unit in $\CO(\Par_{M^\vee,\phi})$. Thus our result about the composition of intertwining operators differs from the Langlands conjecture by some epsilon factors, which are units. 
    
    We do not have a good explanation of the unit here. It may rely on finding a corresponding object on the spectral side of the integration which appears in the definition of the inclusion $\mathrm{id}\to r_Q^G\circ i_P^G/F^{>1}$ in the Mackey's formula.
\end{remark}

This article is divided into three sections. The first section deals with the unipotent component of unramified reductive groups, which can be treated by the ``twisted Springer theory'' developed by Xiao--Zhu (\cite{xiao2018}). The second section deals with the general case, which can be reduced to the case of principal blocks using the structure results about $\Par_G$ stated in \cite{dat2025}. In the last section we prove the compatibility of the spectral intertwining operators and the classical intertwining operators under the conjectural categorical local Langlands correspondence.

\subsection*{Acknowledgment}
The author expresses his deepest gratitude to his Ph.D. supervisor Jean-François Dat for introducing the subject of intertwining operators and local Langlands correspondence and very helpful discussions and corrections about this article. The author also sincerely thanks Alexis Bouthier, David Hansen, Jean-Loup Waldspurger and Konrad Zou for useful comments.
The author is funded by Institut de Mathématiques de Jussieu-Paris Rive Gauche, Sorbonne Université and École Normale Supérieure.

%% file: 2.tex
\section{The unipotent component}

In this section, let $G$ be a reductive group over some algebraically closed field $k$ of characteristic $0$, $B$ be a Borel subgroup of $G$ and $T$ be a torus of $B$. Let $\sigma$ be an automorphism of $G$ fixing $B$, $T$ and a pinning of $G$. We note that these data may come from the Langlands dual of some quasisplit reductive group over a $p$-adic local field $F$.

From the pair $(G,\sigma)$, we construct the semidirect product $G\rtimes \langle\sigma \rangle$ of $G$ and the group generated by $\sigma$. We will consider $G\sigma := \{(g\in G,\sigma)\}$, a subscheme of the product. Then for $h\in G$, viewed as $(h,1)\in G\rtimes \langle \sigma \rangle$, the adjoint action of $h$ on $G\sigma$ is given by $h(g\sigma)h^{-1}=hg\sigma(h)^{-1}\sigma$.

We define the scheme
\[
\tilde{\BL}_{G\sigma} := \{(g\sigma, n)\in G\sigma\times \CN \mid \ad(g)(\sigma(n)) = qn \}
\]
and its quotient stack
\[
\BL_{G\sigma} := \{(g\sigma, n)\in G\sigma\times \CN \mid \ad(g)(\sigma(n)) = qn \}/ G
\]
where $q=p^n\in k$ for some prime $p$ and $n\in \BN_+$ and $\CN$ is the nilpotent cone of $G$. The $G$-action here is given by $h(g, n) := (hg\sigma(h)^{-1}, \ad(h)(n))$.
Similarly we define $\BL_{B\sigma}$ and $\BL_{T\sigma}$ and note that $\BL_{T\sigma} \simeq T\sigma/ T$. Note that $\BL_{G\sigma}$ is isomorphic to the so-called ``unipotent'' component $\Par_{G,1}$ containing the unipotent representations in the moduli space of Weil--Deligne L-parameters $\Par_G^{\text{WD}}$ (which is isomorphic to $\Par_G:= Z^1(W_F, G)/G$ by \cite{fargues2024}, Proposition VIII.2.5) containing those L-parameters whose semisimplifications are unramified if we take $q$ be the cardinal of the residue field of the local field $F$. 
Here the $W_F$-action on $G$ satisfies that the inertia subgroup acts trivially on $G$ and the action of the \textit{arithmetic} Frobenius is given by $\sigma$. The identification $\Par_{G,1}\simto \BL_{G\sigma}$ is given by
\[
(\phi, N) \mapsto (\phi(\Fr)\sigma, N),
\]
where $\Fr$ is the arithmetic Frobenius.
We remark that the object $\BL_{G\sigma}$ is considered in \cite{Hellmann_2023} (where $q$ is replaced by $q^{-1}$) and \cite{ben2024coherent} when $\sigma = 1$. 

Next we define the Eisenstein series.

\begin{definition}[Spectral Eisenstein series]
    For any Borel subgroup $B\subset G$ containing $T$ and fixed by $\sigma$, define
    $p_B \colon \BL_{B\sigma} \to \BL_{G\sigma}$ induced by the inclusion $B\to G$ and 
    $q_B \colon \BL_{B\sigma} \to \BL_{T\sigma}$ induced by the projection $B\to T$. The spectral Eisenstein operator $\Eis_B$ is defined as 
    \[
        \Eis_B := p_{B*} \circ q_B^* \colon D\Coh(\BL_{T\sigma}) \to D\Coh(\BL_{G\sigma}).
    \]
    Note that the morphism $p_B$ is proper by \cite{zhu2025coherentsheavesstacklanglands}, Lemma 3.21. Thus $\Eis_B$ will send $D\Coh$ to $D\Coh$.
\end{definition}

We will also need a variant of the spectral Eisenstein series by ignoring the monodromy part.

\begin{definition}[Variant]
\label{small eis}
    For any Borel subgroup $B\subset G$ containing $T$ and fixed by $\sigma$, abusing the notations, define
    $p_{B,0} \colon B\sigma/ B \to G\sigma/ G$ (denoted by $p_0$ is there is no ambiguity) induced by the inclusion $B \to G$ and 
    $q_{B,0} \colon B\sigma/ B \to T\sigma/ T$ (denoted by $q_0$ is there is no ambiguity) induced by the projection $B\to T$. The variant spectral Eisenstein operator $\eis_B$ is defined as 
    \[
        \eis_B := p_{0*} \circ q_0^* \colon D\Coh(T\sigma/ T) \to D\Coh(G\sigma/ G).
    \]
\end{definition}

Let $B'$ be another Borel subgroup of $G$ fixed by $\sigma$ and containing $T$. 
We will construct the spectral intertwining operator 
\[
J_{B'|B} \colon \Eis_B \to \Eis_{B'}
\]
in the following steps:

\begin{itemize}
    \item Define two open subsets of $\BL_{T\sigma}$ with respect to $G$, called the regular semisimple locus and the generic locus. Then construct the intertwining operator after restricting to the intersection of these loci. The operator is the normalized intertwining operator $J^\circ_{Q|P}$.
    \item Extend the operator to the generic locus.
    \item After multiplication by some function on $\BL_{T\sigma}$, extend the operator to all $\Par_T$ and one obtains an operator $j_{Q|P}$ defined on all $\Par_T$.
    \item Calculate the adjunction $\CT\circ \Eis \to \mathrm{id}$, construct an operator 
    $\mathrm{id}\to \CT\circ \Eis/F$ for some subfunctor $F$ and calculate their composition to get a function in $\CO(\BL_{T\sigma})$. Divide by this function and one obtains the unnormalized intertwining operator $J_{Q|P}$.
\end{itemize}

\subsection{Some results about the pair \texorpdfstring{$(G, \sigma)$}{G,sigma}}

We recall some basic results about reductive groups with an automorphism preserving a fixed pinning of it. The proofs of those results can be found in the article of Xiao--Zhu (\cite{xiao2018}).

\begin{proposition}[The root system for $\sigma$, \cite{xiao2018} Lemma 5.1.1]
    $\sigma$ permutes the roots of $G$. Moreover, there exists a root system for the torus $T_\sigma$ and weight lattice $X^*(T_\sigma)\simeq X^*(T)^\sigma$ where the roots are of the form
    \[
    \alpha_O := \sum_{\beta \in O} \beta \in X^*(T)^\sigma
    \]
    where $O$ is an orbit of the $\sigma$-action on roots of $G$. 
    Moreover, the Weyl group of this root system is $W^\sigma$. The Weyl group of the root system can also be identified as
    \[
    N_G(T\sigma)/T = \{g\in G\mid g(T\sigma)g^{-1}\subset T\sigma\}/T
    = \{g\in N_G(T)\mid \sigma(g)g^{-1}\in T\}/T.
    \]
\end{proposition}

From the construction, the roots in the root system can be viewed as weights of the torus $T_\sigma$.

\begin{definition}[Type of an orbit, \cite{xiao2018} Definition 5.1.4]
\label{orbit}
    \begin{itemize}
        \item If $\alpha_O$ is not equal to $\alpha_{O'}$ for other orbits $O'$,
        $O$ is called of type A.
        \item If there exists an another $O'$ such that $\alpha_{O}=\alpha_{O'}$,
        $O'$ is unique. Moreover, $|O|=2|O'|$ or $|O'|=2|O|$. We call the orbit consisting less elements of type BC+ and the other of type BC- and we say $\alpha_O$ is of type BC.
    \end{itemize}
\end{definition}

Denote $\fn$ the Lie algebra of the unipotent radical of $B$. Then there is an adjoint action of $T$ on $\fn$, denoted $\ad_\fn \colon T \to \End(\fn)$. 
We consider the function in $\CO(T)$
\[
L(s, \ad_\fn) := t \mapsto \det(1 - q^{-s}\ad_\fn(t)\circ \sigma)^{-1}
\]
for $s = 0, 1$. The function is invariant under the twisted $T$-action, thus induces a function in $\CO(\BL_{T\sigma})$. We use the notation L because the function is the same as the L-function of Galois representations on the principal block. The notations in the L-function can be generalized as follows: $T$ can be replaced by any subgroup of $G$ and $\fn$ can be replaced by any Lie subalgebra of $\fg$ fixed by $\sigma$ and the subgroup.
The next proposition expresses the $L$-function
$L(s, \ad_\fn)$ in terms of functions on the torus $T_\sigma$.

\begin{proposition}[$L$-function]
    \label{L-function}
    The $L$-function can be expressed as
    \[
    L(s, \ad_\fn)^{-1} = \prod_{O} L(s, \ad_{\fn_O})^{-1}
    \]
    where $O$ is taken among the orbits of roots under the $\sigma$-action
    and $\fn_O$ is the direct sum of $\fn_\alpha$'s for $\alpha\in O$.
    Moreover,
    \[
    L(s, \ad_{\fn}|_O)^{-1}:= L(s, \ad_{\fn_O})^{-1} = 
    \begin{cases}
    1 - q^{-s|O|}e^{\alpha_O}(t) & \text{$O$ is of type A or BC-} \\
    1 + q^{-s|O|}e^{\alpha_O}(t) & \text{$O$ is of type BC+}
    \end{cases}
    \]
    (where $e^\alpha$ is exactly $\alpha$, emphasizing the fact that it is a weight of $T_\sigma$).
\end{proposition}

\begin{proof}
     The case $s=0$ is Lemma 5.2.9 in \cite{xiao2018} and the proof works the same for $s=1$.
\end{proof}

The zeros of the $L$-function corresponds to the singular locus of the $W^\sigma$-action on $T_\sigma$.

\begin{proposition}[\cite{xiao2018}, Lemma 5.1.12]
    \label{reflection}
    Suppose that $G$ is simply connected.
    For a root $\alpha$ of $T_\sigma$, the fixed point subgroup of $T_\sigma$ under the reflection $s_\alpha$ is the zero locus of the function
    \[
    \begin{cases}
    1 - e^{\alpha}(t) & \text{$\alpha$ is of type A} \\
    (1 + e^{\alpha}(t))(1 - e^{\alpha}(t)) & \text{$\alpha$ is of type BC}
    \end{cases}
    \]
    and the multiplicity is $1$.
\end{proposition}

We have the following analog of the Chevalley isomorphism $G\sslash G\simeq T\sslash W$.

\begin{proposition}[Twisted Chevalley isomorphism, Proposition 4.2.3 in \cite{xiao2018}]
\label{Chevalley}
    The natural inclusion $T\sigma \subset G\sigma$ induces
    an isomorphism $\chi \colon G\sigma\sslash G \simeq T\sigma\sslash N_G(T\sigma)$, called the twisted Chevalley isomorphism.
\end{proposition}

Moreover, we will call $G\sigma \to T\sigma\sslash N_G(T\sigma)$ the twisted Chevalley map, denoted also by $\chi$.

\subsection{Intertwining over the regular semisimple and generic locus}
We firstly define the regular semisimple locus and the generic locus with respect to the $\sigma$-action.

\begin{definition}[Generic locus]
\label{Generic locus}
    We define the generic locus of $\BL_{G\sigma}$, $\BL_{B\sigma}$ and $\BL_{T\sigma}$ as follows.
    \begin{itemize}
        \item For $T\sigma$, the generic locus of $T\sigma$, $(T\sigma)^g$ is the non-vanishing locus of $L(1, \ad_\fn)^{-1}\cdot L(1, \ad_{\bar{\fn}})^{-1}$, where $\bar{n}$ is the opposite nilpotent subalgebra of $\fn$. Then we define the generic locus $\BL_{T\sigma}^g$ of $\BL_{T\sigma}$ as $(T\sigma)^g/T$. Note that $(T\sigma)^g$ is stable under the action of the Weyl group $W^\sigma$ and $T$-conjugation.
        \item For $B\sigma$, we pull back the function $L(1, \ad_\fn)^{-1}\cdot L(1, \ad_{\bar{\fn}})^{-1}$ from $T$ to $B$. The generic locus $(B\sigma)^g$ of $B\sigma$ is the non-vanishing locus of $L(1, \ad_\fn)^{-1}\cdot L(1, \ad_{\bar{\fn}})^{-1}$.
        From its construction, $(B\sigma)^g$ is stable under $B$-conjugation.
        Then the generic locus $\BL_{B\sigma}^{g}$ of $\BL_{B\sigma}$ is the (not necessarily dense) open substack containing those points $(b\sigma,n)$ with $b\sigma \in (B\sigma)^g$. 
        \item For $G\sigma$, consider the twisted chevalley map
        $\chi \colon G\sigma \to (T\sigma)\sslash N_G(T\sigma)$. Then
        $(G\sigma)^g$ is defined as $\chi^{-1}((T\sigma)^g)$. From its construction, $(G\sigma)^g$ is stable under $G$-conjugation.
        Then the generic locus $\BL_{G\sigma}^{g}$ of $\BL_{G\sigma}$ is the (not necessarily dense) open substack containing those points $(g\sigma,n)$ with $g\sigma \in (G\sigma)^g$. 
    \end{itemize}
\end{definition}

We remark that the generic locus is not the open substack of generic parameters in the sense of Hamann (\cite{hamann2026geometriceisensteinseriesintertwining} defined by vanishing of Galois cohomology. The generic parameters in the sense of Hamann form the regular semisimple and generic locus in our sense, as seen in the proof of \ref{complex}.

We note that over the generic locus, the nilpotent part vanishes.

\begin{lemma}
    \label{generic locus}
    The projection $\BL_{G\sigma}^{g} \to (G\sigma)^g/ G$ is an isomorphism, similarly for $B$ and $T$.
\end{lemma}

\begin{proof}
    For the argument about $G$, by the proof of \cite{xiao2018}, Proposition 5.3.1, $B\sigma/B\to G\sigma/G$ is surjective. Unwinding the definition, it suffices to prove that: for $b \in (B\sigma)^g$, $n \in \CN$ satisfying
    $\ad(b)(n) = qn$, then $n=0$. 
    
    We have
    \[
    L(1, \ad_\fg(b))^{-1} = \det(1-q^{-1}\ad_\fg(b)\circ \sigma)
    = \det(1-q^{-1}\ad_\fg(\bar{b})\circ \sigma)
    \]
    where $\bar{b}$ is the image of $B$ in $T$ via the projection $B\to T$. Thus
    \[
    L(1, \ad_\fg(b))^{-1} = 
    L(1, \ad_\fn(\bar{b}))^{-1}L(1, \ad_\ft(\bar{b}))^{-1}L(1, \ad_{\bar{\fn}}(\bar{b}))^{-1}.
    \]
    By definition of the generic locus, the factor $L(1, \ad_\fn(\bar{b}))^{-1}L(1, \ad_{\bar{\fn}}(\bar{b}))^{-1}$ is invertible. As $\sigma$ is of finite index, the eigenvalues of the $\sigma$-action on $\ft$ are roots of unity. Thus the factor $L(1, \ad_\ft(\bar{b}))^{-1}$ is also invertible. Hence $\det(1-q^{-1}\ad_\fg(b)\circ \sigma)$ is invertible and we conclude that $n = 0$.
\end{proof}

The following definition of the regular semisimple locus is a generalization of the usual definition of ``regular semisimple'' in the $\sigma=1$ case.

\begin{definition}[Regular semisimple locus]
\label{Regular semisimple locus}
    We define the regular semisimple locus of $\BL_{G\sigma}$, $\BL_{B\sigma}$ and $\BL_{T\sigma}$ as follows.
    \begin{itemize}
        \item For $T\sigma$, the regular semisimple locus $(T\sigma)^{rs}$ is the non-vanishing locus of $L(0, \ad_\fn)^{-1}$ and $\BL_{T\sigma}^{rs}$.
        \item For $B\sigma$, we pull back the function $L(0, \ad_\fn)^{-1}$. Then the regular semisimple locus $(B\sigma)^{rs}$ is the non-vanishing locus of $L(0, \ad_\fn)^{-1}$. Then $\BL_{B\sigma}^{rs}$ is defined to be the 
        open substack $\{(b\sigma,n)\mid b\sigma\in (B\sigma)^{rs}\}$
        of $\BL_{B\sigma}$.
        \item For $G\sigma$, $(G\sigma)^{rs}$ the inverse image of $(T\sigma)^{rs}$ under the twisted Chevalley map.
        Then $\BL_{G\sigma}^{rs}$ is defined to be the 
        open substack $\{(g\sigma,n)\mid g\sigma\in (G\sigma)^{rs}\}$
        of $\BL_{G\sigma}$.
    \end{itemize}
\end{definition}

\begin{remark}
    \label{rs=r+s}
    By Lemma 5.2.14 of \cite{xiao2018},
    the locus $(G\sigma)^{rs}$ defined here is the same as that defined in Xiao--Zhu, where $(G\sigma)^{rs}$ is defined to be the image of $G\times T\sigma\to G\sigma$, where $G$ acts by conjugation.
    Moreover, by Lemma 5.2.11 in \cite{xiao2018}, the notion ``regular semisimple'' is the same as ``regular and semisimple'' in $G\sigma$.
    Here the notion ``regular'' is defined in Definition \ref{regular} 
    and an element in $G\sigma$ is called semisimple if it is conjugate to an element in $T\sigma$, or equivalently saying, its corresponding L-parameter is semisimple.
\end{remark}

By definition, $\BL_{B\sigma}^?$ is the inverse image of $\BL_{T\sigma}^?$ under $q$, where ? is $rs$ or $g$. We will see later that similar result holds for $p\colon \BL_{B\sigma}\to \BL_{G\sigma}$.

There is a slightly different notion about the regular semisimplicity.

\begin{definition}[Strongly regular semisimple]
    For $x\in (B\sigma)^{rs}/B\simeq (T\sigma)^{rs}/T$ (the isomorphism is proved in \ref{rs locus}), $x$ is called strongly regular semisimple if the subgroup of $W^\sigma$ that fixes $x$ is trivial.
\end{definition}

The strongly regular semisimple $x$'s form an open subset of $T\sigma$, as $W^\sigma$ is the Weyl group of $T_\sigma$.

To justify the definition of the generic locus and the regular semisimple locus, we may consider the cotangent complex.

\begin{proposition}[The cotangent complex]
    \label{complex}
    Denote by $i \colon \BL_{T\sigma} \to \BL_{G\sigma}$ the map induced by the natural inclusion. Then the cotangent complex $Li$ satisfies the following:
    \begin{itemize}
        \item The generic locus $\BL_{T\sigma}^g$ is the open substack of $\BL_{T\sigma}$ where $L^1 i = 0$.
        \item The regular semisimple locus $\BL_{T\sigma}^g$ is the open substack of $\BL_{T\sigma}$ where $L^{-1} i = 0$.
        \item In particular, over $\BL_{T\sigma}^{rs, g}$, the cotangent complex vanishes and then $\BL_{T\sigma}^{rs, g} \to \BL_{G\sigma}^{rs, g}$ is \'etale.
    \end{itemize}
\end{proposition}

Hence the regular semisimple and generic locus is the same as the open substack of parameters of ``Langlands--Shahidi type'' in the unipotent component.

\begin{proof}
    As mentioned above, $\BL_{G\sigma}$ is a connected component of $Z^1(W_F, G)$. Thus by \cite{fargues2024}, Proposition VIII.2.1, the cotangent complex of $\BL_{G\sigma}$ at a parameter $\phi\in Z^1(W_F, G)$ is
    \[
    \phi^*L_{\BL_{G\sigma}} \simeq R\Gamma(W_F, \fg_\phi)[1]
    \]
    where the subscript $\phi$ means twisting by $\phi$. Thus the cotangent complex $Li$ at $\phi$ is
    \[
    \phi^*Li = \operatorname{cofib}(\phi^*L_{\BL_{T\sigma}} \to \phi^*L_{\BL_{G\sigma}}) = R\Gamma(W_F, (\fg/\ft)_\phi)[1]
    = R\Gamma(W_F, (\fn\oplus\bar{\fn})_\phi).
    \]
    We have $H^0(W_F, (\fn\oplus\bar{\fn})_\phi)$ is the fixed points of $\fn\oplus \bar{\fn}$ under $\ad_{\fn\oplus\bar{\fn}}(t)\circ \sigma$ if the parameter $\phi$ corresponds to $(t, 0)$ in $\BL_{T\sigma}$. Thus the subset where $L^{-1}i=0$ is the non-vanishing set of $L(0, \ad_{\fn}\circ \sigma)^{-1}L(0, \ad_{\fn}\circ \sigma)^{-1}$, which is the same as the regular semisimple locus (i.e., the non-vanishing set of $L(0, \ad_{\fn}\circ \sigma)^{-1}$) by the concrete form of $L$-functions in \ref{L-function}.

    By \cite{dat2025}, Lemma 5.1, 
    \[
    H^2(W_F, (\fn\oplus\bar{\fn})_\phi)^* \simeq H^0(W_F, ((\fn\oplus\bar{\fn})_\phi)^*(1))
    \]
    where $(1)$ means the cyclotomic twist. Thus it is the fixed points of $(\fn\oplus\bar{\fn})^*$ under $q\ad_{(\fn\oplus\bar{\fn})^*}(t)\circ \sigma$. Thus the subset where $L^1i=0$ is the non-vanishing set of $L(1, \ad_{\fn}\circ \sigma)^{-1}L(1, \ad_{\fn}\circ \sigma)^{-1}$, which is exactly the generic locus.

    Again by \cite{dat2025}, Lemma 5.1, for a finitely dimensional $W_F$-representation $V$ satisfying $H^0(W_F,V)=H^2(W_F,V)=0$, we have $R\Gamma(W_F, V)=0$. Thus $Li=0$ over the regular semisimple and generic locus and the morphism $\BL_{T\sigma}^{rs,g}\to \BL^{rs_g}_{G\sigma}$ is \'etale.
\end{proof}

Next we study the morphism $\BL_{B\sigma} \to \BL_{T\sigma}$ and $\BL_{B\sigma} \to \BL_{G\sigma}$.

\begin{lemma}
    \label{rs locus}
    The pair of morphisms $\iota_B \colon (T\sigma)^{rs}/T \to (B\sigma)^{rs}/ B$ and $q_0 \colon (B\sigma)^{rs}/ B \to (T\sigma)^{rs}/ T$ (the morphisms are well-defined by definition of the regular semisimple locus) induced by the natural inclusion and projection provide an isomorphism $(B\sigma)^{rs}/ B \simeq (T\sigma)^{rs}/ T$ and thus 
    $\BL_{B\sigma}^{rs, g} \simeq \BL_{T\sigma}^{rs, g}$.
\end{lemma}

\begin{proof}
    We have $q_0\circ \iota_B = \mathrm{id}$. By deformation theory in \ref{complex}, $\iota_B$ is \'etale. Thus $\iota_B$ is open and closed.
    Moreover, $B$ is connected. Thus $\iota_B$ is an isomorphism.
    The argument for $\BL_{B\sigma}$ follows from \ref{generic locus}.
\end{proof}

\begin{lemma}
    The pair of morphisms $i \colon T\sigma\sslash T \to B\sigma\sslash B$ and $q \colon B\sigma\sslash B \to T\sigma\sslash T$ induced by the natural inclusion and projection provide an isomorphism $B\sigma\sslash B \simeq T\sigma\sslash T$.
\end{lemma}

\begin{proof}
    By Lemma \ref{rs locus} above, $(B\sigma)^{rs}\sslash B \simeq (T\sigma)^{rs}\sslash T$. We have $q^*\colon \CO(T\sigma/T)\to \CO(B\sigma/B)$ is injective and splits as an $\CO(T\sigma/T)$-module homomorphism. As $\CO((T\sigma)^{rs}\sslash T)=\CO(T\sigma/T)[L(0,\ad_\fn)]$ and similarly for $B$, $q^*$ becomes an isomorphism after inverting $L(0,\ad_\fn)^{-1}$. Thus $q^*$ is an isomorphism as $\CO(T\sigma/T)$ and $\CO(B\sigma/B)$ are integral domains.
\end{proof}

Thus we have the following factorization of the twisted Chevalley isomorphism.
\begin{corollary}
    \label{factor Chevalley}
    The morphism $T\sigma\sslash T\to G\sigma\sslash G$ inducing the twisted Chevalley isomorphism (Proposition \ref{Chevalley}) factors as
    \[
    T\sigma\sslash T \xleftarrow[\sim]{q} B\sigma\sslash B\to G\sigma\sslash G.
    \]
\end{corollary}

From the corollary above, $\BL_{B\sigma}^?$ is the inverse image of $\BL_{G\sigma}^?$ under $q$, where ? is $rs$ or $g$.
Hence the morphisms $\BL_{T\sigma}\leftarrow \BL_{B\sigma}\to \BL_{G\sigma}$ restrict to their generic locus $\BL_{T\sigma}^g\leftarrow \BL_{B\sigma}^g\to \BL_{G\sigma}^g$ and similarly for the regular semisimple locus. We will denote by $\Eis^{rs}$, $\Eis^g$, $\Eis^{rs,g}$ the corresponding functors, similarly for $\eis$. Note that $\Eis^?\circ j_T^{?*}$ is isomorphic to $j_G^{?*}\Eis$, where ? is rs or g or rs,g and $j_T^{?*}$, $j_G^{?*}$ is the open immersion of the ?-locus of $\BL_{T\sigma}$ and $\BL_{G\sigma}$ respectively. In the following notations, for $\CF\in D\Coh(\BL_{T\sigma})$, abusing the notations, we will abbreviate $\Eis^? j^{?*}_T \CF$ simply by $\Eis^? \CF$.

\begin{lemma}[\cite{xiao2018}, Lemma 5.2.14]
\label{iota-G}
    The morphism $\iota_G\colon (T\sigma)^{rs}/ T\to (G\sigma)^{rs}/ G$ induced by the natural inclusion $T\to G$ induces an isomorphism
    $(T\sigma)^{rs}/ N_{G}(T\sigma)\simeq (G\sigma)^{rs}/ G$. Thus the morphism
    $(B\sigma)^{rs}/ B \to (G\sigma)^{rs}/ G$ is a Galois covering and the same holds for
    $\BL_{B\sigma}^{rs, g} \to \BL_{G\sigma}^{rs, g}$.
\end{lemma}

Now we define the intertwining operator over the regular semisimple and generic locus.

\begin{definition}[Intertwining over the regular semisimple and generic locus]
    \label{rs, g}
    Consider the morphism $\iota_G \colon \BL_{T\sigma}^{rs,g} \to \BL_{G\sigma}^{rs,g}$ induced by $T\to G$ and similarly for $\iota_B$. We define an isomorphism
    \[
    t_B\colon \Eis_B^{rs, g} \simeq \iota_{G*}.
    \]
    as
    \[
    \Eis_B^{rs, g} = p_{B*}q_B^* \simeq p_{B^*}\iota_{B*}= \iota_{G*},
    \]
    where the second isomorphism is deduced from the fact that $\iota_B$ and $q_B$ are mutually inverse and the third is from $\iota_G=p_B\circ \iota_B$.

    For another Borel $B'$ of $G$ containing $T$ and stable under $\sigma$, we define the canonical isomorphism 
    \[
    J_{B'|B}^{rs, g} \colon \Eis_B^{rs, g} \simeq \Eis_{B'}^{rs, g}
    \] 
    as the composition of isomorphisms $t_{B'}^{-1}t_B$, called the intertwining operator over the regular semisimple and generic locus.

    In the same way, we obtain an isomorphism $J_{B'|B,0}^{rs}\colon \eis_B^{rs} \simeq \eis_{B'}^{rs}$.
\end{definition}

We will sometimes denote $J^{rs, g}$ by $J^\circ$.
The operators $J^\circ$ satisfy the transitivity by definition.

\begin{proposition}
    \label{trans}
    For any three Borel subgroups $B_1, B_2, B_3$ of $G$ fixed by $\sigma$ and containing $T$, 
    $J_{B_3|B_2}^\circ \circ J_{B_2|B_1}^\circ = J_{B_3|B_1}^\circ$.
\end{proposition}

Next we show that the isomorphism of functors $\eis_B^{rs}\simeq \eis_{B'}^{rs}$ above generates all transformations $\eis_B^{rs}\to \eis_{B'}^{rs}$ in some sense. 

We introduce the following formalism. Let $F\colon \CC\to \CD$ be a functor between additive categories. Then there is a natural ring homomorphism $\CZ(\CC):=\End(\mathrm{id}_\CC) \to \End(F)$. Moreover, let $G$ be an another functor from $\CC$ to $\CD$. Then $\Hom(F, G)$ becomes a right $\End(F)$-module, and restricts to a right $\CZ(\CC)$-module.
We remark that we may also consider the $\End(G)$-action and deduce a $\CZ(\CC)$-module structure. These two module structure coincide.

Applying the above formalism to 
$\CC = D\Coh(\BL_{T\sigma})$, $\CD = D\Coh(\BL_{G\sigma})$, $F = \Eis_B$, $G=\Eis_{B'}$, we obtain a ring homomorphism $\CZ(D\Coh(\BL_{T\sigma}))\to \End(\Eis_B)$ and $\Hom(\Eis_B, \Eis_{B'})$ is an $\CZ(D\Coh(\BL_{T\sigma}))$-module. All those constructions may apply over the regular semisimple or generic loci. 

Now we state the proposition about the generator.

\begin{proposition}
\label{rs,g,gen}
    The ring homomorphism $\CZ(D\Coh(\BL_{T\sigma}^{rs,g}))\to \End(\Eis^{rs,g})$ is an isomorphism.
    Moreover, the isomorphism $J_{B'|B}^{rs,g}$ in Definition \ref{rs, g} generates the free $\CZ(D\Coh(\BL_{T\sigma}^{rs,g}))$-module
    $\Hom(\Eis_B^{rs, g}, \Eis_{B'}^{rs, g})$. Similarly for $\eis^{rs}$.
\end{proposition}

In fact, we can prove the following general statement.  

\begin{proposition}
\label{gen-quotient}
    Let $f \colon X \to Y$ be a finite \'etale morphism between connected smooth algebraic stacks over $k$.
    Suppose that there exists a closed point $x\in X$ 
    such that the isotropy groups of $x$ and $y:=f(x)$ are isomorphic.
    Then applying the formalism above to $\CC=D\Coh(X)$, $\CD = D\Coh(Y)$ and $F = f_*$, the homomorphism
    \[
    \CZ(D\Coh(X)) \to \End(f_*)
    \]
    is an isomorphism.
\end{proposition}

\begin{proof}
    As $f$ is \'etale, $f_*$ is faithful and the ring homomorphism $\CZ(D\Coh(X))\to \End(f_*)$ is injective. Then we prove that it is surjective. An element in $\End(f_*)$ is given by a compatible family of homomorphisms 
    \[
    \{a_\CF\}_{\CF\in \Coh(X)},\quad a_\CF \in \End_Y(f_*\CF).
    \]
    As $X$ is smooth and the characteristic of $k$ is $0$, locally free sheaves generate the entire category $D\Coh(X)$. It suffices to prove that for each locally free sheaf $\CF \in D\Coh(X)$, $a_\CF \in f_*(\End_X(\CF))$.

    In fact, consider the completion $\CF_x^\wedge$ (though it is not a coherent sheaf, we view it as an inverse system of coherent sheaves). It suffices to prove the following claim: 
    \[
    \End_Y(f_*\CF)\ \text{``}\cap\text{''}\ \End_Y(f_*\CF_x^\wedge) = f_*(\End_X(\CF)),
    \]
    i.e., if we have compatible homomorphisms $a'_\CF\in \End_Y(f_*\CF)$
    and $a'_{\CF_x^\wedge} \in \End_Y(f_*\CF_x^\wedge)$, then $a'_\CF\in f_*(\End_X(\CF))$.

    By definition of algebraic stacks, 
    there exists a surjective smooth morphism $\tilde{Y} \to Y$ such that $\tilde{Y}$ and $\tilde{Y}\times_Y \tilde{Y}$ are schemes. 
    Replacing $\tilde{Y}$ by some Galois covering of it, we may assume that $f_*\colon \tilde{X}:=X\times_Y \tilde{Y}\to \tilde{Y}$ is isomorphic to a disjoint union of some $\tilde{Y}$'s.   
    Over the smooth site over $Y$, the construction $S \mapsto \End_S(\CG_S)$ is sheafy for any $\CG\in \Coh(Y)$. Here we use the subscript to denote the pullback sheaf on it. It suffices to prove that
    \[
    \End_{\tilde{Y}}(\tilde{f}_*\CF_{\tilde{X}}) \ \text{``}\cap\text{''}\ \End_{\tilde{Y}}(\tilde{f}_*\CF_{\tilde{X},\tilde{x}}^\wedge) = \tilde{f}_*(\End_{\tilde{X}}(\CF_{\tilde{X}})),
    \]
    where $\tilde{x}$ the fiber of $x$ in $\tilde{X}$.
    By the condition about $x$, denoting by $\tilde{y}$ the fiber of $y=f(x)$ in $\tilde{Y}$, $\tilde{f}$ induces an isomorphism $\tilde{x}\to \tilde{y}$. Write $\tilde{X} = \coprod \tilde{Y}_i$. We have
    \[
    \End_{\tilde{Y}}(\tilde{f}_*\CF_{\tilde{X}}) =
    \End_{\tilde{Y}}(\bigoplus_i \CF_{\tilde{Y}_i})
    \]
    and 
    \[
    \End_{\tilde{Y}}(\tilde{f}_*\CF_{\tilde{X},\tilde{x}}^\wedge)
    = \End_{\tilde{Y}}(\bigoplus_i \CF_{\tilde{Y_i},\tilde{x}\cap \tilde{Y}_i}^\wedge) = \bigoplus_i\End_{\tilde{Y}}(\CF_{\tilde{Y_i},\tilde{x}\cap \tilde{Y}_i}^\wedge).
    \]
    Here the last identity is from the following reasons:
    As $\tilde{f}$ induces an isomorphism $\tilde{x}\to \tilde{y}$, $\tilde{f}(\tilde{x}\cap \tilde{Y}_i)$ and $\tilde{f}(\tilde{x}\cap \tilde{Y}_j)$ are disjoint for different $i, j$. Thus the entries in the expression of $\End_{\tilde{Y}}(\tilde{f}_*\CF_{\tilde{X},\tilde{x}}^\wedge)$ away  from the diagonal must vanish. 
    
    Moreover, as $X$ is connected and smooth, the homomorphism $\CF \to \CF_x^\wedge$ is injective. As $\tilde{Y}_i \to X$ is faithfully flat, for any $i$, its pullback on $\tilde{Y}_i$
    \[
    \CF_{\tilde{Y}_i} \to \CF_{\tilde{Y}_i, \tilde{x}\cap \tilde{Y}_i}^\wedge
    \]
    is injective. Thus we have
    \[
    \End_{\tilde{Y}}(\tilde{f}_*\CF_{\tilde{X}}) \ \text{``}\cap\text{''}\ \End_{\tilde{Y}}(\tilde{f}_*\CF_{\tilde{X},\tilde{x}}^\wedge) = \bigoplus_i \End_{\tilde{Y}}(\CF_{\tilde{Y}_i})=\tilde{f}_*(\End_{\tilde{X}}(\CF_{\tilde{X}})).
    \qedhere
    \]
\end{proof}

Now we turn to the proof of the original proposition.

\begin{proof}[Proof of Proposition \ref{rs,g,gen}]
    We prove the statement for $\eis^{rs}$ and the statement for $\Eis^{rs,g}$ is similar.
    By Lemma \ref{iota-G}, $\iota_G\colon (T\sigma)^{rs}/T\to (G\sigma)^{rs}/G$ is a Galois covering. For any point $x$ in $T\sigma$, its stabilizer is $T^\sigma$ and thus of the same dimension. Thus its image in $T\sigma/T$ is closed. Moreover, for strongly regular semisimple $x$, the stabilizer of $x$ under the $N_G(T\sigma)$-action is also $T^\sigma$. Thus $\iota_G$ satisfies the conditions in Lemma \ref{gen-quotient} and then
    \[
    \CZ(D\Coh((T\sigma)^{rs}/T)) \simto \operatorname{End}(\iota_{G*})
    \simeq \operatorname{Hom}(\eis_B^{rs}, \eis_{B'}^{rs}) 
    \]
    As the intertwining $\eis_B^{rs}\simeq \eis_{B'}^{rs}$ is defined by $\eis_B^{rs}\simeq \iota_{G*}\simeq \eis_{B'}^{rs}$ in Definition \ref{rs, g},
    the isomorphism $J_{B'|B,0}^{rs}\colon \eis_B^{rs}\simeq \eis_{B'}^{rs}$ in \ref{rs, g} corresponds to $1$ in the left hand side and thus it is a generator.
\end{proof}

Next we analyze the structure of $T\sigma/T$ and the category $D\Coh(T\sigma/T)$.

\begin{lemma}
\label{splitting}
    The morphism $\rho \colon T\sigma/T\to T\sigma\sslash T=T_\sigma$ 
    is a $T^\sigma$-gerbe and 
    it has a (non-canonical) splitting $T\sigma/T\simeq T_\sigma \times */T^\sigma$.
\end{lemma}

\begin{proof}
    We prove the following general argument: For a homomorphism of tori $\psi \colon T_1\to T_2$, $T_2/T_1\simeq T_2/\im(\psi)\times */\ker(\psi)$ (the argument in the lemma is the special case $T_1=T_2=T$ and $\psi=1-\sigma$). In this case, $\im(\psi)$ is again a torus. We claim that there exists a subtorus $T_3$ of $T_2$ such that $T_2=T_3\times \im(\psi)$. In fact, the inclusion $\im(\psi)\subset T_2$ induces a surjection $X^*(T_2)\to X^*(\im(\psi))$ and we have a splitting $i$ of the surjection. Then we may take $T_3$ as the subtorus corresponding to the quotient $X^*(T_2)/\im(i)$ of $X^*(T_2)$.
    
    Then we have
    \[
    T_2/T_1=(T_3\times \im(\psi))/T_1\simeq T_3\times \im(\psi)/T_1\simeq T_2/\im(\psi)\times */\ker(\psi).\qedhere
    \]
\end{proof}
Thus there is a block decomposition of $D\Coh(\BL_{T\sigma})$, according to the character of $T^\sigma$:
\[
D\Coh(\BL_{T\sigma}) \simeq \bigoplus_{\chi\in X^*(T^\sigma)} D\Coh(\BL_{T\sigma})_\chi.
\]
The summand $D\Coh(\BL_{T\sigma})_1$ of this decomposition is equivalent to the category $D\Coh(T_\sigma)$ under $\rho^*$.
We remark that the decomposition is independent of the trivialization.

Moreover, denote by $\CO_\chi$ the invertible sheaf $\CO_{T_\sigma}\boxtimes k_\chi$ on $T\sigma/T$ (it depends on the choice of the trivialization a priori).
Then $\CF\mapsto \CF\otimes \CO_\chi$ defines an equivalence of categories
\[
D\Coh(\BL_{T\sigma})_1\simeq D\Coh(\BL_{T\sigma})_\chi.
\]

Since $\CO_{\BL_{T\sigma}}$ generates the subcategory $D\Coh(\BL_{T\sigma})_1$, the sheaf $\CO_\chi$ generates the subcategory $D\Coh(\BL_{T\sigma})_\chi$.
Hence the Bernstein center of such a subcategory is
\[
\CZ(D\Coh(\BL_{T\sigma})_\chi)\simeq
\End(\CO_\chi) \simeq \CO(T_\sigma).
\]

Applying the block decomposition to Proposition \ref{rs,g,gen},
we see that the intertwining operator for each subcategory is also a generator:

\begin{proposition}
\label{rs,g,gen,1}
    For any character $\chi$ of $T^\sigma$,
    the ring homomorphism 
    \[
    \CO(T_\sigma^{rs, g})= \CZ(D\Coh(\BL_{T\sigma}^{rs,g})_\chi)\to \End(\Eis^{rs,g})
    \] 
    is an isomorphism. 
    Here $\Eis$ is restricted to the subcategory $\CZ(D\Coh(\BL_{T\sigma}^{rs,g})_\chi)$.
    Moreover, the isomorphism $J_{B'|B}^{rs,g}$ in \ref{rs, g} generates the free $\CO(T_\sigma^{rs, g})$-module
    $\Hom(\Eis_B^{rs, g}, \Eis_{B'}^{rs, g})$. Similarly for $\eis^{rs}$.
\end{proposition}

\subsection{Intertwining over the generic locus}

We will extend the isomorphism in Definition \ref{rs, g} to the generic locus. 
More precisely, in this subsection we will prove the following statement.

\begin{proposition}[Intertwining over the generic locus]
    \label{g}
    Let $\chi$ be a character of $T^\sigma$ that is invariant under the $W^\sigma$-action.
    Denote by $\Eis^{g}$ the restriction of the spectral Eisenstein series functor to the generic locus.
    The isomorphism constructed in \ref{rs, g}, $J_{B'|B}^{rs, g} \colon \Eis_B^{rs, g} \simeq \Eis_{B'}^{rs, g}$ can be uniquely extended to the generic locus
    $J_{B'|B}^{g} \colon \Eis_B^{g} \simeq \Eis_{B'}^{g}$,
    after restricting to the category of sheaves 
    $D\Coh(\BL_{T\sigma})_\chi$.
    Moreover, the $\CO(\BL_{T\sigma}^{g})$-module
    $\Hom(\Eis_B^{g}, \Eis_{B'}^{g})$ is free and $J^g_{B'|B}$ is a generator of it.

    In the same way, the isomorphism $J^{rs}_{B'|B,0} \colon \eis_B^{rs} \simeq \eis_{B'}^{rs}$ in \ref{rs, g} can be extended uniquely to an isomorphism $J_{B'|B,0} \colon \eis_B\simeq \eis_{B'}$ after restricted to $D\Coh(T\sigma/T)_\chi$ and the resulting isomorphism generates the free $\CO(T\sigma/T)$-module
    $\Hom(\eis_B, \eis_{B'})$.
\end{proposition}

In the following proofs, we will consider the functor $\eis$ only. The statement about $\Eis^g$ follows directly from the isomorphism $\BL_{G\sigma}^g\simeq (G\sigma)^g/G$.

We remark that the condition that $\chi\in X^*(T^\sigma)^{W_\sigma}$ is equivalent to that after restricting to the neutral component $(T^\sigma)^\circ$ of $T^\sigma$, $\chi|_{(T^\sigma)^\circ}$ is $W^\sigma$-invariant. In fact, we have the following lemma.

\begin{lemma}
    Denote by $Z(G)$ the center of $G$. Then the homomorphism $Z(G)^\sigma \to \pi_0(T^\sigma)$ is surjective. In particular, $W^\sigma$ acts trivially on $\pi_0(T^\sigma)$.
\end{lemma}

\begin{proof}
    It suffices to prove that the homomorphism of characters $(X^*(T^\sigma))_{\text{tor}}\to X^*(Z(G)^\sigma)$ is injective, where "tor" denotes the torsion part. For any multiplicative group $S$, we have $X^*(S^\sigma)\simeq X^*(S)_\sigma$.
    Moreover, 
    $X^*(Z(G))\simeq X^*(T)/Q$, where $Q$ denotes the root lattice.
    Thus from the exact sequence
    \[
    0 \to Q \to X^*(T) \to X^*(T)/Q \to 0,
    \]
    we have a right exact sequence
    \[
    Q_\sigma \to X^*(T)_\sigma \to (X^*(T)/Q)_\sigma \to 0.
    \]
    As $\sigma$ acts by permutation of the roots, $Q_\sigma$ is free. Thus $(X^*(T)_\sigma)_{\text{tor}}\to (X^*(T)/Q)_\sigma$ is injective and we conclude.
\end{proof}

In particular, it $\chi$ is trivial on $(T^\sigma)^\circ$, it is $W^\sigma$-invariant.

It is worth noting that the requirement that $\chi$ is $W^\sigma$-invariant is necessary.

\begin{remark}
     The sheaves $\eis_B \CF$ and $\eis_{B'}\CF$ are not isomorphic for general $\CF$ in $D\Coh(T\sigma/T)$. For example, let $B'=w(B)$ and $\CF = i_{1*}\CO_\lambda$, the skyscraper sheaf at $1\in T$ for a weight $\lambda$ of $T$. Then the fiber of $\eis_B\CF$ and $\eis_{B'}\CF$ at $1 \in G$ is $R\Gamma(G^\sigma/B^\sigma, \CO_\lambda)$ and $R\Gamma(G^\sigma/B^\sigma, \CO_{w\lambda})$ respectively. By Borel--Weil--Bott theorem, if they are isomorphic for any $w \in W^\sigma$, $\lambda$ must be invariant under $W^\sigma$-action.
\end{remark}

For the proof of the main proposition,
we consider the following auxiliary definition of regular locus.
In \cite{xiao2018}, the regular locus is defined as follows.

\begin{definition}[Regular locus]
\label{regular}
    The regular locus $(G\sigma)^{r}$ is defined as
    \[
        (G\sigma)^r := \{g\sigma \in G\sigma \mid \dim C_{G}(g\sigma) = \dim T^\sigma \}
    \]
    where $C_{G}(g\sigma)$ is the twisted centralizer
    \[
        C_{G}(g\sigma) := \{h \in G \mid h(g\sigma)h^{-1} = g\sigma \} = \{h \in G \mid hg\sigma(h)^{-1} = g \}.
    \]
    The regular locus $(B\sigma)^r$ is defined as $(B\sigma)^r := (G\sigma)^r \cap B$. 
\end{definition}

Note that when $\sigma = 1$, it is the usual definition of regular locus.
We have the following lemma about the complement of $(B\sigma)^r$ and $(G\sigma)^r$.

\begin{lemma}[\cite{xiao2018}, corollary 5.2.21]
    The complement of the regular locus is of codimension at least $2$:
    $\operatorname{codim}(B\sigma-(B\sigma)^r)\geq 2$ in $B$ and 
    $\operatorname{codim}(G\sigma-(G\sigma)^r)\geq 2$ in $G$.
\end{lemma}

We have the following generalization of the Kostant's theorem. 

\begin{proposition}
    \label{kostant}
    Let $G$ be as above. Suppose that $G_0\to G$ is a covering such that $G_0$ is a product of a semisimple and simply connected group and a torus and the action $\sigma$ on $G$ can be lifted to a compatible action on $G_0$, which is denoted also by $\sigma$. 
    (Note that if such a lifting exists, then it is unique. Moreover, fixing $G$,
    we have $G_0 = G^{sc}\times Z(G)^\circ$ satisfies this condition, where $G^{sc}$ is the universal cover of the derived subgroup of $G$ and $Z(G)^\circ$ is neutral component of the center of $G$.) Denote by $T_0$ the inverse image of $T$ in $G_0$ and similarly for $B_0$. We have $G_0\to G$ is a Galois covering with Galois group $\ker(G_0\to G)=\ker(B_0\to B)=\ker(T_0\to T)$.
    
    Then the squares in the commutative diagram are cartesian:
    \[
    \begin{tikzcd}
        \widetilde{G\sigma}^r \ar[r, "\Pi^r"]\ar[d, "P_B^r"] & (B\sigma)^r/ B_0 \ar[r]\ar[d, "p_B^r"] & (T_0\sigma\sslash T_0)/\Gal\ar[d]\ar[r] & T\sigma\sslash T \ar[d, "p_T"]  \\
        (G\sigma)^r \ar[r, "\pi^r"] & (G\sigma)^r/ G_0 \ar[r] & (T_0\sigma\sslash N_{G_0}(T_0\sigma))/\Gal \ar[r] &T\sigma\sslash' N_G(T\sigma),
    \end{tikzcd}
    \]
    where $\widetilde{G\sigma}$ is the twisted Grothendieck--Springer resolution
    \[
    \widetilde{G\sigma} := \{(g\sigma, h) \in G\sigma\times G/B\mid h^{-1}g\sigma(h)\in B \},
    \]
    $\widetilde{G\sigma}^r$ is the inverse image of $(G\sigma)^r$,
    $T\sigma\sslash' N_G(T\sigma)$ is 
    \[
    (T_0\sigma\sslash N_{G_0}(T_0\sigma))/\ker(T_{0,\sigma}\to T_\sigma)
    \]
    and $(G\sigma)^r/ G_0\to (T_0\sigma\sslash N_{G_0}(T_0\sigma))/\Gal$ is induced by the twisted Chevalley morphism for $G_0$.
\end{proposition}

We will call $G_0\to G$ to be a simply connected cover of $G$ if it satisfies the conditions in this proposition.

\begin{proof}
    We firstly assume that $G_0\to G$ is an isomorphism.
    The case that $G_0$ is semisimple is proved by Xiao and Zhu
    (\cite{xiao2018}, Proposition 5.3.1). The case that $G$ is a torus is trivial as $G=B=T$ in this case.
    In the case that $G$ is the product of a semisimple and simply connected subgroup and a torus,
    any construction in the diagram about $G$ is the product of such construction on the simply connected subgroup and such construction on the torus and the result holds.
    
    In general, we have a commutative diagram for $G_0$ by discussions above. As all constructions except the two right terms in this diagram are obtained by quotienting out $\Gal$ on those constructions about $G_0$, it suffices to prove that the rightmost square is cartesian. Note that $\Gal=\ker(G_0\to G)=\ker(T_0\to T)$ acts on $T_{0,\sigma}$ by multiplication. We have 
    \[
    T_\sigma\simeq T_{0,\sigma}\sslash \Gal\simeq T_{0,\sigma}/\im(\Gal\to T_{0,\sigma})
    = T_{0,\sigma}/\ker(T_{0,\sigma}\to T_\sigma)
    \]
    and thus the rightmost square is cartesian.
\end{proof}

\begin{remark}
    The space $T\sigma\sslash' N_G(T\sigma)$ is independent of the choice of $G_0$.
    The proof of this statement is as follows.
    Let $G_0\to G$ be as above. Then $G_0$ is of the form $G^{sc}\times \tilde{Z}$ with a covering $\tilde{Z}\to Z(G)^\circ$. 
    Denote by $G_1:= G^{sc}\times Z(G)^\circ$ and $T_1$ its torus lifting $T$. 
    By definition of $T\sigma\sslash' N_G(T\sigma)$ for $G_0$, it is
    \[
    (T_0\sigma\sslash N_{G_0}(T_0\sigma))/\ker(T_{0,\sigma}\to T_\sigma)
    \simeq ((T_0\sigma\sslash N_{G_0}(T_0\sigma))/\ker(T_{0,\sigma}\to T_{1,\sigma}))/
    \ker(T_{1,\sigma}\to T_\sigma).
    \]
    Moreover, as the construction $T_0\sigma\sslash N_{G_0}(T_0\sigma)$ about $G_0$ is the product of such construction about $G^{sc}$ and 
    $\tilde{Z}$ and similarly for $G_1$, we have
    \[
    (T_0\sigma\sslash N_{G_0}(T_0\sigma))/\ker(T_{0,\sigma}\to T_{1,\sigma})
    \simeq (T_0\sigma\sslash N_{G_0}(T_0\sigma))/\ker(\tilde{Z}_\sigma\to Z(G)^\circ_\sigma)
    \simeq T_1\sigma\sslash N_{G_1}(T_1\sigma).
    \]
    Thus $T\sigma\sslash' N_G(T\sigma)$ defined via $G_0$ and $G_1$ are isomorphic.
\end{remark}

Let $G_0\to G$ be a simply connected cover.
Denote by $\beta$ the $\Gal$-gerbe $G\sigma/G_0\to G\sigma/G$. 
Then $\beta^*$ induces an equivalence of categories
\[
D\Coh(G\sigma/G) \xrightarrow{\sim} D\Coh(G\sigma/G_0)_{1,\Gamma},
\]
where the latter denotes the subcategory of $D\Coh(G\sigma/G_0)$ that consists of sheaves with trivial $\Gamma=\Gal$-action. Denote by $q_B^r\colon (B\sigma)^r/B_0\to T_\sigma$ and $\pi_T^r\colon (G\sigma)^r/G_0\to T\sigma\sslash' N_G(T\sigma)$. By definition of $\eis_B$ and proper base change in the diagram in Proposition \ref{kostant}, we have
\[
\beta^*\eis_B^r\simeq p_{B*}^rq_B^{r*} \simeq  \pi_T^{r*}p_{T*}.
\]

To prove Proposition \ref{g}, we firstly consider the block $D\Coh(T\sigma/T)_1$.
By Lemma \ref{splitting}, $\rho^*$ induces an equivalence of categories
\[
    D\Coh(T\sigma\sslash T) \simeq D\Coh(T\sigma/T)_1.
\]
We will then identify these two categories if there is no ambiguity. 

The following proposition is about the coherent Springer sheaf, which is a generalization of a classical result in Springer theory,

\begin{proposition}
\label{vanish}
    $\eis_B \CO_{T_\sigma} = (p_B)_*\CO_{B\sigma/ B}$ is an honest coherent sheaf on $G\sigma/ G$.
\end{proposition}

\begin{proof}
    In the case $\sigma = 1$, it is proved in \cite{hesselink1976cohomology}, Theorem B. In general, we have not found a reference which contains a simple proof of the argument and we refer to the article of Zhu about the categorical local Langlands correspondence.
    In \cite{zhu2025tamecategoricallocallanglands}, Corollary 5.11, it is proved that $\Eis_B \CO_{T_\sigma}$ is an honest coherent sheaf. As $\Eis_B^g\simeq \eis_B^g$, $(p_B)_*\CO_{B\sigma/ B}$ is an honest coherent sheaf on $(G\sigma)^g/ G$. We may change the parameter $q=p^n$ in the definition of $\BL_{G\sigma}$ and conclude that $\eis_{B}\CO_{T_\sigma}$ is an honest coherent sheaf on $G^{g,q}$, the generic locus for $q$. For all such $q$, the union of $G^{g,q}$ is $G$. Thus $\eis_B\CO_T$ is an honest coherent sheaf. 
\end{proof}

\begin{lemma}
\label{locallyfreeeis}
    Let $G_0\to G$ be a covering as in Proposition \ref{kostant}. 
    Consider the following diagram, extending the diagram above:
    \begin{equation}
    \label{kostantdiagram}
    \begin{tikzcd}
        \widetilde{G\sigma} \ar[r, "\Pi"]\ar[d, "P_B"] & B\sigma/ B_0 \ar[r, "q_B"]\ar[d, "p_B"] & T\sigma\sslash T \ar[d, "p_T"] \\
        G\sigma \ar[r, "\pi"] & G\sigma/ G_0 \ar[r, "\pi_T"] & T\sigma\sslash' N_G(T\sigma).
    \end{tikzcd}
    \end{equation}
    Then the isomorphism $\beta^*\eis_B^r\simeq \pi_T^{r*}p_{T*}$ above extends uniquely to $\beta^*\eis_B\simeq \pi_T^*p_{T*}$. Moreover, $\eis_B \CO_{T_\sigma}$ is locally free.
\end{lemma}

\begin{proof}
    In this proof, we will denote $j_X$ the open immersion $X^r\to X$ where $X$ appears in the above diagram.

    It suffices to show that after pulling back to $G\sigma$ via $\pi$, the isomorphism extends uniquely to the whole $G\sigma$ and $\pi^*\beta^*\eis_B\CO_{T_\sigma}$ is locally free. We have
    \[
    \pi^*\beta^*\eis_B\CO_{T\sigma/T} \simeq \pi^*(p_B)_*\CO_{B\sigma/ B}
    \simeq (P_B)_*\CO_{\widetilde{G\sigma}}.
    \]
    As $\widetilde{G\sigma}-\widetilde{G\sigma}^r$ is of codimension at least $2$, $\CO_{\widetilde{G\sigma}}\simeq R^0j_{\widetilde{G\sigma}*}\CO_{\widetilde{G\sigma}^r}$ by \cite{stacks-project}, tag 0EBJ about the unique extension of reflexive coherent sheaves. By Proposition \ref{vanish}, the higher cohomology of $(P_B)_*\CO_{\tilde{G}_\sigma}$
    vanishes.
    Thus $(P_B)_*\CO_{\widetilde{G\sigma}} \simeq R^0j_{G*}(P_B^r)_*\CO_{\widetilde{G\sigma}^r}$.
    Then by Proposition \ref{kostant} and proper base change,
    \[
    (P_B)_*\CO_{\widetilde{G\sigma}}\simeq R^0j_{G*}j^*_G(\pi_T\pi)^*p_{T*} \CO_{T_\sigma}.
    \]
    As $p_T$ is finite and flat, $(\pi_T\pi)^*p_{T*} \CO_{T\sslash_\sigma T}$ is locally free over $G\sigma$ and again by \cite{stacks-project}, tag 0EBJ we have
    \[
        R^0j_{G*}j^*_G(\pi_T\pi)^*p_{T*} \CO_{T_\sigma}
        \simeq (\pi_T\pi)^*p_{T*} \CO_{T_\sigma}.
    \]
    Thus  
    \[
    \pi^*\beta^*\eis_B\CO_{T\sigma/T}\simeq (\pi_T\pi)^*p_{T*} \CO_{T_\sigma}
    \]
    and $\eis_B \CO_{T\sigma/ T}$ is a locally free $\CO_{G\sigma/G}$-module. From its construction, this isomorphism extends the isomorphism $\beta^*\eis^r_B\simeq \pi_T^{r*}p_{T*}$ over the regular locus applying to $\CO_{T_\sigma}$ .
    
    Recall that $\CO_{T_\sigma}$ generates of the category $D\Coh(T_\sigma)$.
    The extension of the isomorphism is unique as the regular locus in $G\sigma$ is open and dense and $\pi^*\eis_B \CO_{T\sigma/ T}$ is locally free.
    The construction is $\CO_{T_\sigma}$-linear by uniqueness of the extension and by linearity we may induce an isomorphism of functors $\beta^*\eis_B\simeq \pi_T^*p_{T*}$.
\end{proof}

Now we can conclude the proof of the main proposition for $\chi=1$.

\begin{proof}[Proof of Proposition \ref{g} for $\chi=1$]
    We prove the arguments about $\eis_B\simeq \eis_{B'}$ and then the arguments about $\Eis^g_B\simeq \Eis^g_{B'}$ follow in the same way.

    Take a covering $G_0\to G$ such that $\sigma$ can be lifted to $G_0$ and $G_0$ is a product of a simply connected semisimple group and a torus. 
    Then there is an isomorphism $\beta^*\eis_B\simto \beta^*\eis_{B'}$ via the composition
    \[
    \beta^*\eis_B\simto \pi_T^*p_{T*}\simto \beta^*\eis_{B'}
    \]
    in Lemma \ref{locallyfreeeis}. As $\beta^*$ is fully faithful, this isomorphism induces an isomorphism $\eis_B\simto \eis_{B'}$ and we  define $J_{B'|B}^g$ to be this isomorphism.

    As we have $(G\sigma)^{rs}/G \simeq (T\sigma)^{rs}/N_G(T\sigma)$ via the Chevalley map, we have $\beta^*\iota_{G*}\simeq \pi_T^*p_{T*}$ over the regular semisimple locus. Thus by definition of the intertwining operator $J_{B'|B}^\circ$ in Definition \ref{rs, g},
    $J_{B'|B}^g$ extends $J_{B'|B}^\circ$.

    For the uniqueness of the extension, it suffices to prove that the restriction
    $\Hom(\eis_B, \eis_{B'}) \to \Hom(\eis_B^{rs}, \eis_{B'}^{rs})$
    is injective. As $\CO_{T_\sigma}$ is a generator of $D\Coh(T_\sigma)$,
    $\eis_B\CO_{T_\sigma}$ and $\eis_{B'}\CO_{T_\sigma}$ are locally free by Lemma \ref{locallyfreeeis} and the regular semisimple locus is open and dense in $G\sigma$. The statement holds.
    
    For the statement about the generator, notice that
    \[
    \Hom(\eis_B, \eis_{B'}) \subset \Hom(\eis_B^{rs}, \eis_{B'}^{rs})
    \simeq \CO(\BL_{T\sigma}^{rs, g}) \simeq \CO(\BL_{T\sigma}^g)[L(0, \ad_\fn)].
    \]
    Thus if the isomorphism $J_{B'|B}^g$ were not a generator, it would be divisible by some nontrivial factor of $L(0, \ad_\fn)^{-1}$. But this is impossible, as $J^g_{B'|B}(\CO_{T_\sigma})$ defines an isomorphism of locally free $G\sigma/G$-modules and it cannot be divisible by such factor.
\end{proof}

Next we consider the case for general $\chi$.

\begin{lemma}
    \label{Gabsigma}
    If $G$ is semisimple and simply connected, $X^*(T^\sigma)^{W^\sigma}=1$.
\end{lemma}

\begin{proof}
    As $G$ is semisimple and simply connected, $G^\sigma$ is a connected reductive group with Weyl group $W^\sigma$ by \cite{steinberg1968endomorphisms}, Theorem 8.1. Thus the abelianization $(G^\sigma)_{\text{ab}}$ of $G^\sigma$ is trivial and
    \[
    X^*(T^\sigma)^{W^\sigma} \simeq X^*((G^\sigma)_{\text{ab}}) = 1.
    \]
\end{proof}

\begin{lemma}
\label{locally_free_general}
    For any $\chi \in X^*(T^\sigma)^{W^\sigma}$, the sheaf $\eis_B\CO_\chi$ is locally free (and in particular concentrated in degree $0$).
\end{lemma}

\begin{proof}
    Choose a simply connected cover $G_0\to G$. It induces a morphism $\alpha_G\colon G_0\sigma/G_0\to G\sigma/G$, which is a composition of a Galois covering $G_0\sigma/G_0\to G\sigma/G_0$ and a gerbe $G\sigma/G_0\to G\sigma/G$. Thus it suffices to prove that $\alpha_G^*\eis_B\CO_\chi$ is locally free.

    Denote by $\alpha_T$ the morphism $T_0\sigma/T_0\to T\sigma/T$.
    By definition of $\eis_B$ and proper base change, 
    \[
    \alpha_G^*\eis_B\CO_\chi\simeq \eis_{B_0}\alpha_T^*\CO_\chi
    \simeq \eis_{B_0}\CO_{\tilde{\chi}},
    \]
    where $\tilde{\chi}$ is the pullback of $\chi$ along $T_{0,\sigma}\to T_\sigma$. Hence it suffices to prove the statement for $G_0$.

    By definition of $G_0$, it is a product of a torus and a semisimple and simply connected group. As the construction of $\eis$ is compatible with direct product and the statement for a torus is trivial (since it is isomorphic to its Borel and its torus), it remains to prove the case that $G$ is a semisimple and simply connected group.

    By lemma \ref{Gabsigma}, $\chi=1$ in this case. 
    Thus this lemma is reduced to Lemma \ref{locallyfreeeis} and we conclude.
\end{proof}

\begin{lemma}
    For $\chi\in X^*(T^\sigma)^{W^\sigma}$, 
    $\CO_\chi$ can be descended to an invertible sheaf $\overline{\CO_\chi}$ over $T\sigma/N_G(T\sigma)$.
    In particular, over the regular and semisimple locus $(T\sigma)^{rs}/T$, $\CO_\chi^{rs}$ can be descended to an invertible sheaf $\overline{\CO_\chi^{rs}}$ over $(G\sigma)^{rs}/G\simeq (T\sigma)^{rs}/N_G(T\sigma)$.
\end{lemma}

\begin{proof}
    Let $G_0=Z \times G^{sc} \to G$ be a simply connected cover of $G$.
    Then we have a covering $T_0=Z\times T^{sc}\to T$. Denote by $\Gamma$ the kernel of the covering. Consider the following commutative diagram
    \[
    \begin{tikzcd}
	T_0\sigma & T\sigma/T_0 & T\sigma/T \\
	& T\sigma/N_{G_0}(T_0\sigma) & T\sigma/N_G(T\sigma)
	\arrow[from=1-1, to=1-2]
	\arrow[from=1-1, to=2-2]
	\arrow["\beta_T", from=1-2, to=1-3]
	\arrow["\gamma_T", from=1-2, to=2-2]
	\arrow[from=1-3, to=2-3]
	\arrow[from=2-2, to=2-3]
    \end{tikzcd}
    \]
    By Lemma \ref{Gabsigma}, 
    $\chi|_{(T^{sc})^\sigma}$ is trivial. 
    Thus for $\tilde{\alpha}_T\colon T_0\sigma\to T\sigma/T$, we have a $T_0 \times \Gamma$-equivariant sheaf on $T_0\sigma$,
    \[
    \tilde{\alpha}_T^* \CO_\chi \simeq \CO_{\chi|_{Z\sigma}} \boxtimes \CO_{T^{sc}\sigma},
    \]
    where $T^{sc}$ acts trivially on $\CO_{T^{sc}\sigma}$.
    As $N_{G_0}(T_0\sigma) = Z \times N_{G^{sc}}(T^{sc}\sigma)$,
    the sheaf $\CO_{\chi|_{Z\sigma}} \boxtimes \CO_{T^{sc}}$ on $T_0\sigma$ has a $\Gamma\times N_{G_0}(T_0\sigma)$-equivariant structure extending the equivariant structure above, such that $N_{G^{sc}}(T^{sc}\sigma)$ acts trivially on it. 

    Thus the sheaf above descends to a sheaf (denoted by $\CO_\chi'$) on the stack
    \[
    T_0\sigma/(\Gamma\times N_{G_0}(T_0\sigma))\simeq T\sigma/N_{G_0}(T_0\sigma).
    \] 
    The morphism $T\sigma/N_{G_0}(T_0\sigma) \to T\sigma/N_G(T\sigma)$ is a $\Gamma$-gerbe. By construction, $\gamma_T^*\CO_\chi'\simeq \beta_T^*\CO_\chi$. Thus the $\Gamma$-action from the $\Gamma$-gerbe structure of $\beta_T$ on any fiber of $\gamma_T^*\CO_\chi'$ is trivial.  
    Hence the $\Gamma$-action on any fiber of $\CO_\chi'$ is trivial. Thus $\CO_\chi'$ comes from a sheaf on $T\sigma/N_G(T\sigma)$ as we desired.
\end{proof}

Recall that for any $\chi\in X^*(T^\sigma)$, 
the automorphism of $D\Coh(T\sigma/T)$ given by 
\[
\CF \mapsto \CF_\chi := \CF \otimes \CO_\chi
\]
induces an equivalence of categories
\[
D\Coh(T\sigma/T)_1 \simto D\Coh(T\sigma/T)_\chi.
\]
By the projection formula, we have for any $\CF\in D\Coh(T\sigma/T)$,
\[
\eis_B^{rs}\CF_\chi\simeq \eis_B^{rs}\CF \otimes \overline{\CO_\chi^{rs}}.
\]
Moreover, tensoring by $\overline{\CO_\chi^{rs}}$ defines an isomorphism
(which will be denoted by $(-)\otimes \overline{\CO_\chi}$
\begin{equation}
\label{tensor_by_chi_rs}
    \Hom(\eis_B^{rs}\CF, \eis_{B'}^{rs}\CG)\simto \Hom(\eis_B^{rs}\CF_\chi, \eis_{B'}^{rs}\CG_\chi)  
\end{equation}

For a locally free sheaf $\CF\in D\Coh(T\sigma/T)_\chi$ where $\chi$ is invariant by $W^\sigma$, $\eis_B\CF$ is locally free by Lemma \ref{locally_free_general} and thus $\eis_B\CF\to \eis_B^{rs}\CF$ is injective.

\begin{lemma}
\label{tensor_by_chi}
    For locally free sheaves $\CF, \CG \in D\Coh(T\sigma/T)_1$,
    the isomorphism $(-)\otimes \overline{\CO_\chi}$ in \ref{tensor_by_chi_rs} restricts to an isomorphism
    \[
    \Hom(\eis_B\CF, \eis_{B'}\CG)\simto \Hom(\eis_B\CF_\chi, \eis_{B'}\CG_\chi). 
    \]
    Moreover, this isomorphism extends to all $\CF, \CG\in D\Coh(T\sigma/T)_1$.
\end{lemma}

\begin{proof}
    For the statement about locally free sheaves,
    it suffices to prove the statement after pulling back along $\alpha_G\colon G_0\sigma/G_0\to G\sigma/G$ as before. 
    Thus we reduce to the case for $G_0$. As all constructions in the statement are compatible with direct products and the statement for a torus is trivial, we reduce to the case that $G$ is simply connected.
    In this case, $\chi=1$ by Lemma \ref{Gabsigma} and $(-)\otimes \overline{\CO_\chi}$ is the identity. Thus we conclude.

    For general $\CF, \CG$, note that the construction $(-)\otimes \overline{\CO_\chi}$ in \ref{tensor_by_chi_rs} is $\CO_{T_\sigma}$-linear. Thus the isomorphism extends by linearity to all $\CF, \CG\in D\Coh(T\sigma/T)_1$.
\end{proof}

Then we can prove the main proposition for general $\chi$.

\begin{proof}[Proof of Proposition \ref{g}]
    We have already an isomorphism 
    $J_{B'|B,0}\in \Hom(\eis_B, \eis_{B'})_1$. We claim that 
    \[J_{B'|B,0}\otimes \overline{\CO_\chi} \in \Hom(\eis_B,\eis_{B'})_\chi\] 
    defined in Lemma \ref{tensor_by_chi} is an isomorphism and it is the desired intertwining operator. 
    Here the subscript $\chi$ denotes the $\Hom$-module for the category $D\Coh(T\sigma/T)_\chi$.

    Take a simply connected cover $G_0\to G$.
    The statement that $J_{B'|B,0}\otimes \overline{\CO_\chi}$ defines an isomorphism holds true by $\chi=1$ case and the same descent arguments in the previous proofs. Moreover, by definition of the intertwining operator over the regular and semisimple locus, we have for any $\CF \in D\Coh(T\sigma/T)$,
    \[
    (J_{B'|B,0}^{rs}\otimes \overline{\CO_\chi})(\CF) = J^{rs}_{B'|B,0}(\CF_\chi)
    \]
    and thus $J_{B'|B,0}\otimes \overline{\CO_\chi}$ is the extension of the intertwining operator in the regular semisimple locus. Finally, the statement that $J_{B'|B,0}\otimes \overline{\CO_\chi}$ is a generator of the free module $\Hom(\eis_B,\eis_{B'})_\chi$ follows from the $\chi=1$ case and Lemma \ref{tensor_by_chi}.
\end{proof}

\subsection{Integral intertwining operator}

We will now extend the intertwining operator over the generic locus in \ref{g} to the entire $\BL_{T\sigma}$ and obtain the integral intertwining operator. More precisely, we will prove the following statement.

\begin{proposition}[Integral intertwining operator for the principal block]
    \label{all}
    Let $\chi$ be a character of $T^\sigma$ which is invariant under the $W^\sigma$-action.
    The homomorphism constructed in \ref{g}, $J_{B'|B}^{g} \colon \Eis_B^{g} \simeq \Eis_{B'}^{g}$ over $D\Coh(\BL_{T\sigma}^g)_\chi$ can be uniquely extended to a homomorphism
    $j_{B'|B} \colon \Eis_B \to \Eis_{B'}$,
    after the multiplication by $L(1, \ad_{\fn_{B'}/\fn_{B}})^{-1}$. The result is denoted by $j_{B'|B}$.
    Moreover, the operator $j_{B'|B}$ generates the free $\CO(\BL_{T\sigma})$-module $\Hom(\Eis_B, \Eis_{B'})$.
\end{proposition}

We firstly consider the structure of the mild singularities of $\BL_{B\sigma}$ and $\BL_{G\sigma}$. 

We introduce the following notations.
By the explicit formula \ref{L-function} of the L-function,
over the unique factorization ring $\CO(T_\sigma)$,
$L(1, \ad_{\fn\oplus\bar{\fn}})^{-1}$ has no multiple factors.
Let $f$ be an irreducible factor of $L(1, \ad_{\fn\oplus\bar{\fn}})^{-1}$, corresponding to the orbit $O$ of roots. Let $\BL_{B\sigma, f}$ denote the non-vanishing locus of $L(1, \ad_{\fn\oplus\bar{\fn}})^{-1}/f$ in $(B\sigma)^{rs}$. Similarly we define $(B\sigma)_f\subset B\sigma$, $\BL_{T\sigma, f}\subset\BL_{T\sigma}$ and $(T\sigma)_f\subset T\sigma$. Let $\BL_{B\sigma, f}^\wedge$ denote the completion of $\BL_{B\sigma, f}$ along the zero locus $V(f)$ of $f$. Similarly we define $\BL_{T\sigma, f}^\wedge$.

\begin{lemma}[Mild singularities of $\BL_{B\sigma}$]
    \label{formal-B}
    Notations as above.
    $\BL_{B\sigma, f}$ is of the form 
    \[
    \operatorname{Spec}(\mathcal{O}((T\sigma)_f)[u]/(fu))/ T
    \]
    over $(T\sigma)_f/T$ if $f$ divides $L(1, \ad_{\fn})^{-1}$ and it is of the form $\BL_{T\sigma, f}\simeq (T\sigma)_f/ T$
    if $f$ divides $L(1, \ad_{\bar{\fn}})^{-1}$, where $T$ acts on $T\sigma$ by $\sigma$-conjugation and $T$ acts on $u$ by character $\alpha^{-1}$, for a chosen root $\alpha\in O$.
    
    Moreover, the projection $\BL_{B\sigma, f} \to \BL_{T\sigma. f}$ is the annihilation of the coordinate $u$
    if $f$ divides $L(1, \ad_{\fn})^{-1}$ and it is an isomorphism if $f$ divides $L(1, \ad_{\bar{\fn}})^{-1}$.
\end{lemma}

\begin{proof}
    Consider the pullback of $\BL^{rs}_{B\sigma}$ along the isomorphism $(T\sigma)^{rs}/ T\to (B\sigma)^{rs}/ B$. We have that 
    \[
    \BL^{rs}_{B\sigma} \simeq \{ (t, n)\in (T\sigma)^{rs} \times \fn \mid
    \ad(t)(\sigma(n))=qn\}/ T,
    \]
    and similarly for $\BL_{B\sigma, f}$.
    By explicit description \ref{L-function} of the $L$-function $L(1, \ad_\fn)$, the irreducible factors of $L(1, \ad_\fn)^{-1}$ and $L(1, \ad_{\bar{\fn}})^{-1}$ are disjoint. 
    
    If $f$ does not divide $L(1, \ad_\fn)^{-1}$, $\ad_\fn(t) \circ \sigma - q$ is invertible on $\fn$ for $t \in \BL_{B\sigma, f}$. Thus $(t,n)\in \BL_{B\sigma, f}$ should satisfy $n=0$ and $\BL_{B\sigma, f}$ is isomorphic to $\BL_{T\sigma, f}$.

    If $f$ divides $L(1, \ad_\fn)^{-1}$, again by \ref{L-function}, $f$ divides $L(1, \ad_{\fn_O})^{-1}$ for some orbit $O$ and we have that
    \[
    \BL_{B\sigma, f} \simeq \{ (t, n)\in (T\sigma)_f \times \fn_O \mid
    \ad(t)(\sigma(n))=qn\}/ T,
    \]
    where $\fn_O = \bigoplus_{\alpha\in O}\fn_\alpha$, as the operator $\ad_\fn(t) \circ \sigma - q$ is invertible on $\fn/\fn_O$. Over $\fn_O$, $\sigma$ is a cyclic permutation matrix (up to a sign) and $\ad(t)$ is an invertible diagonal matrix $\operatorname{diag}(\{\alpha(t)\}_{\alpha\in O})$ under the basis in the pinning $\{X_\alpha\in \fn_\alpha\}$. 
    
    Fix an $\alpha\in O$. Then $n\in \fn_O$ satisfying $\ad(t)(\sigma(n))=qn$ is uniquely determined by its coordinate $u_\alpha$ in $\fn_\alpha$ and the equation $\ad(t)(\sigma(n))=qn$
    reduces to $L(1, \ad_{\fn_O})^{-1}u_\alpha = 0$. More precisely, the projection
    \[
    \{ (t, n)\in (T\sigma)_f \times \fn_O \mid
    \ad(t)(\sigma(n))=qn\}/ T \to 
    \{ (t, n)\in (T\sigma)_f \times \fn_\alpha \mid
    L(1, \ad_{\fn_O})^{-1}u_\alpha \}/ T
    \]
    is an isomorphism.
    
    Abbreviate $u_\alpha$ by $u$. Then
    \[
    \BL_{B\sigma, f} \simeq \{ (t, x)\in (T\sigma)_f \times k \mid
    fu = 0\}/ T \simeq \operatorname{Spec}(\mathcal{O}((T\sigma)_f)[u]/(fu))/ T.
    \]
    As $T$ acts on $\fn_\alpha$ by the character $\alpha$, $T$ acts on $u$ by $\alpha^{-1}$. 
\end{proof}

In the rank $1$ case, the bad locus in $\BL_{B\sigma}$ decomposes into a disjoint union of irreducible components.

\begin{lemma}[Rank $1$ case]
    \label{ord2}
    If the Weyl group $W$ satisfies $|W^\sigma| = 2$, for any irreducible factor $f$ of $L(1,\ad_{\fn\oplus \bar{\fn}})^{-1}$ in $\CO(T\sigma/T)$, its zero locus in $\BL_{B\sigma}$ (similarly for $B\sigma/B$) lies in $\BL_{B\sigma}^{rs}$ and these zero loci are mutually disjoint.

    For $G$, we have a similar result: 
    Let $f$ be as above. The zero locus of $fs(f)$ in $\BL_{G\sigma}$ (similarly for $G\sigma/G$) is contained in the regular semisimple locus $\BL_{G\sigma}^{rs}$ and these zero loci are mutually disjoint. 
    Here, as the function $fs(f)$ is $W^\sigma$-invariant, we may view it as a function on $G\sigma/G$ and $\BL_{G\sigma}$ via the twisted Chevalley map. 

    Moreover, the morphism $p_0\colon B\sigma/B \to G\sigma/G$ induces an isomorphism between $V(f)\subset B\sigma/B$ and the zero locus of $fs(f)$ in $G$, $V(fs(f))_G$ for such $f$.
\end{lemma}

\begin{proof}
    In this case, there is a unique positive root $\alpha$ of $W^\sigma$ and the number of positive orbits is $1$ (the orbit is denoted by $O$ in this case) or $2$ (the orbit is denoted by $O^+$, $O^-$ in this case). Then by \ref{L-function}, the $L$-function is 
    \[
    L(1, \ad_{\fn}|_O)^{-1} = 
    \begin{cases}
    1 - q^{-|O|}e^{\alpha}(t) & \text{$\alpha$ is of type A} \\
    (1 + q^{-|O^+|}e^{\alpha}(t))(1 - q^{-|O^-|}e^{\alpha}(t)) & \text{$\alpha$ is of type BC}
    \end{cases}
    \]
    Moreover, $\BL_{B\sigma}^{rs}$ is the nonzero locus of $L(0,\ad_{\fn})^{-1}$, which is $1 - e^{\alpha}(t)$ or $(1 - e^{\alpha}(t))(1 - e^{2\alpha}(t))$ according to the number of orbits.
    In both cases, the zero loci of the factors of $L(1,\ad_{\fn\oplus \bar{\fn}})^{-1}$ are disjoint and lie in $\BL_{B\sigma}^{rs}$.

    The statements about $G$ follows similarly from the above description of the $L$-function. For the statement about $V(f)\to V(fs(f))_G$, note that $V(f)\subset (B\sigma)^{rs}/B$ and $V(fs(f))_G\subset (G\sigma)^{rs}/G$. After restricting to $(B\sigma)^{rs}/B$, $p_0$ is \'etale by Proposition \ref{complex}. Thus $p_0^{-1}(V(fs(f))_G) = V(fs(f))\to V(fs(f))_G$ is a 2-fold Galois covering. For $f$ dividing $L(1, \ad_{\fn_O})^{-1}$, $s(f)$ divides $L(1, \ad_{\fn, s(O)})^{-1}$ and we have $\alpha_{s(O)}=-\alpha_O$. Thus $V(f)$ and $V(s(f))$ are disjoint by descriptions of the $L$-function above. Thus $V(f)\to V(fs(f))_G$ is an isomorphism.
\end{proof}   

In particular, in the rank $1$ case, $V(f)$ is closed in $\BL_{B\sigma}$. 

Now we will pass to the completions along a point in $V(f)$ or the entire $V(f)$ in the rank $1$ case.

\begin{definition}[Completions]
    Let $f$ be a irreducible factor of $L(1, \ad_{\fn})^{-1}$ in $\CO(T\sigma/T)$.
    Let $\pi_G \colon \BL_{G\sigma}\to G\sigma/ G$ be the projection. 
    
    For general $G$, let $x\in V(f)$ be a strongly regular semisimple element. We define $\BL_{G\sigma, p_0(x)}^\wedge$ as the completion of $\BL_{G\sigma}$ along $\pi_G^{-1}(p_0(x))$.

    Moreover, in the rank $1$ case, we define a global version: $\BL_{G\sigma, f}^\wedge$ is defined as the completion of $\BL_{G\sigma}$ along $\pi_G^{-1}(V(fs(f))$. 
\end{definition}

Here we consider the strongly regular semisimple points for the following reason: The completion of $p_0$ at $x$, $(B\sigma/B)_x^\wedge \to (G\sigma/G)_{p_0(x)}^\wedge$ is an isomorphism for strongly regular semisimple $x$, since $p_0$ is \'etale at $x$ and the isotropy subgroups at $x$ and $p_0(x)$ are isomorphic. 

\begin{lemma}
    Let $f$ be a irreducible factor of $L(1, \ad_{\fn})^{-1}$ in $\CO(T\sigma/T)$.
    The strongly regular semisimple points form a dense open set in $V(f)\subset B\sigma/B$.
\end{lemma}

\begin{proof}
    For any $w\in W^\sigma$, the action of $w$ on $T_\sigma$ defines an automorphism of $T_\sigma$. Thus for $w\neq 1$, there is a finite number of nontrivial characters $\chi_{1,w}, \dots, \chi_{i_w,w}$ such that $x$ is fixed by $w$ if and only if $x$ lies in the zero locus of $\chi_{j,w}-1$ for each $j$.

    The strongly regular semisimple locus is the complement of 
    \[
    \bigcup_{w\in W-\{1\}} \bigcap_j V(\chi_{j,w}-1).
    \]
    Thus it is open. Moreover, by the explicit description of the $L$-function in Proposition \ref{L-function}, $f$ is a irreducible factor of $1\pm q^{-|O|}\alpha(t)$. Thus $f$ does not divides any $\chi_{j,w}-1$ and the intersection of $V(f)$ and the strongly regular semisimple locus is dense.
\end{proof}

\begin{lemma}[Mild singularities of $\BL_{G\sigma}$]
    \label{formal-G}
    Let $f$ be a irreducible factor of $L(1, \ad_{\fn})^{-1}$ in $\CO(T\sigma/T)$ and $x\in V(f)\subset B\sigma/B$ be a strongly regular and semisimple point.
    The pullback of $\BL_{G\sigma, p(x)}^\wedge$ along $p_B \colon \BL_{B\sigma} \to \BL_{G\sigma}$ is of the form
    \[
    \coprod_{w\in W^\sigma} \BL_{B\sigma, w(x)}^\wedge
    \]
    (note that there is a $W$-action on $(B\sigma)^{rs}/B\simeq (T\sigma)^{rs}/T$ and $x \in (B\sigma)^{rs}/B$).
    
    Moreover, under the representation of $\BL_{B\sigma,f}$ in Lemma \ref{formal-B}, the morphism $\BL_{B\sigma, w(x)}^\wedge \to \BL_{G\sigma, p_0(x)}^\wedge$ can be identified with the identity on 
    \[
    \operatorname{Spf}(\mathcal{O}((T\sigma)_{w(x)}^\wedge[u]/(fu))/ T
    \]
    if $w(f)$ divides $L(1, \ad_\fn)^{-1}$, or it can be identified with
    \[
    \operatorname{Spf}(\mathcal{O}(T\sigma)_{w(x)}^\wedge)/ T
    \to \operatorname{Spf}(\mathcal{O}(T\sigma)_{w(x)}^\wedge[u]/(w(f)u))/ T
    \]
    if $w(f)$ divides $L(1, \ad_{\bar{\fn}})^{-1}$. Here $\CO(T\sigma)_x^\wedge$ is an abuse of notation for the completion of $\CO(T\sigma)$ along the inverse image of $x$ in $T$.

    Moreover, in the rank $1$ case, the following global version holds.
    The pullback of $\BL_{G\sigma, f}^\wedge$ along $p_B \colon \BL_{B\sigma} \to \BL_{G\sigma}$ is of the form
    \[
    \coprod_{w\in W^\sigma} \BL_{B\sigma, w(f)}^\wedge.
    \]
    Moreover, the morphism $\BL_{B\sigma, w(f)}^\wedge \to \BL_{G\sigma, f}^\wedge$ can be identified with the identity on 
    \[
    \operatorname{Spf}(\mathcal{O}(T\sigma)_{w(f)}^\wedge[u]/(fu))/ T
    \]
    if $w(f)$ divides $L(1, \ad_\fn)^{-1}$, or it can be identified with
    \[
    \operatorname{Spf}(\mathcal{O}(T\sigma)_{w(f)}^\wedge/ T
    \to \operatorname{Spf}(\mathcal{O}(T\sigma)_{w(f)}^\wedge[u]/(w(f)u))/ T
    \]
    if $w(f)$ divides $L(1, \ad_{\bar{\fn}})^{-1}$.
\end{lemma}

\begin{proof}
    We will consider the general case and the arguments in the rank $1$ case hold in the same way.
    Consider the following commutative diagram
    \[
    \begin{tikzcd}
    \BL_{B\sigma} \arrow[ddr, "p"'] \arrow[dr, "i_B'"] \arrow[drr, "\pi_B", bend left = 30] & & \\
    &\BL_{B\sigma}' \arrow[r, "\pi_B'"] \arrow[d, "p'"] & B\sigma/ B \arrow[d, "p_0"] \\
    &\BL_{G\sigma} \arrow[r, "\pi_G"] & G\sigma/ G
    \end{tikzcd}
    \]
    where $\BL_{B\sigma}'$ is the pullback of the diagram $B\sigma/ B\to G\sigma/ G\leftarrow \BL_{G\sigma}$. We calculate that 
    \[
    \BL_{B\sigma}' = \{(b\sigma, n)\in B\sigma\times \CN \mid \ad(b)(\sigma(n))=qn\}/ B.
    \]
    and thus $\BL_{B\sigma}\to \BL_{B\sigma}'$ is a closed immersion.
    
    We claim that the pullback of $\BL_{G\sigma,p_0(x)}^\wedge$ along $p'$ is isomorphic to $|W^\sigma|$ copies of itself and it is $\coprod_{w\in W^\sigma} \BL_{B\sigma, \pi_B'^{-1}(w(x))}'^\wedge$ explicitly. Indeed, for each $n\in \BN$, since $p_0$ is \'etale, the pullback of the $n$-th jet neighborhood of $p_0(x)$ in $G\sigma/G$ along $p_0$ is the disjoint union of the $n$-th jet neighborhoods of $w(x)$ in $B\sigma/B$ for $w\in W^\sigma$. Pulling back along $\pi_G$, the pullback of the $n$-th jet neighborhood of $\pi_G^{-1}(p_0(x))$ in $\BL_{G\sigma}$ along $p'$ is the disjoint union of the $n$-th jet neighborhoods of $\pi_B'^{-1}(w(x))$ in $\BL_{B\sigma}$. Passing to the limits and we prove the claim.
    
    Then it suffices to determine each $\BL_{B\sigma, \pi_B^{-1}(w(x))}^\wedge \to \BL_{B\sigma, \pi_B'^{-1}(w(x))}'^\wedge$.
    By \ref{L-function}, denoting by $O$ the orbit of roots corresponding to $f$, the formal neighborhood of $\pi_B'^{-1}(w(x))$ lies in the substack
    \[
    \{(b\sigma, n)\in B\sigma\times \fn_{w(O)} \mid \ad(b)(\sigma(n))=qn\}/ B
    \]
    of $\BL_{B\sigma}'$ and its inverse image in $\BL_{B\sigma}$ is in 
    \[
    \{(b\sigma, n)\in B\sigma\times \fn_{w(O)}\cap \fn \mid \ad(b)(\sigma(n))=qn\}/ B.
    \]
    $\fn\cap \fn_{w(O)}$ is either $0$ or $\fn_{w(O)}$, according to whether $w(f)$ divides $L(1,\ad_\fn)^{-1}$.
    Thus the result follows from the calculations in \ref{formal-B}.
\end{proof}

The following corollary is a translation of the previous result in terms of $\Eis$.

\begin{corollary}
    \label{formal-Eis}
    For an irreducible factor $f$ of $L(1,\ad_{\fn\oplus\bar{\fn}})$ and strongly regular semisimple $x \in V(f)$, consider $\Eis \CO_{\BL_{T\sigma, x}^\wedge}$ as an 
    $\CO_{\BL_{G\sigma, p_0(x)}^\wedge}$-module. Then $\BL_{G\sigma, p_0(x)}^\wedge$ can be identified with $\Spf(\mathcal{O}(T\sigma)_x^\wedge[u]/(fu))/T$ and $\Eis \CO_{\BL_{T\sigma, x}^\wedge}$ can be identified with:
    \begin{itemize}
        \item The $\mathcal{O}(T\sigma)_x^\wedge[u]/(fu)$-module $\mathcal{O}(T\sigma)_x^\wedge[u]/(fu)$ with its natural $T$-action (as in Lemma \ref{formal-B}) if $f$ divides $L(1, \ad_\fn)^{-1}$;
        \item Or the $\mathcal{O}(T\sigma)_x^\wedge[u]/(fu)$-module $\mathcal{O}(T\sigma)_x^\wedge$ with natural $T$-action and $u=0$ if not.
    \end{itemize}

    Moreover, in the rank $1$ case, the following global version holds.
    Consider $\Eis \CO_{\BL_{T\sigma, f}^\wedge}$ as an 
    $\CO_{\BL_{G\sigma, f}^\wedge}$-module. Then $\BL_{G\sigma, f}^\wedge$ can be identified with $\Spf(\mathcal{O}(T\sigma)_f^\wedge[u]/(fu))/T$ and $\Eis \CO_{\BL_{T\sigma, f}^\wedge}$ can be identified with:
    \begin{itemize}
        \item The $\mathcal{O}(T\sigma)_f^\wedge[u]/(fu)$-module $\mathcal{O}(T\sigma)_f^\wedge[u]/(fu)$ with natural $T$-action if $f$ divides $L(1, \ad_\fn)^{-1}$;
        \item Or the $\mathcal{O}(T\sigma)_f^\wedge[u]/(fu)$-module-module $\mathcal{O}(T\sigma)_f^\wedge$ with natural $T$-action and $u=0$ if not.
    \end{itemize}
\end{corollary}

Next we construct an inclusion $\Hom(\Eis_B, \Eis_{B'}) \to \Hom(\eis_{B}, \eis_{B'})$. For this purpose, we construct natural transformations $t_l \colon \pi_G^*\eis \to \Eis$ and $t_r \colon \pi_{G*}\Eis\to \eis$ as follows.

Consider the following commutative diagram:
\begin{equation}
\label{commutative}
    \begin{tikzcd}
    \BL_{B\sigma} \arrow[ddr, "p"'] \arrow[dr, "i_B'"] \arrow[drr, "\pi_B", bend left = 60] & & \\
    &\BL_{B\sigma}' \arrow[r, "\pi_B'"] \arrow[d, "p'"] & B\sigma/ B \arrow[d, "p_0"]\arrow[ull, "i_B"', bend right = 30] \\
    &\BL_{G\sigma} \arrow[r, "\pi_G", shift left = 1ex] & G\sigma/ G
    \ar[l, "i_G", shift left = 1ex]
    \end{tikzcd}
\end{equation}
where $\BL_{B\sigma}'$ is the pullback of the diagram $B\sigma/ B\to G\sigma/ G\leftarrow \BL_{G\sigma}$.

For a sheaf $\mathcal{F}$ in $D\Coh(\BL_{T\sigma})_\chi$, abusing the notation, denote also by $\CF$ its pullback via $B\sigma/B$.

We firstly construct $t_l$.
By proper base change, we have $\pi_G^*\eis_B\CF\simeq p'_*\pi_B'^*\CF$. Moreover, we have $\Eis_B\CF\simeq p_*\pi_B^*\CF\simeq p'_*i_{B*}'i_B'^*\pi_B'^*\CF$. Then we define
\[
t_l\colon \pi_G^*\eis_B \simeq p'_*\pi_B'^*\to p'_*i_{B*}'i_B'^*\pi_B'^*\simeq \Eis_B
\]
by precomposing the adjunction $\mathrm{id}\to i_{B*}'i_B'^*$ by $\pi_B'^*$ and then postcomposing $p'_*$.
We remark that the adjunction of $t_l$, $\eis\to \pi_*\Eis_B$ is given by
\[
\eis_B = p_{0*} \xrightarrow{\text{adj}} p_{0*}\pi_{B*}'\pi_B'^* 
\to p_{0*}\pi_{B*}'i_{B*}'i_B'^*\pi_B'^* \simeq p_{0*}\pi_{B*}\pi_B^*
\simeq \pi_{G*}\Eis_B.
\]

Then we construct $t_r$. In the commutative diagram above, we have $\pi_{G*}\Eis\simeq p_{0*}\pi_{B*}\pi_B^*$. As $\pi_{B*}\circ i_{B*}=\mathrm{id}$, $\eis_B\simeq p_{0*}\pi_{B*}i_{B*}$.
Moreover, there is an adjunction $\pi_{B}^*\to i_{B*}$ of $1\xrightarrow{\sim} \pi_{B*}\circ i_{B*}$.
Then we define
\[
t_r\colon \pi_{G*}\Eis_B\to \eis_B
\]
by applying $p_{0*}\pi_{B*}$ to the adjunction $\pi_{B}^*\to i_{B*}$.

We remark that $t_r\circ \operatorname{adj}(t_l)$ is the identity on $\eis_B$.

\begin{definition}[Restriction]
\label{restriction}
    For two sheaves $\CF$ and $\CG$ in $D\Coh(\BL_{T\sigma})_\chi$, we define a map
    \[
    \res\colon \Hom(\Eis_B\CF, \Eis_{B'}\CG) \to  \Hom(\eis_{B}\CF, \eis_{B'}\CG)
    \]
    by
    \[
    \Hom(\Eis_B\CF, \Eis_{B'}\CG)\xrightarrow{t_l(\CF)^*}
    \Hom(\pi_G^*\eis_B\CF, \Eis_{B'}\CG)
    \xrightarrow[\sim]{\text{adj}}\Hom(\eis_B\CF, \pi_{G*}\Eis_{B'}\CG)
    \xrightarrow{t_r(\CG)_*}\Hom(\eis_B\CF, \eis_{B'}\CG),
    \]
    where we denote adj for the adjunction.
\end{definition}

As the projection $\BL_{G\sigma}^g \to (G\sigma)^g/G$ is an isomorphism 
(Lemma \ref{generic locus}), the following statement is immediate.

\begin{lemma}
    \label{transformation}
    After postcomposing $j^{g*}_G$, where $j^g_G$ is the open immersion $\BL_{G\sigma}^g\to \BL_{G\sigma}$ or $(G\sigma)^g/G\to G\sigma/G$, $t_l$ and $t^r$ become isomorphisms.
\end{lemma}

Note that $\Hom(\Eis_B\CF, \Eis_{B'}\CG)$ and $\Hom(\eis_B\CF, \eis_{B'}\CG)$ are $\CO(\BL_{T\sigma})=\CO(T\sigma/T)$-bimodules:
$\CO(\BL_{T\sigma})=\CZ(D\Coh(\BL_{T\sigma})_1)$ acts on $\CF$ and $\CG$ and then acts on $\Eis_B\CF$ and $\Eis_{B'}\CG$. 
There is also a $\CO(G\sigma/G)=\CO(\BL_{G\sigma})$-module structure on the $\Hom$'s, as the base space is $\BL_{G\sigma}$ or $G\sigma/G$ and we have the following compatibility.

\begin{lemma}
    Both the left and the right $\CO(T\sigma/T)$-module structure of $\Hom(\eis_B\CF,\eis_{B'}\CG)$ and $\Hom(\Eis_B\CF, \Eis_{B'}\CG)$ are compatible with the $\CO(G\sigma/G)$-structure, via the twisted Chevalley map $\chi^*\colon \CO(G\sigma/G)\to \CO(T\sigma/T)$.
\end{lemma}

\begin{proof}
    We consider the case for $\eis$, the case for $\Eis$ is similar as the corresponding coarse moduli spaces are isomorphic.
    For $f\in \CO(G\sigma/G)$ and $\CF\in D\Coh(T\sigma/T)_1$, by Corollary \ref{factor Chevalley}, the Chevalley map is equal to $\CO(G\sigma/G)\xrightarrow{p_0^*} \CO(B\sigma/B)\xleftarrow[\sim]{q_0^*}\CO(T\sigma/T)$, where $q_0^*$ is the inverse of $i_0^*\colon \CO(B\sigma/B)\to \CO(T\sigma/T)$ induced by the inclusion $T\to B$. 
    
    By definition, the action of $f$ on $\eis_B\CF$ is identified with the action obtained by applying the functor $p_{0*}$ to the action of $p_0^*(f)$ on $q_0^*\CF$. Since $p_0^*(f)=q_0^*i_0^*p_0^*(f)=q_0^*\chi^*(f)$, the action of $\chi^*(f)\in \CO(T\sigma/T)$ on $\eis_B\CF$ and the action of $f\in \CO(G\sigma/G)$ coincide. Thus the $\CO(G\sigma/G)$-structure is compatible with the $\CO(T\sigma/T)$-module structure induced by $\eis_B$ and similarly for $\eis_{B'}$. 
\end{proof}

\begin{lemma}
\label{ring homo}
    The map $\res$ is an $\CO(\BL_{T\sigma})=\CO(T\sigma/T)$-bimodule homomorphism. Moreover, the construction $\res$ preserves the monoidal structure in the following sense:
    for an other Borel $B''$ stable under $\sigma$ and containing $T$, and an other sheaf $\CH$ over $T\sigma\sslash T$ and homomorphisms $f\in \Hom(\Eis_B\CF, \Eis_{B'}\CG)$
    and $g \in \Hom(\Eis_{B'}\CG, \Eis_{B''}\CH)$, we have $\res(g\circ f)=\res(g)\circ \res(f)$.
\end{lemma}

\begin{proof}
    The $\CO(T\sigma/T)$-linearity follows from the functoriality of $\res$. Then it suffices to prove the monoidality.
    Replacing $\CF$, $\CG$ and $\CH$ by free resolutions, it suffices to prove the case when $\CF$, $\CG$ and $\CH$ are free $\CO_{T\sigma/T}$-modules.
    In this case, $\eis_B\CF$ and $\eis_{B'}\CG$ are locally free $\CO_{G\sigma/G}$-module by Lemma \ref{locallyfreeeis}. Hence, the homomorphism
    \[
    \Hom(\eis_{B}\CF, \eis_{B'}\CG)\to \Hom(\eis_{B}^g\CF, \eis_{B'}^g\CG)
    \]
    is injective as the generic locus is open and dense (Recall that $\eis_{B}^g\CF$ is an abuse of notation for $\eis_B^g j^{g*}_T\CF\simeq j_G^{g*}\eis_B\CF$).
    Thus it suffices to prove that $\res$ is monoidal after postcomposing the pullback $j^{g*}_G$.
    
    In fact, by definition of $\res$ and Lemma \ref{transformation}, the composition
    \[
    \Hom(\Eis_B\CF, \Eis_{B'}\CG) \to  \Hom(\eis_{B}\CF, \eis_{B'}\CG)\to 
    \Hom(\eis_{B}^g\CF, \eis_{B'}^g\CG)=\Hom(\Eis_{B}^g\CF, \Eis_{B'}^g\CG)
    \]
    is just the restriction to the generic locus and thus it is monoidal. Thus $\res$ is monoidal.  
\end{proof}

In particular, the construction defines an $\CO(\BL_{T\sigma})=\CO(T\sigma/T)$-module homomorphism $\Hom(\Eis_B, \Eis_{B'}) \to \Hom(\eis_{B}, \eis_{B'})$ and a ring homomorphism $\End(\Eis_B)\to \End(\eis_B)$.

Next we prove that $\res$ is injective.

\begin{lemma}
    \label{inclusion}
    For $\chi \in X^*(T^\sigma)^{W^\sigma}$ and two locally free sheaves $\CF$ and $\CG$ in $\Coh(\BL_{T\sigma})_\chi$,
    the modules
    \[R\Hom(\Eis_B\CF, \Eis_{B'}\CG)\ \text{and}\  R\Hom(\eis_{B}\CF, \eis_{B'}\CG)\] 
    are concentrated in degree $0$.
    Moreover, the homomorphism 
    \[\res\colon \Hom(\Eis_B\CF, \Eis_{B'}\CG) \to  \Hom(\eis_{B}\CF, \eis_{B'}\CG)\] 
    is injective,
    and thus induces an inclusion $\Hom(\Eis_B, \Eis_{B'}) \to \Hom(\eis_B, \eis_{B'})$ for the functor version.
\end{lemma}

\begin{proof}
    We firstly prove the case that $\chi = 1$.
    By the unipotent Langlands correspondence proved by Zhu (\cite{zhu2025tamecategoricallocallanglands}, Theorem 1.2), there is a fully faithful embedding (in the $\infty$-categorical sense)
    \[
    D\mathsf{Rep}(G^\vee(F), I) \to D\Coh(\BL_{G\sigma})
    \]
    for an unramified $p$-adic reductive group $G^\vee(F)$ with Langlands dual $(G,\sigma)$ (in the case that $\sigma=1$, it is already proved in \cite{ben2024coherent}, Theorem 1.9). 
    It suffices to consider the case that $\CF$ and $\CG$ are free $\CO_{T_\sigma}$-modules only.
    Under the correspondence,
    $\CF$ corresponds to a free module $M_\CF$ over $k[T^\vee/{T^\vee}^0]$ and $\Eis_B \CF$ corresponds to $i_{B^\vee}^{G^\vee} M_\CF$ (\cite{zhu2025tamecategoricallocallanglands}, Theorem 1.6(4), where $\Eis$ is denoted by $\mathrm{CohSpr}$). Similarly for $\CG$.

    Thus we have 
    \[
    R\Hom(\Eis_B\CF, \Eis_{B'}\CG) \simeq R\Hom(i_{B^\vee}^{G^\vee} M_\CF, i_{{B'}^\vee}^{G^\vee} M_\CG).
    \]
    As $i_{B^\vee}^{G^\vee} M_\CF$ is projective in $\mathsf{Rep}(G^\vee(F))$, the latter $R\Hom$ is concentrated in degree $0$ 
    and thus the previous one is. For the module $R\Hom(\eis_B\CF, \eis_{B'}\CG)$, 
    it is concentrated in degree $0$ as $\eis_B\CF$ is locally free by Lemma \ref{locallyfreeeis}.

    Then we prove the argument about the inclusion.
    For the same reason as in the proof of Lemma \ref{ring homo}, it suffices to prove that 
    \[
    \Hom(\Eis_B\CF, \Eis_{B'}\CG) \to \Hom(\eis_B\CF, \eis_{B'}\CG) \to \Hom(\eis^g_B\CF, \eis^{g}_{B'}\CG) = \Hom(\Eis_Bj^g_{*}j^{g*}\CF, \Eis_{B'}j^g_{*}j^{g*}\CG) 
    \]
    is an inclusion,
    where $j^g$ is the inclusion of the generic locus into the entire space and the last arrow is from $j_*^g\circ\Eis^g\simeq \Eis\circ 
    j^g_*$.
    Under the local Langlands correspondence, $j^g_{*}j^{g*}\CF$ corresponds to the localization $M_\CF[L(1, \ad_{\fn})L(1, \ad_{\bar{\fn}})]$ as $j^g_{*}j^{g*}\CF$ is the colimit $\CF\xrightarrow{L(1,\ad_{\fn\oplus\bar{\fn}})^{-1}}\CF\to \cdots$
    and similarly for $M_\CF[L(1, \ad_{\fn})L(1, \ad_{\bar{\fn}})]$.

    Thus
    \[
    \Hom(\Eis_Bj^g_{*}j^{g*}\CF, \Eis_{B'}j^g_{*}j^{g*}\CG)  \simeq \Hom(i_{B^\vee}^{G^\vee} M_\CF[f], i_{{B'}^\vee}^{G^\vee} M_\CG[f])
    \]
    where $f = L(1,\ad_\fn)L(1,\ad_{\bar{\fn}})$. Thus the homomorphism is injective as $M_\CF \to M_\CF[f]$ and $M_\CG\to M_\CG[f]$ is injective.

    Finally, as $D\Coh(\BL_{T\sigma})_1$ is generated by $\CO_{T_\sigma}$,
    we have $\Hom(\Eis_B, \Eis_{B'}) \subset \Hom(\Eis_B\CO_{T_\sigma}, \Eis_{B'}\CO_{T_\sigma})$ and $\Hom(\eis_B, \eis_{B'}) \subset \Hom(\eis_B\CO_{T_\sigma}, \eis_{B'}\CO_{T_\sigma})$. Thus $\Hom(\Eis_B, \Eis_{B'})\to \Hom(\eis_B, \eis_{B'})$ is an inclusion.

    In the case for general $\chi$, take a simply connected cover $G_0\to G$. It induces a morphism $\alpha_G\colon \BL_{G_0\sigma}\to \tilde{\BL}_{G\sigma}/G_0\to \BL_{G\sigma}$, which is a composition of a Galois covering and a gerbe. By descent, it suffices to prove the statements for $\alpha_G^*\Eis_B\CF \simeq \Eis_{B_0}\alpha_T^*\CF$ and $\Eis_{B_0}\alpha_T^*\CG$ (where $\alpha_T\colon\BL_{T_0\sigma}\to \BL_{T\sigma}$. Thus we reduce to the case for $G_0$.
    Again, the constructions in the statement are compatible with direct products, the statement for a torus is trivial and the statement for a semisimple and simply connected group follows from the $\chi=1$ case. Thus we conclude.
\end{proof}

\begin{remark}
    There may be a simpler proof of the injective argument. By definition of the restriction homomorphism \ref{restriction}, $\res$ is the composition
    \[
    \Hom(\Eis_B\CF, \Eis_{B'}\CG)\xrightarrow{t_l(\CF)^*}
    \Hom(\pi_G^*\eis_B\CF, \Eis_{B'}\CG)
    \xrightarrow[\sim]{\text{adj}}\Hom(\eis_B\CF, \pi_{G*}\Eis_{B'}\CG)
    \xrightarrow{t_r(\CG)_*}\Hom(\eis_B\CF, \eis_{B'}\CG).
    \]
    Thus it suffices to prove that $t_l(\CF)^*$ is injective and $t_r(\CG)_*$ is an isomorphism. Indeed, we can prove that $t_r(\CG)_*$ is an isomorphism and $t_l(\CF)^*$ is injective, supposing that $t_l(\CF)$ is a surjection of honest sheaves if $\CF$ is locally free. However, we do not know how to prove this statement in general.
\end{remark}

The following lemma about the $\Hom$-modules follows from the same argument as in the proof of Lemma \ref{tensor_by_chi}. 

\begin{lemma}
\label{tensor_by_chi_Eis}
    For $\chi \in X^*(T^\sigma)^{W^\sigma}$ and two locally free sheaves 
    $\CF, \CG \in \Coh(\BL_{T\sigma})_1$, the isomorphism in Lemma \ref{tensor_by_chi}
    \[
    \otimes \overline{\CO_\chi} \colon \Hom(\eis_B\CF, \eis_B\CG) \to \Hom(\eis_B\CF_\chi, \eis_B\CG_\chi)
    \]
    restricts to an isomorphism (under the inclusion in Lemma \ref{inclusion})
    \[
    \otimes \overline{\CO_\chi} \colon \Hom(\Eis_B\CF, \Eis_B\CG) \to \Hom(\Eis_B\CF_\chi, \Eis_B\CG_\chi).
    \]
    Moreover, this isomorphism extends to all $\CF, \CG \in D\Coh(\BL_{T\sigma})_1$.
\end{lemma}

We turn to the final proof of Proposition \ref{all}.

\begin{proof}
    By lemma \ref{inclusion} above, we have that
    \[
    \Hom(\Eis_B, \Eis_{B'}) \subset \Hom(\eis_B, \eis_{B'}) \simeq \CO(\BL_{T\sigma})
    \]
    where the isomorphism can be inferred from the proof of \ref{g}, making $\Hom(\Eis_B, \Eis_{B'})$ an ideal of $\CO(\BL_{T\sigma})$. As in the proof of \ref{generic locus}, the operator $J^g_{B'|B}$ extends to an isomorphism $\eis_B\simeq \eis_{B'}$ and corresponds to $1\in \CO(\BL_{T\sigma})$,  it suffices to find an element $m \in \CO(\BL_{T\sigma})$ such that $m J^g_{B'|B}$ induces a homomorphism $\Eis_B \to \Eis_{B'}$ and such that any other such multiplier is divisible by $m$.

    In the beginning, we consider the case for $\chi=1$.
    We firstly prove that $m$ is divisible by $L(1, \ad_{\fn_{B'}/\fn_{B}})^{-1}$. For any irreducible factor $f$ of $L(1, \ad_{\fn_{B'}})^{-1}$ that does not divide $L(1, \ad_{\fn_B})^{-1}$ and strongly regular semisimple $x\in V(f)$,
    consider
    $\Eis_B \CO_{\BL_{T\sigma, x}^\wedge}$ and $\Eis_{B'} \CO_{\BL_{T\sigma, x}^\wedge}$ as
    $\CO_{\BL_{G\sigma, p_0(x)}^\wedge}$-modules. By Corollary \ref{formal-Eis}, $\BL_{G\sigma, p_0(x)}^\wedge$ can be identified with $\Spf(\mathcal{O}(T\sigma)_x^\wedge[u]/(fu))/T$, $\Eis_{B'} \CO_{\BL_{T\sigma, x}^\wedge}$ can be identified with the $\mathcal{O}(T\sigma)_x^\wedge[u]/(fu)$-module $\mathcal{O}(T\sigma)_x^\wedge[u]/(fu)$ with natural $T$-action and $\Eis_{B} \CO_{\BL_{T\sigma, x}^\wedge}$ can be identified with the $\mathcal{O}(T\sigma)_x^\wedge[u]/(fu)$-module $\mathcal{O}(T\sigma)_x^\wedge$ with natural $T$-action and $u=0$.
    
    Thus $mJ^g_{B'|B}$ defines an $\CO_{\BL_{G\sigma, p_0(x)}^\wedge}$-module homomorphism $\Eis_B \CO_{\BL_{T\sigma, x}^\wedge} \to \Eis_{B'} \CO_{\BL_{T\sigma, x}^\wedge}$ only if $m$ is divisible by $f$ in the formal neighborhood of $x$, as $u=0$ in $\mathcal{O}(T\sigma)_x^\wedge$ is sent to $mu=0$ in $\mathcal{O}(T\sigma)_x^\wedge[u]/(fu))$. Varying $x$ in the open and dense subset of strongly regular semisimple elements in $V(f)$, we see that $m=0$ on $V(f)$ and thus $m$ is divisible by $f$ as $f$ is irreducible and $\CO(T_\sigma)$ is factorial. 
    Multiplying all those $f$, we have that $m$ is divisible by $L(1, \ad_{\fn_{B'}/\fn_{B}})^{-1}$ as $L(1, \ad_{\fn_{B'}/\fn_{B}})^{-1}$ has no multiple factors by \ref{L-function}.

    It suffices to show that $L(1, \ad_{\fn_{B'}/\fn_{B}})^{-1}J^g_{B'|B}$ defines a transformation $\Eis_B \to \Eis_{B'}$. 
    
    Let $B = B_1, B_2, \dots, B_n = B'$ be a chain of Borel subgroups in $G$ fixed by $\sigma$ such that for any $i$, there exists a simple reflection $s_i \in W^\sigma$ such that $B_{i+1} = s_i B_i s_i$ and that $n$ is the minimal length of the chain (such chain exists as Borel subgroups of $G$ fixed by $\sigma$ correspond to Weyl chamber of the root system $\{X^*(T_\sigma),\{\alpha_O\}\}$). For $i$, denote by $P_i$ the minimal parabolic subgroup containing $B_i$ and $B_{i+1}$ and denote by $M_i$ its Levi subgroup. By \cite{Hellmann_2023}, Corollary 2.11 (for the case $\sigma=1$, and the general case is similar), for $j = i$ or $j = i+1$ we have
    \[
    \Eis_{B_j}^G = \Eis_{P_i}^G \circ \Eis_{B_j \cap M_i}^{M_i}
    \]
    Moreover, we have that the $L$ function is multiplicative:
    \[
    L(1, \ad_{\fn_{B'}/\fn_{B}}) = \prod_i L(1, \ad_{\fn_{B_{i+1}}/\fn_{B_i}})
    = \prod_i L(1, \ad_{\fn_{B_{i+1}\cap M_i}/\fn_{B_i \cap M_i}})
    \]
    It suffices to prove that
    \[
    L(1, \ad_{\fn_{B_{i+1}\cap M_i}/\fn_{B_i \cap M_i}})^{-1} J^g_{B_{i+1}\cap M_i|B_i\cap M_i}
    \colon \Eis_{B_i \cap M_i} \to \Eis_{B_{i+1} \cap M_i}
    \]
    is well-defined. Note that $M_i$ satisfies the condition in Lemma \ref{ord2}. Thus for any irreducible factor $f$ of $L(1, \ad_{\fn_{M_i}\oplus \bar{\fn}_{M_i}})^{-1}$, $V(f)\subset \BL_{T\sigma, f}$ is exactly the vanishing locus of $f$ in $\BL_{T\sigma}$. Moreover, different $V(f)$'s are disjoint and the non-generic locus (relative to $M_i$) of $\BL_{T\sigma}$ is the union of those $V(f)$'s.
    Thus by Beauville--Laszlo gluing, it suffices to prove that $L(1, \ad_{\fn_{B_{i+1}\cap M_i}/\fn_{B_i \cap M_i}})^{-1} J^g_{B_{i+1}\cap M_i|B_i\cap M_i}$ is well-defined over the formal neighborhood of $V(f)$ for all such $f$.
    That is, for $V(f)\subset \BL_T - \BL_T^g$, $L(1, \ad_{\fn_{B_{i+1}\cap M_i}/\fn_{B_i \cap M_i}})^{-1} J^g_{B_{i+1}\cap M_i|B_i\cap M_i}(\CO(T\sigma)^\wedge_f)$ is well-defined.

    By Corollary \ref{formal-Eis}, the homomorphism $L(1, \ad_{\fn_{B_{i+1}\cap M_i}/\fn_{B_i \cap M_i}})^{-1} J^g_{B_{i+1}\cap M_i|B_i\cap M_i}(\CO(T\sigma)^\wedge_f)$
    is of the form
    \[
    \text{multiplying by}\ L(1, \ad_{\fn_{B_{i+1}\cap M_i}/\fn_{B_i \cap M_i}})^{-1} \colon N \to N'
    \]
    where the $\CO_{\BL_{G\sigma, f}^\wedge}$-modules $N$ and $N'$ are of the following forms: $N$ can be identified with the module
    $\mathcal{O}(T\sigma)_f^\wedge[u]/(fu)$ or $\mathcal{O}(T\sigma)_f^\wedge$ with natural $T$-action,
    according to whether $f$ divides $L(1, \ad_{\fn_{B_i\cap M_i}})^{-1}$ and similarly for $N'$, according to whether $f$ divides $L(1, \ad_{\fn_{B_{i+1}\cap M_i}})^{-1}$. In the case that $f$ divides $L(1, \ad_{\fn_{B_{i+1}\cap M_i}})^{-1}$ but does not divide $L(1, \ad_{\fn_{B_i\cap M_i}})^{-1}$, the module homomorphism $L(1, \ad_{\fn_{B_{i+1}\cap M_i}/\fn_{B_i \cap M_i}})^{-1} J^g_{B_{i+1}\cap M_i|B_i\cap M_i}(\CO(T\sigma)^\wedge_f)$ is well-defined, being identified as a multiple of $\cdot f\colon \mathcal{O}(T\sigma)_f^\wedge\to \mathcal{O}(T\sigma)_f^\wedge[u]/(fu)$. For other possibilities of $f$, multiplication by $1$ suffices to obtain a well-defined homomorphism. Thus multiplying by $L(1, \ad_{\fn_{B_{i+1}\cap M_i}/\fn_{B_i \cap M_i}})^{-1}$ is always a well-defined homomorphism.

    Finally, we prove the case for general $\chi$. We have already an intertwining operator $j_{B'|B}\in \Hom(\Eis_B, \Eis_{B'})_1$. We claim that
    \[
    j_{B'|B}\otimes \overline{\CO_\chi} \in \Hom(\Eis_B, \Eis_{B'})_\chi
    \]
    is the desired intertwining operator. Here the subscripts means the $\Hom$-module over the corresponding subcategory of $D\Coh(\BL_{T\sigma})$.
    
    Recall that in the proof of Proposition \ref{g} (in the end of section 2.3), the intertwining operator in $\Hom(\Eis^g_B, \Eis^g_{B'})_\chi$ is defined as $J^g_{B'|B}\otimes \overline{\CO_\chi}$, where $J^g_{B'|B}\in \Hom(\Eis^g_B, \Eis^g_{B'})_1$. Thus we have
    \[
    j_{B'|B}\otimes \overline{\CO_\chi} = L(1, \ad_{\fn_{B'}/\fn_B})^{-1}(J_{B'|B}^g\otimes \overline{\CO_\chi})
    \]
    and we prove the statement about the multiplier. The statement about the generator follows from the $\chi=1$ case and Lemma \ref{tensor_by_chi_Eis}.
\end{proof}

Thus we proved Theorem \ref{intertwining} for the in the unipotent case.

\subsection{Constant term and the unnormalized intertwining operator}

We discuss the left adjoint functor of $\Eis = p_*q^*$ and study its properties in this section. We use the same notations as in the previous sections.

We will denote by $\CT_B$ the putative left adjoint of $\Eis_B$ and consider the composition $\CT_B\circ \Eis_B$. By definition, 
\[
\Hom(\CT_B\circ\Eis_B\CO_{T_\sigma}, \CO_{T_\sigma})= 
\End(\Eis_B\CO_{T_\sigma})\subset \End(\eis_B\CO_{T_\sigma})
\]
where the inclusion is induced by $\res$, as in Lemma \ref{inclusion}.
This construction motivates us to consider the endomorphism ring $\End(\eis_B\CO_T)$.

In the next lemmas we sometimes consider simply connected reductive groups, by which we mean products of a simply connected semisimple group and a torus. Moreover, when we consider a simply connected cover $G_0\to G$, we always require that the automorphism $\sigma$ of $G$ can be lifted to an automorphism $\sigma$ of $G_0$. 

\begin{lemma}
\label{Endeis}
    Fix a covering $G_0\to G$ as in the diagram \ref{kostantdiagram}. We have morphisms $p_T\colon T_\sigma=T\sigma\sslash T\to T\sigma\sslash' N_G(T\sigma)$ and 
    $\pi_T\colon G\sigma/G_0\to T\sigma\sslash'N_G(T\sigma)$.
    The endomorphism ring $\End(\eis_B\CO_{T_\sigma})$ is isomorphic to the endomorphism ring
    \[
    \End_{T\sigma\sslash'N_G(T\sigma)}(p_{T*}\CO_{T_\sigma}).
    \]
    Moreover, consider the morphism $\eta\colon T\sigma\sslash' N_G(T\sigma)\to T_\sigma\sslash W^\sigma$ from an algebraic stack to its coarse moduli space.
    Then the homomorphism
    \[
    \eta_* \colon 
    \End_{T\sigma\sslash'N_G(T\sigma)}(p_{T*}\CO_{T_\sigma}) \to 
    \End_{T_\sigma\sslash W^\sigma}(\eta_*p_{T*} \CO_{T_\sigma})=
    \End_{\CO(T_\sigma)^{W^\sigma}}(\CO(T_\sigma))
    \]
    is injective and when $G$ is simply connected, it is an isomorphism.
\end{lemma}

\begin{proof}
    We firstly construct an isomorphism $\End(\eis_B\CO_{T_\sigma})\simeq \End_{T\sigma\sslash'N_G(T\sigma)}(p_{T*}\CO_{T_\sigma})$.
    Recall that we have a gerbe $\beta\colon G\sigma/G_0\to G\sigma/G$ and we have $\End(\eis_B\CO_{T_\sigma})\simeq \End(\beta^*\eis_B\CO_{T_\sigma})$.
    By Lemma \ref{locallyfreeeis}, $\beta^*\eis_B\CO_{T_\sigma}$ is isomorphic to $\pi_T^*p_{T*}\CO_{T_\sigma}$.
    We claim that $\pi_T^*$ induces an isomorphism $\End(p_{T*}\CO_{T_\sigma})\to \End(\beta^*\eis_B\CO_{T_\sigma})$.
    It suffices to prove that the unit map
    \[
    p_{T*}\CO_{T_\sigma}\xrightarrow{\sim}\pi_{T*}\pi_T^*p_{T*}\CO_{T_\sigma}
    \]
    is an isomorphism. 
    
    When $G=G_0$ is simply connected, $T\sigma\sslash'N_G(T\sigma)\simeq T\sigma\sslash N_G(T\sigma)$ is a scheme. Moreover, by Chevalley isomorphism \ref{Chevalley}, $\pi_T$ is a morphism from an algebraic stack to its coarse moduli space. Thus in this case the unit map $\mathrm{id}\to \pi_{T*}\pi_T^*$ is an isomorphism and the lemma holds.

    In general, as in the diagram in Proposition \ref{kostant}, $\pi_T$ is decomposed into 
    \[
    G\sigma/G_0\xrightarrow{\pi_T'} (T_0\sigma\sslash N_{G_0}(T_0\sigma))/\Gal
    \xrightarrow{\pi_T''} T\sigma\sslash' N_G(T\sigma).
    \]
    By pulling back via the Galois covering $T_0\sigma\sslash N_{G_0}(T_0\sigma) \to T_0\sigma\sslash N_{G_0}(T_0\sigma)/\Gal$ and the result in the $G=G_0$ case, we have the unit map $\mathrm{id}\to \pi'_{T*}\pi'^*_T$ is an isomorphism. Moreover, as $\pi_{T}''$ is a $\ker(\Gal\to T_{0,\sigma})$-gerbe, the unit map $\mathrm{id}\to \pi''_{T*}\pi''^*_T$ is an isomorphism.
    Thus the unit map $\mathrm{id}\to \pi_{T*}\pi^*_T$ is an isomorphism.

    Then we prove that $\eta_*$ is injective. Over the open subset $T^{s.rs}_\sigma$ consisting of strongly regular semisimple elements of $T_\sigma$,
    $W^\sigma$ acts freely and thus $T^{s.rs}_\sigma\sslash W^\sigma = T^{s.rs}_\sigma/W^\sigma$. Moreover, the group $\ker(T_{0,\sigma}\to T_\sigma)\times W^\sigma$
    acts freely on the inverse image of $T^{s,rs}_\sigma$ in $T_{0,\sigma}$, denoted $T_{0,\sigma}^{s.rs}$. Thus 
    \[
    (T\sigma)^{s.rs}\sslash' N_G(T\sigma)  \simeq T_{0,\sigma}^{s.rs}/(\ker(T_{0,\sigma}\to T_\sigma)\times W^\sigma) \simeq 
    T_\sigma^{s.rs}/W^\sigma \simeq T_\sigma^{s.rs}\sslash W^\sigma
    \]
    and $\eta$ is an isomorphism over $T_\sigma^{s.rs}$. As $T^{s.rs}_\sigma$ is open and dense in $T_\sigma$ and $p
    _{T*}\CO_{T_\sigma}$ is torsion-free, denoting by $j^{s.rs}\colon T^{s.rs}_\sigma\subset T_\sigma$ the open immersion,
    \[
    j^{s.rs*}\eta_* \colon \End_{T\sigma\sslash'N_G(T\sigma)}(p_{T*}\CO_{T_\sigma})
    \to \End_{T_\sigma^{s.rs}\sslash W^\sigma}(j^{s.rs*}\eta_*p_{T*} \CO_{T_\sigma})
    \]
    is injective. Thus $\eta_*$ is injective. Moreover, when $G=G_0$ is simply connected, $T\sigma\sslash' N_G(T\sigma)= T_\sigma\sslash W^\sigma$ and thus $\eta_*$ is an isomorphism.
\end{proof}

We will describe the endomorphism ring of $\CO(T_\sigma)$ as an $\CO(T_\sigma)^{W^\sigma}$-module
via the following notion of (multiplicative) Demazure operators.

\begin{definition}[Demazure operators]
    For a simple reflection $s_\alpha$ in $W^\sigma$, define the Demazure operator $\delta_{s_\alpha}$ as
    \[
    \delta_{s_\alpha} \colon \CO(T_\sigma) \to \CO(T_\sigma), \quad
    f \mapsto \frac{f - s_\alpha(f)}{1 - e^{\tilde{\alpha}}},
    \]
    where $\tilde{\alpha}$ is defined in the following process:
    \begin{itemize}
        \item Suppose that $G$ is simply connected. Then $\tilde{\alpha}$ is defined to be
        $2\alpha$ or $\alpha$ according to whether $\alpha^\vee \in X_*(T_\sigma)$ is divisible by $2$ or not.
        \item For general $G$, take a simply connected cover $G_0\to G$. Then the set of roots $R_{G_0}$ of $G_0$ and $R_G$ of $G$ can be naturally identified and similarly for the $\sigma$-orbits $R_{G_0,\sigma}$ and $R_{G, \sigma}$ in the root system. Denote by $\pi\colon R_{G_0,\sigma}\simto R_{G,\sigma}$ the identification. Then $\tilde{\alpha}$ is defined to be $\pi(\widetilde{\pi^{-1}(\alpha)})$. We remark that this construction is independent of $G_0$, 
        as it depends only on the semisimple group part of $G_0$.
    \end{itemize}
    For general $w \in W^\sigma$, let $w = s_1\dots s_n$ be a reduced expression of $w$, define $\delta_w := \delta_{s_1}\dots \delta_{s_n}$ (by the same proof as for the additive case in Théorème 1 in \cite{demazure}, it is independent of the choice of $s_1\dots s_n$). 
\end{definition}
The definition is a multiplicative analog of the results in \cite{demazure}, where $T$ is replaced by its Lie algebra $\ft$ and 
$\delta_s$ is defined by $D_s := \frac{1-s}{\alpha_s}$ and the operators is defined as $D_w := D_{s_1}\dots D_{s_n}$. 

By definition, $\delta_w$ is $\CO(T_\sigma)^{W^\sigma}$-linear and can be viewed as an element in $\End_{\CO(T_\sigma)^{W_\sigma}}(\CO(T_\sigma))$. We will see that it lies in
$\End(\eis_B\CO_{T_\sigma})$.

Recall that for a simply connected cover $G\to G_0$, in the proof of Proposition \ref{kostant}, we have the following cartesian diagram
\[
    \begin{tikzcd}
	T_{0,\sigma} & T_\sigma \\
	T_0\sigma\sslash N_{G_0}(T_0\sigma) & T\sigma\sslash' N_G(T\sigma),
	\arrow["\pi'", from=1-1, to=1-2]
	\arrow["p_{T_0}"', from=1-1, to=2-1]
	\arrow["p_T", from=1-2, to=2-2]
	\arrow["\pi", from=2-1, to=2-2]
    \end{tikzcd}
\]
where the horizontal arrows are Galois coverings with Galois group $\Gamma'=\ker(T_{0,\sigma}\to T_\sigma)$. Thus we have $p_{T_0*}\CO_{T_{0,\sigma}}\simeq \pi^*p_{T*}\CO_{T_\sigma}$ and we have a $\Gamma'$-equivariant isomorphism
\[
    \End_{T\sigma\sslash'N_G(T\sigma)}(p_{T*}\CO_{T_\sigma})\otimes_{\CO(T_\sigma)} \CO(T_{0,\sigma}) \simeq \End_{T_0\sigma\sslash N_{G_0}(T_0\sigma)}(p_{T_0*}\CO_{T_{0,\sigma}}) = \End_{\CO(T_{0,\sigma})^{W^\sigma}}(\CO(T_{0,\sigma})).
\]

We will show the compatibility of Demazure operators under coverings.

\begin{lemma}
    \label{compatible-Demazure}
    Under the inclusion in Lemma \ref{Endeis},
    \[
    \End(\eis_B\CO_{T_\sigma})\subset \End_{\CO(T_\sigma)^{W^\sigma}}(\CO(T_\sigma)),
    \]
    $\delta_w \in \End_{\CO(T_\sigma)^{W^\sigma}}(\CO(T_\sigma))$ lies in 
    $\End(\eis_B\CO_{T_\sigma})$. 

    Moreover, for a simply connected cover $G_0\to G$, 
    denote by $\delta_{w, T_0}$ the operator for $T_{0,\sigma}$
    and $\delta_{w,T}$ the operator for $T_\sigma$.
    we have the following compatibility between $\delta_{w, T_0}$ and $\delta_{w,T}$:
    Under the identity
    \begin{equation}
    \label{T_0,sig}
    \End_{T\sigma\sslash'N_G(T\sigma)}(p_{T*}\CO_{T_\sigma})\otimes_{\CO(T_\sigma)} \CO(T_{0,\sigma}) \simeq \End_{T_0\sigma\sslash N_{G_0}(T_0\sigma)}(p_{T_0*}\CO_{T_{0,\sigma}}) = \End_{\CO(T_{0,\sigma})^{W^\sigma}}(\CO(T_{0,\sigma}))
    \end{equation}
    above, $\delta_{w,T}\otimes 1$ in the left hand side is equal to 
    $\delta_{w, T_0}$ in the right hand side. 
\end{lemma}

\begin{proof}
    We claim that $\delta_{w,T_0}\in \End_{\CO(T_{0,\sigma})^{W^\sigma}}(\CO(T_{0,\sigma}))$ lies in $\End_{T\sigma\sslash'N_G(T\sigma)}(p_{T*}\CO_{T_\sigma})\otimes 1$ under the identity \ref{T_0,sig} above. Any element $\gamma$ in $\ker(T_{0,\sigma}\to T_\sigma)$
    is $W^\sigma$-invariant and satisfies $\alpha(\gamma)=1$ for any root $\alpha$ of $T_\sigma$. Thus the following diagram commutes for any $\gamma \in \ker(T_{0,\sigma}\to T_\sigma)$ and $w\in W^\sigma$: 
    \[
    \begin{tikzcd}
	\CO(T_{0,\sigma}) & \CO(T_{0,\sigma}) \\
	\CO(T_{0,\sigma}) & \CO(T_{0,\sigma}),
	\arrow[from=1-1, to=1-2, "\delta_w"]
	\arrow[from=1-1, to=2-1, "t_\gamma"]
	\arrow[from=1-2, to=2-2, "t_\gamma"]
	\arrow[from=2-1, to=2-2, "\delta_w"]
    \end{tikzcd}
    \]
    where $t_\gamma$ denotes the translation by $\gamma$. 
    Hence $\delta_w$ for $T_{0,\sigma}$ is invariant under the $\ker(T_{0,\sigma}\to T_0)$-action, proving the claim by Galois descent.

    It suffices to prove that $\delta_{w,T_0}$ equals $\delta_{w,T}$, viewed as elements in $\End_{\CO(T_\sigma)^{W^\sigma}}(\CO(T_\sigma))$
    under the inclusion in Lemma \ref{Endeis}.
    Consider the strongly regular semisimple locus $j^{s.rs}\colon T^{s.rs}_\sigma\subset T_\sigma$,
    which is open and dense. It suffices to prove that they are equal in $\End_{T^{s.rs}_\sigma/W^\sigma}(j^{s.rs*}p_{T*}\CO_{T^{s.rs}_\sigma})$.
    
    Recall that in the proof of Lemma \ref{Endeis}, we have 
    $(T\sigma)^{s.rs}\sslash' N_G(T\sigma)\simeq T_\sigma^{s.rs}\sslash W^\sigma$.
    Thus the identity \ref{T_0,sig} becomes 
    \[
    \End_{T^{s.rs}_\sigma/W^\sigma}(j^{s.rs*}p_{T*}\CO_{T_\sigma})\otimes_{\CO(T_\sigma)} \CO(T_{0,\sigma}) \simeq \End_{T_{0,\sigma}^{s.rs}/W^\sigma}(j^{s.rs*}p_{T_0*}\CO_{T_{0,\sigma}})
    \]
    after localization on $T^{s.rs}_\sigma$
    ($T_{0,\sigma}^{s.rs}$ is defined to be the inverse image of $T^{s.rs}_\sigma$ in $T_{0,\sigma}$). Via this identity, $\delta_{w,T_0}$ and $\delta_{w,T}$ are equal in 
    $\End_{T_{0,\sigma}^{s.rs}/W^\sigma}(j^{s.rs*}p_{T_0*}\CO(T_{0,\sigma}))$, 
    as the following diagram
    \[
    \begin{tikzcd}
	\CO(T_{\sigma}) & \CO(T_{\sigma}) \\
	\CO(T_{0,\sigma}) & \CO(T_{0,\sigma}),
	\arrow[from=1-1, to=1-2, "\delta_{w,T}"]
	\arrow[hook, from=1-1, to=2-1]
	\arrow[hook, from=1-2, to=2-2]
	\arrow[from=2-1, to=2-2, "\delta_{w,T_0}"]
    \end{tikzcd}
    \]
    commutes by definition. Thus $\delta_{w,T_0}$ equals $\delta_{w,T}$, viewed as elements in $\End_{\CO(T_\sigma)^{W^\sigma}}(\CO(T_\sigma))$
    and the result follows.
\end{proof}

We have the following facts on the formal neighborhood of a closed point in $T_\sigma$.
For any point $x\in T_\sigma$, the translation by $x$ provides an isomorphism $T_{\sigma,x}^\wedge\simeq T_{\sigma,1}^\wedge$. For $w \in W^\sigma$, the $w$-action on $T_\sigma$ induces an isomorphism $T_{\sigma, x}^\wedge\simto T_{\sigma, w(x)}^\wedge$ and the diagram
\[
\begin{tikzcd}
	T_{\sigma, 1}^\wedge & T_{\sigma, x}^\wedge \\
	T_{\sigma, 1}^\wedge & T_{\sigma, w(x)}^\wedge
	\arrow[from=1-1, to=1-2, "\cdot x"]
	\arrow[from=1-1, to=2-1, "w"]
	\arrow[from=1-2, to=2-2, "w"]
	\arrow[from=2-1, to=2-2, "\cdot w(x)"]
\end{tikzcd}
\]
commutes.
In particular, there is a $W^\sigma_x$-action on $T_{\sigma,x}$ and $T_{\sigma,1}$ and the translation isomorphism is $W^\sigma_x$-equivariant.
Similarly we have $\ft_{\sigma.\xi}^\wedge\simeq \ft_{\sigma,0}^\wedge$ for any $\xi\in \ft_\sigma$. Moreover, we have $T_{\sigma,1}^\wedge\simeq \ft_{\sigma,0}^\wedge$ via the formal exponential and the isomorphism is $W^\sigma$-equivariant. Thus we have $\CO(T_\sigma)_x^\wedge \simeq \CO(\ft_\sigma)_{\xi}^\wedge$ naturally. We will denote these formal neighborhoods by $\hat{\ft}_\sigma$ if there is no ambiguity.

The following lemma shows that the formal neighborhood of any $W^\sigma$-orbit in $T_\sigma$ is isomorphic to the formal neighborhood of some $W^\sigma$-orbit in $\ft_\sigma$. 

\begin{lemma}
\label{loc-isom}
    Assume that $G$ is simply connected.
    For each closed point $x$ of $T_\sigma$, there exists a closed point $\xi \in \ft_\sigma$ such that the stabilizer subgroup satisfies $W^\sigma_x=W_\xi^\sigma \subset W^\sigma$ and there is an isomorphism of rings
    \[
    \CO(T_\sigma) \otimes_{\CO(T_\sigma \sslash W^\sigma)}
    \CO(T_\sigma \sslash W^\sigma)_{\bar{x}}^\wedge
    \simeq 
    \CO(\ft_\sigma) \otimes_{\CO(\ft_\sigma \sslash W^\sigma)}
    \CO(\ft_\sigma \sslash W^\sigma)_{\bar{\xi}}^\wedge
    \simeq \CO(W^\sigma\times^{W_x}\hat{\mathfrak{t}}_\sigma)
    \]
    compatible with the $W^\sigma$-action and the natural isomorphism 
    $\CO(T_\sigma)_x^\wedge \simeq \CO(\ft_\sigma)_{\xi}^\wedge$,
    where $\bar{x}$ denotes the image of $x$ in $T_\sigma \sslash W^\sigma$ and $\bar{\xi}$ denotes the image of $\xi$ in $\ft_\sigma \sslash W^\sigma$.

    Moreover, in these isomorphic rings, for any root $\alpha$ of $T_\sigma$, $1-e^{\tilde{\alpha}}(t) \in \CO(T_\sigma)$ and $\alpha \in \CO(\ft_\sigma)$ differ only by a unit.
\end{lemma}

\begin{proof}
    We have for any closed point $x\in T_\sigma$,
    \[
    \CO(T_\sigma) \otimes_{\CO(T_\sigma \sslash W^\sigma)}
    \CO(T_\sigma \sslash W^\sigma)_{\bar{x}}^\wedge
    \simeq \prod_{x'\in W^\sigma x} \CO(T_\sigma)_{x'}^\wedge
    \simeq \CO(W^\sigma\times^{W^\sigma_x} T_{\sigma, x}^\wedge)
    \]
    and similarly for any point $\xi\in \ft_\sigma$,
    \[
    \CO(\ft_\sigma) \otimes_{\CO(\ft_\sigma \sslash W^\sigma)}
    \CO(\ft_\sigma \sslash W^\sigma)_{\bar{\xi}}^\wedge
    \simeq \prod_{\xi'\in W^\sigma \xi} \CO(\ft_\sigma)_{\xi'}^\wedge
    \simeq \CO(W^\sigma\times^{W^\sigma_x} \ft_{\sigma, \xi}^\wedge).
    \]
    By those discussions before this lemma, $T_{\sigma,x}^\wedge\simeq \ft_{\sigma,\xi}^\wedge$ naturally.   
    Hence it suffices to find $\xi \in \ft_\sigma$ satisfying $W_x^\sigma=W^\sigma_\xi$.
    
    We claim that for any $x \in T_\sigma$, 
    $W^\sigma_x$ is generated by some reflections of $W^\sigma$. 
    When the root system $(T_\sigma, R^\sigma)$ is simply connected, denote by $G_\sigma$ the reductive group corresponding to the root system. By \cite{steinberg1968endomorphisms}, Theorem 8.1 applying on 
    $G_\sigma$ and its $x$-conjugation, the centralizer $C_{G_\sigma}(x)$ is a connected reductive group. Its Weyl group is $N_{C_{G_\sigma(x)}}(T_\sigma)/T_\sigma\simeq W^\sigma_x$. Thus $W^\sigma_x$ is generated by reflections. 
    If the root system $(T_\sigma, R^\sigma)$ is not simply connected, by classification of the pair $(G,\sigma)$ (\cite{xiao2018}, Lemma 5.1.1), a simple component of $G_\sigma$ is of type $\mathrm{B}_n$. In this case, replacing the roots in the root system $(T_\sigma, R^\sigma)$ via $\alpha \mapsto \tilde{\alpha}$, we obtain a simply connected root system $(T_\sigma, R'^\sigma)$ with the same Weyl group. Thus we reduce to the simply connected case and the claim is true.

    Denote by $s_1, \cdots, s_n$ the reflections generating $W^\sigma_x$. 
    We may take $\xi$ such that it lies in the subspace $\bigcap_{1\leq i\leq n} \{s_i(x)=x\}$ but it does not lie in the hyperplane of any other reflection which is linearly independent of $s_1, \cdots, s_n$. Then we have $W^\sigma_\xi = W^\sigma_x$ and the formula in the lemma holds.

    For any $\alpha$, $1-e^{\tilde{\alpha}}(t)$ and $\alpha$ differ only by a unit in the ring $\CO(T_\sigma)_x^\wedge \simeq \CO(\ft_\sigma)_\xi^\wedge$. Thus the second statement holds. 
\end{proof}

Now we can prove that the endomorphism ring is generated by the Demazure operators.
    
\begin{lemma}
\label{demazure}
    Viewing $\delta_w$'s as elements in $\End(\eis_B\CO_{T_\sigma})$, via Lemma \ref{compatible-Demazure}, 
    the set $\{\delta_w\}_{w\in W^\sigma}$ forms an $\CO(T_\sigma)$-basis of 
    \[
    \End(\eis_B\CO_{T_\sigma}) \simeq \End_{T\sigma\sslash'N_G(T\sigma)}(p_{T*}\CO_{T_\sigma}),
    \]
    viewed as an $O(T_\sigma)$-module by multiplication on the target.
\end{lemma}

\begin{proof}
    As the highest term of $\delta_w$ is a (fraction) multiple of $w$, the set $\{\delta_w\}_{w\in W^\sigma}$ is linearly independent. 
    Then we prove that $\{\delta_w\}_{w\in W^\sigma}$ is a set of generators.
    
    We firstly prove the case that $G$ is simply connected, i.e., that
    $\CO(T_\sigma)\{\delta_w\}_{w\in W^\sigma} = \End_{\CO(T_\sigma)^{W^\sigma}}(\CO(T_\sigma))$. It suffices to prove that for each closed point $x$ of $T_\sigma$, denoting by $\bar{x}$ the image of $x$ in $T_\sigma\sslash W^\sigma$, we have an equality of completions at $\bar{x}$:
    \begin{equation}
    \label{eq}
        \CO(T_\sigma)\{\delta_w\}_{w\in W^\sigma} \otimes_{\CO(T_\sigma \sslash W^\sigma)}
        \CO(T_\sigma \sslash W^\sigma)_{\bar{x}}^\wedge =
        \End_{\CO(T_\sigma)^{W^\sigma}}(\CO(T_\sigma)) \otimes_{\CO(T_\sigma \sslash W^\sigma)} \CO(T_\sigma \sslash W^\sigma)_{\bar{x}}^\wedge.
    \end{equation}
    By Lemma \ref{loc-isom}, there exists $\xi \in \ft_\sigma$ and isomorphisms
    \begin{align*}
    \CO(\ft_\sigma)\{\delta_w\}_{w\in W^\sigma} \otimes_{\CO(\ft_\sigma \sslash W^\sigma)}
    \CO(\ft_\sigma \sslash W^\sigma)_{\bar{\xi}}^\wedge 
    &\simeq \CO(T_\sigma)\{\delta_w\}_{w\in W^\sigma} \otimes_{\CO(T_\sigma \sslash W^\sigma)}
    \CO(T_\sigma \sslash W^\sigma)_{\bar{x}}^\wedge \\
    \End_{\CO(\ft_\sigma\sslash W^\sigma)}(\CO(\ft_\sigma)) 
    \otimes_{\CO(\ft_\sigma\sslash W^\sigma)} \CO(\ft_\sigma\sslash W^\sigma)_{\bar{\xi}}^\wedge 
    &\simeq \End_{\CO(T_\sigma)^{W^\sigma}}(\CO(T_\sigma)) \otimes_{\CO(T_\sigma \sslash W^\sigma)} \CO(T_\sigma \sslash W^\sigma)_{\bar{x}}^\wedge.
    \end{align*}
    Thus it suffices to prove the similar equality for $\ft_\sigma$.

    Again by Lemma \ref{loc-isom} and induction on the length of $w$, we have that for all $w \in W^\sigma$, $\{\delta_w\}_{w\in W^\sigma}$ and $\{D_w\}_{w\in W^\sigma}$ generate the same submodule of
    \[
    \End_{\CO(\ft_\sigma\sslash W^\sigma)}(\CO(\ft_\sigma)) 
    \otimes_{\CO(\ft_\sigma\sslash W^\sigma)} \CO(\ft_\sigma\sslash W^\sigma)_{\bar{\xi}}^\wedge 
    \]
    Thus the left hand side of \ref{eq} is isomorphic to
    \[
    \CO(\ft_\sigma)\{D_w\}_{w\in W^\sigma} \otimes_{\CO(\ft_\sigma \sslash W^\sigma)}
    \CO(\ft_\sigma \sslash W^\sigma)_{\bar{\xi}}^\wedge.
    \]
    
    By Demazure's theorem (\cite{demazure}, Théorème 2),
    we have
    \[
    \CO(\ft_\sigma)\{D_w\}_{w\in W^\sigma} = \End_{\CO(\ft_\sigma\sslash W^\sigma)}(\CO(\ft_\sigma)).
    \]
    Thus the identity \ref{eq} holds by taking the completion at $\bar{\xi}$ on the above formula.

    Then we prove the case for general $G$. Take a simply connected cover $G_0\to G$. 
    Recall that we have an identity
    \[
    \End_{T\sigma\sslash'N_G(T\sigma)}(p_{T*}\CO_{T_\sigma})\otimes_{\CO(T_\sigma)} \CO(T_{0,\sigma}) \simeq \End_{T_0\sigma\sslash N_{G_0}(T_0\sigma)}(p_{T_0*}\CO_{T_{0,\sigma}}) = \End_{\CO(T_{0,\sigma})^{W^\sigma}}(\CO(T_{0,\sigma}))
    \]
    and under this identity the operators $\delta_w$ for $T$ and $\delta_w$ for $T_0$ are equal by Lemma \ref{compatible-Demazure}.
    To prove that $\CO(T_\sigma)\{\delta_w\}_{w\in W^\sigma} = \End_{T\sigma\sslash'N_G(T\sigma)}(p_{T*}\CO_{T_\sigma})$, it suffices to prove that it is so after applying $(-)\otimes_{\CO(T_\sigma)}\CO(T_{0,\sigma})$. Then we reduce to the simply connected case and conclude.
\end{proof}

The following lemma describes $\End(\Eis_B\CO_{T_\sigma})$ as an $\CO(T_\sigma)$-module by multiplication on the target.

\begin{lemma}
    \label{eis-Eis}
    Recall that we have inclusions
    \[
    \End(\Eis_B\CO_{T_\sigma}) \subset \End(\eis_B\CO_{T_\sigma})
    \subset \End_{\CO(T_\sigma)^{W^\sigma}}(\CO(T_\sigma))
    \]
    in Lemma \ref{inclusion} and Lemma \ref{Endeis}.
    For an element $w\in W^\sigma$, abbreviate the L-function 
    $L(1, \ad_{\fn/w^{-1}(\fn)})^{-1}$ by $L_w$. Moreover, 
    choose arbitrarily a reduced expression $s_1\dots s_n$ of $w$. We will denote
    the composition operator $L_{s_1}\delta_{s_1}\dots L_{s_n}\delta_{s_n}$ in $\End(\eis_B\CO_{T_\sigma})$ by $(L\delta)_w$ (it may depend on the choice of $s_1,\dots, s_n$). Then $(L\delta)_w$ lies in $\End(\Eis_B\CO_{T_\sigma})$
    and a basis of $\End(\Eis_B\CO_{T_\sigma})$ is $\{(L\delta)_w\}_{w\in W^\sigma}$.
\end{lemma}

\begin{proof}
    For the statement that for any $w$ and any reduced expression $s_1\dots s_n$, 
    $(L\delta)_w$ lies in $\End(\Eis_B\CO_{T_\sigma})$, 
    by definition, it suffices to prove the case that $w$ is a simple reflection $s$. In this case, denote $P_s$ the parabolic subgroup of $G$ corresponding to $s$ and $M_s$ the corresponding Levi subgroup. Then we have
    \[
    \Eis_B^G = \Eis_{P_s}^G\Eis_{M_s\cap B}^{M_s}
    \]
    and the L-function $L_s$ for $G$ is equal to $L_s$ for $M_s$. 
    It thus suffices to prove that $L_s\delta_s$ defines an endomorphism of $\Eis_{M_s\cap B}\CO_{T_\sigma}$. 
    Hence by replacing $G$ by $M_s$, we reduce to the rank $1$ case. 

    As $\BL_{G\sigma}^g\simeq (G\sigma)^g/G$, $L_s\delta_s$ defines an endomorphism of $\Eis_B\CO_{T_\sigma}$ after restricting to the generic locus. Moreover,
    in the rank $1$ case, by Lemma \ref{ord2}, the non-regular semisimple locus and the non-generic locus of $\BL_{G\sigma}$ are disjoint. Thus over the formal neighborhood of the non-generic locus, $\delta_s = \frac{1-s}{1-e^{-\tilde{\alpha_s}}}$ differs with $1-s$ only by a unit. As $\Eis_{s(B)}\simeq \Eis_B\circ s$, by Proposition \ref{all}, $L_s s$ defines an endomorphism of $\Eis_B\CO_{T_\sigma}$ over the whole space and $L_s\delta_s$ defines an endomorphism over the formal neighborhood of the non-generic locus. Thus by Beauville--Lazslo gluing the result follows.

    We then prove that $\{(L\delta)_w\}_{w\in W^\sigma}$ forms a basis of $\End(\Eis_B\CO_{T_\sigma})$ as an $\CO(T_\sigma)$-module. As $\Eis_B^g\simeq \eis_B^g$ by Lemma \ref{generic locus} and 
    $T\sigma\sslash'N_G(T\sigma)\simeq T_\sigma\sslash W^\sigma$ over the strongly regular semisimple locus, the inclusions in the statement induce isomorphisms
    of $k(T_\sigma)$-vector spaces
    \[
    \End(\Eis_B\CO_{T_\sigma})\otimes_{\CO(T_\sigma)} k(T_\sigma) \simto \End(\eis_B\CO_{T_\sigma})\otimes_{\CO(T_\sigma)} k(T_\sigma)
    \simto \End_{\CO(T_\sigma)^{W^\sigma}}(\CO(T_\sigma))\otimes_{\CO(T_\sigma)} k(T_\sigma)
    \]
    and we may view them as the same vector space.
    Furthermore, as $\CO(T_\sigma)$ is a torsion-free $\CO(T_\sigma)^{W_\sigma}$-module,
    $\End_{\CO(T_\sigma)^{W^\sigma}}(\CO(T_\sigma))$ and hence $\End(\Eis_B\CO_{T_\sigma})$ and $\End(\eis_B\CO_{T_\sigma})$ inject into the vector space.
    Then we may view $w$'s, $\delta_w$'s and $(L\delta)_w$'s for $w\in W^\sigma$ as elements in this vector space. In this vector space $\{w\}_{w\in W^\sigma}$ is a basis.
    
    From the commuting relation $\delta_s f= \delta_s(f)+s(f)\delta_s$ of Demazure operators, $(L\delta)_w$ is a $\CO(T_\sigma)$-linear combination of $\delta_{w'}$'s for $w'\leq w$ with respect to the Bruhat ordering and the highest term of $(L\delta)_w$ is $L_w\delta_w$. 
    Thus $(L\delta)_w$'s are linearly independent as
    $\{\delta_w\}_{w\in W^\sigma}$ forms a basis of $\End(\eis_B\CO_{T_\sigma})$ as an $\CO(T_\sigma)$-module by Lemma \ref{demazure}. It suffices to prove that any element in $\End(\Eis_B\CO_{T_\sigma})$ can be written as an $\CO(T_\sigma)$-linear combination of the set $\{(L\delta)_w\}_{w\in W^\sigma}$.
    Any element $a$ of $\End(\Eis_B\CO_{T_\sigma})$ can be expressed as $\sum_{w\in W^\sigma} a_w\delta_w$ for $a_w \in \CO(T_\sigma)$ as $\{\delta_w\}_{w\in W^\sigma}$
    is a basis of $\End(\eis_B\CO_{T_\sigma})$.
    We use induction to the highest non-zero terms of the expression. Choose a maximal $w\in W^\sigma$ such that $a_w\neq 0$. By definition of $\delta_w$, it is a $k(T_\sigma)$-linear combination of $w'$'s for $w'\leq w$ and the coefficient of $\delta_w$ at $w$ under the basis $\{w\}_{w\in W^\sigma}$ is $L(0, \ad_{\fn/w^{-1}(\fn)})$ up to a unit in $\CO(T_\sigma)$. Thus the coefficient of $a$ at $w$ under the basis $\{w\}_{w\in W^\sigma}$ is $a_w L(0, \ad_{\fn/w^{-1}(\fn)})$ up to a unit in $\CO(T_\sigma)$.
    
    For an irreducible factor $f$ of $L_w$ and strongly regular semisimple $x \in V(f)$,
    view it as a point in $T\sigma/T$ and denote by $y$ its image in $G\sigma/G$.
    Recall that we have a projection $\pi_G\colon \BL_{G\sigma}\to G\sigma/G$ and by 
    Lemma \ref{formal-G}, 
    \[(\Eis_B\CO_{T_\sigma})_{\pi_G^{-1}(y)}^\wedge \simeq \bigoplus_{w'\in W^\sigma}\Eis_B\CO_{T_\sigma, w'(x)}^\wedge. \]
    Then we have homomorphisms
    \[
    \End(\Eis_B\CO_{T_\sigma})\to \End((\Eis_B\CO_{T_\sigma})_{\pi_G^{-1}(y)}^\wedge) 
    \simeq \End(\bigoplus_{w'\in W^\sigma}\Eis_B\CO_{T_\sigma, w'(x)}^\wedge) \to
    \Hom(\Eis_B\CO_{T_\sigma, w^{-1}(x)}^\wedge, \Eis_B \CO_{T_\sigma,x}^\wedge).
    \]
    In the rightmost space, only the term at $w$ contributes. Thus $a_w L(0, \ad_{\fn/w^{-1}(\fn)})$ defines a homomorphism in
    $\Hom(\Eis_B\CO_{T_\sigma, w^{-1}(x)}^\wedge, \Eis_B \CO_{T_\sigma,x}^\wedge)$. 
    By the same local calculations in the proof of Proposition \ref{all},
    this condition is equivalent to that $a_w L(0, \ad_{\fn/w^{-1}(\fn)})$ is divisible by $L_w$ in the completion $\CO(T_\sigma)_x^\wedge$. Varying strongly regular semisimple $x \in V(f)$, we see that $a_w L(0, \ad_{\fn/w^{-1}(\fn)})$ is divisible by $f$ as a fraction. Then varying $f\mid L(1, \ad_{\fn/w(\fn)})^{-1}$, 
    we have $a_wL(0, \ad_{\fn/w^{-1}(\fn)})$ is divisible by $L_w$ and thus $a_w$ is divisible by $L_w$ in $\CO(T_\sigma)$. Thus $a - a_w (L_w)^{-1}(L\delta)_w$ also lies in $\End(\Eis_B\CO_{T_\sigma})$ and has smaller highest terms with respect to the basis $\{\delta_w\}_{w\in W^\sigma}$ and we apply the inductive hypothesis to conclude.
\end{proof}

Now we construct the left adjoint functor of $\Eis$ in the next propositions. We will consider the $\CO(T_\sigma)$-bimodule structures
of $\End(\eis_B\CO_{T_\sigma})$ and $\End(\Eis_B\CO_{T_\sigma})$,
where the left $\CO(T_\sigma)$-module structures are defined to be the action on the targets and the right $\CO(T_\sigma)$-structures are the action on the sources.

We firstly consider the $\chi=1$ case.

\begin{proposition}
\label{ct_circ_eis}
    Denote by $D\Coh(\BL_{G\sigma})_{T,1}$ the thick subcategory of $D\Coh(\BL_{G\sigma})$ generated by the image by $\Eis$ (which is independent of the choice of $B$) of $D\Coh(\BL_{T\sigma})_1$. Similarly, we define $D\Coh(G\sigma/G)_{T,1}$ for $\eis$.

    The left adjoint functor (in the $\infty$-categorical sense) of $\eis_B$, 
    \[\ct_B \colon D\Coh(G\sigma/G)_{T,1} \to D\Coh(T\sigma/T)_1\] 
    exists and more precisely we have for $\CF\in D\Coh(T\sigma/T)_1$,
    \[
    p_T^*p_{T*}\CF \simeq \ct_B \circ \eis_B(\CF),
    \]
    where $p_T$ is the projection $T\sigma\sslash T \to T\sigma\sslash' N_G(T\sigma)$.
    Moreover, the left adjoint functor (in the $\infty$-categorical sense) of $\Eis_B$, 
    \[\CT_B \colon D\Coh(\BL_{G\sigma})_{T,1} \to D\Coh(\BL_{T\sigma})_1\] 
    exists and for a locally free sheaf $\CF$ on $T_\sigma$,  
    there is a functorial inclusion of locally free sheaves (concentrated in degree $0$)
    of the same rank:
    \[
    p_T^*p_{T*}\CF \subset \CT_B \circ \Eis_B(\CF)
    \]
    which induces an $\CO(T_\sigma)$-bimodule homomorphism on global sections. Here the left $\CO(T_\sigma)$-module structures of $p_T^*p_{T*}\CF$ and $\CT_B \circ \Eis_B(\CF)$ are defined to be induced by the $\CO_{T_\sigma}$-coherent sheaf structure and the right $\CO(T_\sigma)$-module structures are induced by the $\CO(T_\sigma)$-action on $\CF$.

    More precisely, the global section $\Gamma(p_T^*p_{T*}\CO_{T_\sigma})$
    is isomorphic to $\CO(T_{\sigma}\times_{T\sigma\sslash' N_G(T\sigma)} T_\sigma)$
    as $\CO(T_\sigma)$-bimodules.
    As we have (equivariantly as $\CO(T_\sigma)$-bimodules)
    \[
    \Hom_{\CO(T_\sigma)}(\CO(T_{\sigma}\times_{T\sigma\sslash' N_G(T\sigma)} T_\sigma), \CO(T_\sigma))\simeq \End_{T\sigma\sslash' N_G(T\sigma)}(p_{T*}\CO_{T_\sigma})
    \]
    and $\End(p_{T*}\CO_{T_\sigma})$ and $\CO(T_{\sigma}\times_{T\sigma\sslash' N_G(T\sigma)} T_\sigma)$ are free $\CO(T_\sigma)$-modules, $\CO(T_{\sigma}\times_{T\sigma\sslash' N_G(T\sigma)} T_\sigma)$ can be 
    viewed as the $\CO(T_\sigma)$-dual of $\End(p_{T*}\CO_{T_\sigma})$. 
    Then the global section $\Gamma(\CT_B\circ \Eis_B\CO_{T_\sigma})$ can be identified with the $\CO(T_\sigma)$-submodule of the $k(T_\sigma)$-vector space $\CO(T_{\sigma}\times_{T\sigma\sslash' N_G(T\sigma)} T_\sigma)\otimes_{\CO(T_\sigma)} k(T_\sigma)$ generated by the dual basis $\{(L\delta)_w^*\}_{w\in W^\sigma}$ of $\{(L\delta)_w\}_{w\in W^\sigma}$. 
\end{proposition}

\begin{proof}
    We firstly prove the argument for $\eis$ and $\ct$. Choose a simply connected cover
    $G_0\to G$. Recall the morphisms $\beta\colon G\sigma/G_0\to G\sigma/G$ and $\pi_T \colon G\sigma/ G_0 \to T\sigma\sslash' N_G(T\sigma)$.
    By the proof of Lemma \ref{Endeis}, the unit map $\mathrm{id}\to \pi_{T*}\pi_T^*$ is an isomorphism and $\pi_T^*$ is fully faithful. Thus the left adjoint of $\pi_{T}^*$ in the subcategory of essential image of $\pi_T^*$ is $\pi_{T*}$.
    By Lemma \ref{locallyfreeeis}, we have
    \[
    \beta^*\eis_B \mathcal{F} \simeq \pi_T^* p_{T*}\mathcal{F}
    \]
    and thus $\beta^*D\Coh(G\sigma/G)_{T,1}\subset \pi_T^*D\Coh(T\sigma\sslash' N_G(T\sigma))$. We claim that $p_T^*\pi_{T*}\beta^*$ is the desired $\ct_B$.
    Indeed, for $\CF \in D\Coh(G\sigma/G)_{T,1}$ and $\CG \in D\Coh(T_\sigma)$,
    \[
    \Hom(p_T^*\pi_{T*}\beta^*\CF, \CG)\simeq \Hom(\beta^*\CF, \pi_T^{*}p_{T*}\CG)
    \simeq \Hom(\beta^*\CF, \beta^*\eis_B\CG)\simeq \Hom(\CF, \eis_B\CG).
    \]
    Thus $\ct_B\eis_B \simeq p_T^*p_{T*}$ and $\Gamma(\ct_B\eis_B\CO_{T_\sigma})$
    is isomorphic to $\CO(T_{\sigma}\times_{T\sigma\sslash' N_G(T\sigma)} T_\sigma)$.
    
    To construct $\CT_B$, we firstly construct 
    $\CT_B$ of $\Eis_B\CO_{T_\sigma}$.
    By Lemma \ref{eis-Eis}, we have that (equivariantly as $\CO(T_\sigma)$-bimodules)
    \[
    \End(\Eis_B\CO_{T_\sigma}) \subset \End(\eis_B\CO_{T_\sigma})
    \simeq \Hom(\ct_{B}\circ\eis_B\CO_{T_\sigma}, \CO_{T_\sigma})
    \simeq \Hom_{\CO(T_\sigma)}(\CO(T_{\sigma}\times_{T\sigma\sslash' N_G(T\sigma)} T_\sigma), \CO(T_\sigma)).
    \]
    Thus we construct $\CT_B\circ \Eis_B\CO_{T_\sigma}$ as the coherent sheaf
    corresponding to the $\CO(T_\sigma)$-module
    \[
    \Hom_{\CO(T_\sigma)}(\End(\Eis_B\CO_{T_\sigma}), \CO(T_\sigma))
    \supset \Gamma(\ct_{B}\circ\eis_B\CO_{T_\sigma}),
    \]
    which inherits an $\CO(T_\sigma)$-bimodule structure. The statement about the basis of $\CT_B\circ \Eis_B\CO_{T_\sigma}$ follows from 
    Lemma \ref{eis-Eis}. Thus we prove the statement for $\CF=\CO_{T_\sigma}$.
    
    This construction extends to the whole $D\Coh(\BL_{G\sigma})_{T,1}$ by linearity: For a general element $\CG_\bullet$ in $D\Coh(\BL_{G\sigma})_{T,1}$, it can be written as a complex $\CG_\bullet=\{\CG_n,\delta_n\}$ where $\CG_n=\Eis_B\CO_{T_\sigma}^{k_n}$ for all $n$. Then we may construct the complex $\CT_B(\CG)_\bullet=\{\CT_B(\CG_n), \CT_B(\delta_n)\}$. Here we construct $\CT_B(\delta_n)$ in the following formal process. 
    For any two categories $\CC$ and $\CD$, a functor $G\colon \CD\to \CC$ and $c\in \CC$, $F(c)$ exists if and only if $\Hom(c, G(-)) \in \mathsf{Fun}(\CD,\mathsf{Ani})$ is representable. Thus for $c_1, c_2 \in \CC$ such that $F(c_1)$, $F(c_2)$ exists, 
    $c_1\to c_2$ induces a morphism $\Hom(c_2, G(-))\to \Hom(c_1, G(-))$ and then a morphism $F(c_1)\to F(c_2)$ by Yoneda lemma.
    Then we have isomorphisms
    \[
    R\Hom((\CT_B\CG)_\bullet, \CF)\simeq (\Hom(\CT_B\CG, \CF))_\bullet
    \simeq (\Hom(\CG, \Eis_B\CF))_\bullet \simeq R\Hom(\CG_\bullet, \Eis_B\CF),
    \]
    where the first isomorphism is from that $\CT_B\CG$ is locally free, the second is from the adjunction and the third is from that $R\Hom(\Eis_B\CG, \Eis_B\CF)$ is concentrated in degree $0$ by Lemma \ref{inclusion}.
    Hence $(\CT_B\CG)_\bullet$ is the desired object $\CT_B(\CG_\bullet)$
    and we prove the existence of $\CT_B$.
    The statement about the inclusion $\ct_B\eis_B\subset \CT_B\Eis_B$
    follows from the case for $\CF=\CO_{T_\sigma}$.
\end{proof}

Next we construct $\CT$ for general $\chi$.

\begin{proposition}
\label{CT_chi}
    Let $\chi$ be a character in $X^*(T^\sigma)^{W^\sigma}$.
    Denote by $D\Coh(\BL_{G\sigma})_{T,\chi}$ the thick subcategory of $D\Coh(\BL_{G\sigma})$ generated by the image by $\Eis$ (which is independent of the choice of $B$) of $D\Coh(\BL_{T\sigma})_\chi$. Similarly, we define $D\Coh(G\sigma/G)_{T,\chi}$ for $\eis$.

    The left adjoint functor (in the $\infty$-categorical sense) of $\Eis_B$, 
    \[
    \CT_B \colon D\Coh(\BL_{G\sigma})_{T,\chi} \to D\Coh(\BL_{T\sigma})_\chi
    \]
    and the left adjoint functor (in the $\infty$-categorical sense) of $\eis_B$, 
    \[
    \ct_B \colon D\Coh(G\sigma/G)_{T,\chi} \to D\Coh(T\sigma/T)_\chi
    \]
    exist. More precisely, for any $\CF \in D\Coh(\BL_{T\sigma})_1$, we have
    \[
    \CT_B(\Eis_B\CF_\chi) \simeq (\CT_B\Eis_B\CF)_\chi
    \]
    and similarly for $\eis$.
\end{proposition}

\begin{proof}
    We firstly prove that $(\CT_B\Eis_B\CF)_\chi$ is the desired object
    $\CT_B(\Eis_B\CF_\chi)$ for any $\CF \in D\Coh(\BL_{T\sigma})_1$.
    Indeed, for any $\CG\in D\Coh(\BL_{T\sigma})_1$, we have
    \begin{align*}
        \Hom((\CT_B\Eis_B\CF)_\chi, \CG_\chi) \simeq
        \Hom(\CT_B\Eis_B\CF, \CG) \simeq 
        \Hom(\Eis_B\CF, \Eis_B\CG) \simeq
        \Hom(\Eis_B\CF_\chi, \Eis_B\CG_\chi),
    \end{align*}
    where the last isomorphism is from Lemma \ref{tensor_by_chi_Eis}.
    Then by the same formal arguments in the end of the proof of Proposition \ref{ct_circ_eis}, $\CT_B$ can be extended to the whole category $D\Coh(\BL_{G\sigma})_{T,\chi}$. The statements for $\eis$ and $\ct$ follows in the same way.
\end{proof}

\begin{remark}
\label{CT_B*}
    As mentioned in the introduction, our definition of $\CT$ is not the same as the definition of $\CT$ in \cite{hansen_cllc}, where $\CT$ is defined as a right adjoint of $\Eis$. 
    
    David Hansen indicated a more general way to define $\CT_\chi$ (for general $\chi$) to me. Denote by $\CT^r_\chi$ the right adjoint of $\Eis_\chi$ (which exists and preserves $D\Coh$ by \cite{hansen_cllc}, Theorem 1.9.4). Then the functor
    \[
    \mathsf{D}_{GS, T} \circ \CT^r_{\chi^{-1}} \circ \mathsf{D}_{GS, G}
    \]
    where $\mathsf{D}_{GS, G}$ and $\mathsf{D}_{GS, T}$ is the Grothendieck--Serre duality functor on $\BL_{G\sigma}$ and $\BL_{T\sigma}$ respectively, is the left adjoint $\CT_\chi$. This construction can also be generalized to other blocks in $\Par_G$.
\end{remark}

The $\CO(T_\sigma)$-bimodules $\Gamma(p_T^*p_{T*}\CO_{T_\sigma})$ and $\Gamma(\CT_B\circ\Eis_B\CO_{T_\sigma})$ above are similar to Soergel bimodules. We may consider the following analog of filtrations.
Note that the scheme $T_{\sigma}\times_{T\sigma\sslash' N_G(T\sigma)} T_\sigma$ a union of irreducible components isomorphic to $T_\sigma$, indexed by $w \in W^\sigma$. We may define the following filtration of $\CO(T_{\sigma}\times_{T\sigma\sslash' N_G(T\sigma)} T_\sigma)$,
which is the counterpart of the Bruhat--Mackey's filtration on the spectral side.

\begin{definition}[Filtration]
\label{filtration}
    Suppose that there is a partial order on $W^\sigma$, such that it refines the Bruhat order.
    Denote by $\CO(T_\sigma)_{w}$ the $\CO(T_\sigma)$-bimodule with natural action on the left and twisted action by $w$ on the right.
    The filtration $F^{\geq w}$, $F^{>w}$ index by $w \in W^\sigma$ with respect to the order above of $\CO(T_{\sigma}\times_{T\sigma\sslash' N_G(T\sigma)} T_\sigma)$ is defined by
    \[
    F^{\geq w} := \bigcap_{w'<w} \ker(\CO(T_{\sigma}\times_{T\sigma\sslash' N_G(T\sigma)} T_\sigma) \xrightarrow{(1,w')} \CO(T_\sigma)_{w'}).
    \]
    and 
    \[
    F^{> w} := \bigcap_{w'\leq w} \ker(\CO(T_{\sigma}\times_{T\sigma\sslash' N_G(T\sigma)} T_\sigma) \xrightarrow{(1,w')} \CO(T_\sigma)_{w'}).
    \] 
    Taking saturation, we obtain a filtration on $\CT_B\circ \Eis_B\CO_{T_\sigma}$. Since the filtration is compatible with homomorphisms between locally free sheaves in $D\Coh(\BL_{T\sigma})_1$.
    Extending the filtration by linearity, 
    we obtain uniquely a filtration on $\CT_B\circ\Eis_B\CF$ for all $\CF\in D\Coh(\BL_{T\sigma})_1$. By Proposition \ref{CT_chi}, for any $\chi \in X^*(T^\sigma)^{W^\sigma}$, tensoring by $\CO_\chi$ defines a filtration on $\CT_B\circ\Eis_B\CF_\chi$. In this way, we obtain a filtration on the functor $\CT_B\circ \Eis_B$ over $D\Coh(\BL_{T\sigma})_\chi$. 
    
    Moreover, for $B' = w^{-1}(B)$, we have 
    \[
    \CT_B\circ \Eis_{B'} \simeq \CT_B \circ \Eis_B \circ w_*
    \]
    where $W$ denotes the $w$-twisting, and the filtration $F^{>w'}$ of $\CT_B\circ \Eis_{B'}$ is defined as $F^{>w'w}$ of $\CT_B \circ \Eis_B$, under the isomorphism above.
\end{definition}

Here we use the term ``filtration'' for a compatible family of functors, insisting on the fact that for locally free sheaves, it defines a classical filtration in the abelian category $\Coh(T_\sigma)$. We will see that the piece $F^w=F^{\geq w}/F^{>w}$ is isomorphic to $w^*$, as desired.

\begin{remark}[Geometric intuition of the filtration]
\label{geometric}
    We present a geometric intuition of the filtration here. Consider the following commutative diagram:
    \[
    \begin{tikzcd}
	&& {X:=\BL_{B\sigma}\times_{\BL_{G\sigma}} \BL_{B\sigma}} \\
	& {\BL_{B\sigma}} && {\BL_{B\sigma}} \\
	{\BL_{T\sigma}} && {\BL_{G\sigma}} && {\BL_{T\sigma}}
	\arrow["{p_1}"', from=1-3, to=2-2]
	\arrow["{p_2}", from=1-3, to=2-4]
	\arrow["q_B"', from=2-2, to=3-1]
	\arrow["p_B", from=2-2, to=3-3]
	\arrow["p_B"', from=2-4, to=3-3]
	\arrow["q_B", from=2-4, to=3-5]
    \end{tikzcd}
    \]
    For terms in the diagram, we have
    \[
    \BL_{B\sigma} \simeq G\times^B \tilde{\BL}_{B\sigma}/G \simeq 
    \{(g, h)\in G/B\times \tilde{\BL}_{G\sigma}\mid \ad(g)^{-1}h\in \tilde{\BL}_{B\sigma}\}/G,
    \]
    and 
    \[
    \BL_{B\sigma}\times_{\BL_{G\sigma}} \BL_{B\sigma} \simeq 
    \{(g_1, g_2, h)\in G/B\times G/B \times \tilde{\BL}_{G\sigma}\mid \ad(g_1)^{-1}h,
    \ad(g_2)^{-1}h \in \tilde{\BL}_{B\sigma}\}/G.
    \]
    Via the projection 
    \[
    \BL_{B\sigma}\times_{\BL_{G\sigma}} \BL_{B\sigma} \to (G/B\times G/B)/G
    \simeq B\backslash G/B = \coprod_{w\in W} B\backslash BwB/B,
    \]
    there is a stratification of $\BL_{B\sigma}\times_{\BL_{G\sigma}} \BL_{B\sigma}$ indexed by $W$. Moreover, the stratum indexed by $w$ is non-empty if and only if $wB\sigma(w)^{-1}\cap B\neq \emptyset$, which is equivalent to $w \in W^\sigma$. Thus we obtain a $W^\sigma$-stratification of $\BL_{B\sigma}\times_{\BL_{G\sigma}} \BL_{B\sigma}$, with respect to the Bruhat order. The filtration satisfies that the $w$-stratum of $X$, $X_w$ is isomorphic to $\BL_{B\cap w^{-1}B\sigma}$ by \cite{hansen_cllc}, Proposition 3.6.1 and then it has a closed substack $\BL_{T\sigma}$.
    Thus we may expect that the filtration on $\CT\circ \Eis$ can be deduced from the stratification on $\BL_{B\sigma}\times_{\BL_{G\sigma}} \BL_{B\sigma}$ in some way.

    In \cite{hansen_cllc}, Theorem 3.6.3, a filtration $\CT_{B*}\Eis_B$ is defined in this way for the right adjoint $\CT_{B*}$ of $\Eis_B$. 
\end{remark}

The subquotients of the filtrations are free $\CO(T_\sigma)$-modules. More precisely:
\begin{lemma}
\label{free}
    The $\CO(T_\sigma)$-modules
    $F^{\geq w}\subset \CO(T_{\sigma}\times_{T\sigma\sslash' N_G(T\sigma)} T_\sigma)$, $\CO(T_{\sigma}\times_{T\sigma\sslash' N_G(T\sigma)} T_\sigma)/F^{\geq w}$, $F^{>w}\subset \CO(T_{\sigma}\times_{T\sigma\sslash' N_G(T\sigma)} T_\sigma)$ and $\CO(T_{\sigma}\times_{T\sigma\sslash' N_G(T\sigma)} T_\sigma)/F^{>w}$ are free $\CO(T_\sigma)$-modules (by left action). Similarly for the filtration on $\Gamma(\CT_B\circ \Eis_B\CO_{T_\sigma})$. 

    More precisely, the $\CO(T_\sigma)$-module $\CO(T_{\sigma}\times_{T\sigma\sslash' N_G(T\sigma)} T_\sigma)/F^{> w}$ is generated by the dual basis $\{\delta_{w'}^*\}_{w'\leq w}$ and the $\CO(T_\sigma)$-module $\Gamma(\CT_B\circ \Eis_B\CO_{T_\sigma})/F^{>w}$ is generated by the dual basis $\{(L\delta)_{w'}^*\}_{w'\leq w}$. Similarly for other modules.

    In particular, $\Gamma(F^{\geq w}/F^{>w}(\CO_{T_\sigma}))$ is free and of rank $1$ both for $\Eis$ and $\eis$. Moreover, Denote by $\{e_w\}_{w\in W^\sigma} \subset \CO(T_{\sigma}\times_{T\sigma\sslash' N_G(T\sigma)} T_\sigma)$ the (rational) dual basis of $\{(1,w)\}_{w\in W^\sigma}$. Then $L(0, \ad_{\fn/w^{-1}(\fn)})^{-1}e_w$ is a basis of $F^{\geq w}/F^{>w}$ for $\eis$ and $L(0, \ad_{\fn/w^{-1}(\fn)})^{-1}L_w^{-1} e_w$ is a basis of $F^{\geq w}/F^{>w}$ for $\Eis$.
\end{lemma}

\begin{proof} 
    In the proof, we will denote $\Hom_{\CO(T_\sigma)}$ to be the set of homomorphisms as left $\CO(T_\sigma)$-modules. 
    By definition of the morphism $p_{T*}$, we have an isomorphism of $\CO(T_\sigma)$-bimodules.
    \[
    \Hom_{\CO(T_\sigma)}(\CO(T_{\sigma}\times_{T\sigma\sslash' N_G(T\sigma)} T_\sigma), \CO(T_\sigma))
    \simeq \End_{T\sigma\sslash' N_G(T\sigma)}(p_{T*}\CO_{T_\sigma}),
    \]
    under which $(1,w)$ corresponds to $w$. We also view $\delta_w$ as an element in the left hand side.
    For any $w\in W$, $\{(1, w')\}_{w'<w}$ and $\{\delta_{w'}\}_{w'<w}$ generate the same $k(T_\sigma)$-subspace in the vector space 
    \[
    \Hom_{\CO(T_\sigma)}(\CO(T_{\sigma}\times_{T\sigma\sslash' N_G(T\sigma)} T_\sigma), \CO(T_\sigma))
    \otimes_{\CO(T_\sigma)} k(T_\sigma)
    \] by definition of $\delta_w$.
    Thus we have $F^{\geq w}\subset \CO(T_{\sigma})\otimes_{\CO(T_\sigma\sslash W^\sigma)}\CO(T_{\sigma})$ is $\bigcap_{w'<w}\ker(\delta_{w'})$, which is free as $\{\delta_w\}_{w\in W^\sigma}$ forms a basis of $\End_{T\sigma\sslash' N_G(T\sigma)}(p_{T*}\CO_{T_\sigma})$ by Lemma \ref{demazure}.
    Other submodules or quotient modules of $\CO(T_{\sigma}\times_{T\sigma\sslash' N_G(T\sigma)} T_\sigma)$ in the statement are free for the same reason.
    
    Moreover, for $\CT_B\circ \Eis_B \CO_{T_\sigma}$, $\{(L\delta)_{w'}\}_{w'<w}$ and $\{\delta_{w'}\}_{w'<w}$ generate the same $k(T_\sigma)$-subspace in the vector space
    \[\Hom_{\CO(T_\sigma)}(\CO(T_{\sigma}\times_{T\sigma\sslash' N_G(T\sigma)} T_\sigma), \CO(T_\sigma))
    \otimes_{\CO(T_\sigma)} k(T_\sigma)
    = \Hom(\CT_B\circ \Eis_B \CO_{T_\sigma}, \CO_{T_\sigma})\otimes_{\CO(T_\sigma)} k(T_\sigma)
    \]
    by definition of $(L\delta)_w$. 
    As the filtration of $\Gamma(\CT_B\circ\Eis_B)$ is defined by saturation, we have $F^{\geq w}\subset \Gamma(\CT_B\circ \Eis_B\CO_{T_\sigma})$ is $\bigcap_{w'<w}\ker((L\delta)_{w'})$, which is free as $\{(L\delta)_w\}_{w\in W^\sigma}$ forms a basis of the module 
    \[\End(\Eis_B\CO_{T_\sigma})\simeq\Hom_{\CO(T_\sigma)}(\Gamma(\CT_B\circ \Eis_B\CO_{T_\sigma}), \CO(T_\sigma)).\]
    Other submodules or quotient modules of $\Gamma(\CT_B\circ \Eis_B\CO_{T_\sigma})$ in the statement are free for the same reason.

    For the statement about the generator of $F^{\geq w}/F^{>w}$, note that $\delta_w^*$ is its generator the coefficient of $\delta_w$ at $w$ is $L(0, \ad_{\fn/w^{-1}\fn})$ and other terms are in the places that are smaller than $w$ in the Bruhat order. Thus $\delta_w^*-L(0, \ad_{\fn/w^{-1}\fn})^{-1}e_w$ lies in $F^{>w}$, proving the case for $\eis$. As the leader term of $(L\delta)_w$ is $L_wL(0, \ad_{\fn/w^{-1}\fn})$ up to a unit, the statement for $\Eis$ follows similarly.
\end{proof}

By Lemma \ref{free}, $1\mapsto L(0,\ad_{\fn/w^{-1}\fn})L_w e_w$ defines an isomorphism of functors
\[
w^* \simeq F^{\geq w}/F^{>w}(\CT_{B'} \circ \Eis_B)
\]
as an analog of a classical property of the Mackey filtration.

Now we can prove the main theorem of this subsection, calculating the composition of the inclusion of $\mathrm{id}$ in $\CT_{B}\circ \Eis_{B'}/F^{>1}$ and the adjoint of the intertwining operator.

\begin{theorem}
    \label{mul-sc}
    Let $\chi$ be a character in $X^*(T^\sigma)^{W^\sigma}$.
    Suppose that there is an other Borel subgroup $B'$ in $G$ stable under $\sigma$ and containing $T$. 
    The set of natural transformation of functors over the category $D\Coh(\BL_{T\sigma})_\chi$, viewed as an $\CO(T_{\sigma})=\End(\mathrm{id})$-module acting on the source
    (or equivalently on the target)
    \[
    \Hom(\mathrm{id}, \CT_{B}\circ \Eis_{B'}/F^{>1})
    \]
    is free and of rank $1$. Moreover, there exists a generator of this module such that the composition
    \[
    \mathrm{id} \to \CT_{B}\circ \Eis_{B'}/F^{>1} \to \mathrm{id}
    \]
    is the multiplication by $L(0, \ad_{\fn_{B}/\fn_{B'}})^{-1}$,
    where the morphism 
    \[
    \CT_{B}\circ \Eis_{B'} \to \mathrm{id}
    \]
    is given by the adjunction of the intertwining operator $j_{B|B'} \colon \Eis_{B'} \to \Eis_{B}$.
\end{theorem}

\begin{proof}
    We firstly consider the case for $\chi=1$.
    There is a unique $w\in W^\sigma$ such that $B'=w^{-1}(B)$. 
    As $\CO_{T_\sigma}$ is a generator of the category $D\Coh(\BL_{T\sigma})_1$ and $\End(\CO_{T_\sigma})\simeq \CO(T_\sigma)$, 
    \[
    \Hom(\mathrm{id}, \CT_{B}\circ \Eis_{B'}/F^{>1})
    \simeq \Hom_b(\CO(T_\sigma), \Gamma(\CT_{B}\circ \Eis_{B'}\CO_{T_\sigma})/F^{>1})
    \simeq \Hom_b(\CO(T_\sigma)_w, \Gamma(\CT_{B} \circ \Eis_B\CO_{T_\sigma})/F^{>w}),
    \]
    where the ``$\Hom_b$'' here denotes the bimodule homomorphism space and the second isomorphism is given by twisting by $w$. 
    
    Recall that we may identify the 
    $k(T_\sigma)$-vector spaces
    \[
    \CO(T_{\sigma}\times_{T\sigma\sslash' N_{G}(T\sigma)} T_{\sigma})
    \otimes_{\CO(T_{\sigma})}k(T_\sigma) = 
    \Gamma(\CT_B\circ \Eis_B\CO_{T_\sigma})\otimes_{\CO(T_{\sigma})}k(T_\sigma)
    \]
    and the sets $\{(1,w)\}_{w\in W^\sigma}$, $\{\delta_w\}_{w\in W^\sigma}$ and 
    $\{(L\delta)_w\}_{w\in W^\sigma}$ are bases of its dual, in which lie
    the $\CO(T_\sigma)$-modules
    \[
    \Hom_{\CO(T_\sigma)}(\CT_{B} \circ \Eis_B\CO_{T_\sigma}, \CO_{T_\sigma})\simeq \End(\Eis_B\CO_{T_\sigma})\ \text{and}\ 
    \Hom_{\CO(T_\sigma)}(\CO(T_{\sigma}\times_{T\sigma\sslash' N_{G}(T\sigma)} T_{\sigma}), \CO(T_\sigma))\simeq 
    \End(\eis_B\CO_{T_\sigma}).
    \]  
    There is a splitting of the $\CO(T_\sigma)$-bimodule
    \[
    \CO(T_{\sigma}\times_{T\sigma\sslash' N_{G}(T\sigma)} T_{\sigma})
    \otimes_{\CO(T_{\sigma})}k(T_\sigma)
    \simeq \bigoplus_{w\in W^\sigma}k(T_\sigma)_w e_w
    \]
    such that $\{e_w\}_{w\in W^\sigma}$ is the dual basis of $\{1,w\}_{w\in W^\sigma}$.

    By Lemma \ref{free}, the subquotients of the filtration $F^{>w}$ are free. Thus
    \[
    \Hom_b(\CO(T_\sigma)_w, \Gamma(\CT_{B} \circ \Eis_B\CO_{T_\sigma})/F^{>w})
    \subset \Hom_b(\CO(T_\sigma)_w, \bigoplus_{w'\in W^\sigma}k(T_\sigma)_{w'} e_{w'}/F^{>w})
    \simeq k(T_\sigma)e_w
    \]
    given by taking its value at $1$. Moreover, by construction of the intertwining operator in \ref{all},
    the adjunction of $j_{B|B'}$, $\CT_{B} \circ \Eis_B\CO_{T_\sigma} \to \CO(T_\sigma)_w$
    is given by $L(1, \ad_{\fn_B/\fn_{B'}})^{-1}(1,w)=L_w(1,w)$. 
    By Lemma \ref{free}, $L(0, \ad_{\fn/w^{-1}(\fn)})^{-1}L_w^{-1} e_w$ is a generator of the free $\CO(T_\sigma)$-module $F^{\geq w}/F^{>w}$ and $F^{\geq w}$ is saturated in the module $\Gamma(\CT_{B} \circ \Eis_B\CO_{T_\sigma})$.
    Thus the homomorphism given by $1 \mapsto L(0, \ad_{\fn/w^{-1}(\fn)})^{-1}L_w^{-1} e_w$ 
    is a generator of the free $\CO(T_\sigma)$-module $\Hom_b(\CO(T_\sigma)_w, \Gamma(\CT_{B} \circ \Eis_B\CO_{T_\sigma})/F^{>w})$ and we conclude the $\chi=1$ case.

    Then we consider the case for general $\chi$. Recall that tensoring by $\CO_\chi$ defines an equivalence of categories
    \[
    D\Coh(\BL_{T\sigma})_1 \simeq D\Coh(\BL_{T\sigma})_\chi.
    \]
    Recall that the construction $\CT_B\circ \Eis_{B'}$ is compatible with this equivalence by Proposition \ref{CT_chi}. The filtration $F^{>1}$ is compatible with this equivalence by definition. The integral intertwining operator $j_{B'|B}$ is compatible with this equivalence by its construction in the proof of Proposition \ref{all}. 
    Thus tensoring the generator of $\Hom(\mathrm{id}, \CT_{B}\circ \Eis_{B'}/F^{>1})_1$ defined above by $\CO_\chi$ defines the desired generator of $\Hom(\mathrm{id}, \CT_{B}\circ \Eis_{B'}/F^{>1})_\chi$,
    where the subscripts mean the $\Hom$-modules over the corresponding block.
\end{proof}

Imitating the construction of the classical unnormalized intertwining operator,
we divide the factor in Theorem \ref{mul-sc} to obtain the unnormalized spectral intertwining operator.

\begin{definition}
    The unnormalized intertwining operator $J_{B'|B}$ is defined as
    \[
    L(0, \ad_{\fn_{B'}/\fn_B})j_{B'|B},
    \]
    viewing as a transformation $\Eis_B^{rs} \to \Eis_{B'}^{rs}$.
\end{definition}

We deduce Theorem \ref{Adjunction} in the unipotent case.

%% file: 3.tex
\section{The general case}

In this section, we construct the spectral intertwining operators for some components of $\Par_G$, generalizing the results on the unipotent block.

\subsection{Structure of general components}

Firstly we describe the properties of the general connected components of the moduli stack of L-parameters. We will show that any connected component is isomorphic to the unipotent component of another group, quotiented by a finite group.

Let $F$ be a $p$-adic local field. Let $W_F$ be its Weil group and $I = I_F$ be the inertia subgroup of $W_F$. Denote by $\Fr$ the \textit{arithmetic} Frobenius in $W_F$.

Let $G$ be a reductive group over the algebraically closed field $\bar{\mathbb{Q}}_\ell$ ($\ell\neq p$) with $W_F$-action preserving a pinning. Note that it
comes from the Langlands dual of some $p$-adic reductive group 
(Attention: we denote by $G$ the group on the spectral side, instead of $G^\vee$). 
Let $Z^1(W_F, G)$ be the moduli space of $\ell$-adically continuous 1-cocycles. By \cite{fargues2024}, Proposition VIII.2.5, we may view this space as the moduli space of Weil--Deligne L-parameters, i.e., space of the pairs $(\phi, N)$, where $\phi$ is a continuous 1-cocycle in discrete topology and $N \in \CN_G$ is nilpotent (called the monodromy part), satisfying $\ad(\phi(w))N = q^{|w|}N$ for all $w\in W_F$
($|\cdot|\colon W_F\to W_F/I_F\simeq \mathbb{Z}$, $\Fr\mapsto 1$).
We will denote by $Z^1(W_F,G)_0$ the moduli space of continuous 1-cocycle in discrete topology and we have a projection $Z^1(W_F,G) \to Z^1(W_F,G)_0$
by ignoring the monodromy part.

Let $M$ be a Levi subgroup of $G$ stable under the $W_F$-action and
$\phi \in Z^1(W_F, M)$ be an elliptic semisimple parameter. Let $P$ be a parabolic subgroup of $G$ stable under $W_F$ with Levi subgroup $M$.

Denote by $\Par_{G, \phi}$ the connected component of $\Par_{G} := Z^1(W_F, G)/G$ containing $\phi$ and denote by $\Par_{G,\phi,0}$ the connected component of $\Par_{G,0} := Z^1(W_F, G)_0/G$ containing $\phi$. Similarly for $P$ and $M$. We will denote by ${}^L(-)$ the L-group 
$(-)\rtimes W_F$ for a subgroup of $G$ stable under $W_F$. 
Then $\phi \in Z^1(W_F, M)$ corresponds to a continuous group homomorphism ${}^L\phi\colon W_F\to {}^LM$.

The spectral Eisenstein series operator in general is defined as follows.

\begin{definition}[Spectral Eisenstein series]
    Notations as above, define
    $p_P \colon \Par_{P} \to \Par_{G}$ induced by the inclusion and 
    $q_P \colon \Par_{P} \to \Par_{M}$ induced by the projection. The spectral Eisenstein operator $\Eis_P$ is defined as 
    \[
        \Eis_P := p_{P*} \circ q_P^* \colon D\Coh(\Par_{M}) \to D\Coh(\Par_{G}).
    \]
    Similarly, we define $p_{P,0} \colon \Par_{P,0} \to \Par_{G,0}$ and $q_{P,0} \colon \Par_{P,0} \to \Par_{M,0}$.
    Then $\eis_P$ is defined as
    \[
        \eis_P := p_{P,0*} \circ q_{P,0}^* \colon D\Coh(\Par_{M}) \to D\Coh(\Par_{G}).
    \]
    We will also consider the morphism and operators over the connected components $\Par_{M,\phi}$, $\Par_{P,\phi}$ and $\Par_{G,\phi}$ and use the same notations if there is no ambiguity, such as
    $p_P \colon \Par_{P,\phi} \to \Par_{G,\phi}$, $q_P \colon \Par_{P,\phi} \to \Par_{M,\phi}$ and
    \[
    \Eis_P := p_{P*} \circ q_P^* \colon D\Coh(\Par_{M,\phi}) \to D\Coh(\Par_{G,\phi}).
    \]
\end{definition}

We reduce to the case for the principal block by the following facts that any connected component of $Z^1(W_F, G)$ is similar to the unramified component of some group. Some of these facts are mentioned in a more general way in \cite{dat2025}. 

In the following statements, we will consider the centralizer of a parameter such as $C_G(\phi|_I)$, which always means the centralizer of the corresponding group homomorphism, such as $C_G({}^L\phi|_I)$.

\begin{lemma}(\cite{dat2025}, Theorem 3.12)
    \label{pinning}
    After an unramified twist of $\phi$, there is a pinning of $C_G^\circ(\phi|_I)$ preserved by the adjoint action of ${}^L\phi(\Fr)$
    (such action is well-defined as $I$ is normal in $W_F$).
\end{lemma}

\begin{lemma}
    The connected component of the (twisted) centralizer of $\phi|_I$ (i.e., the restriction of $\phi$ to the inertia subgroup $I_F$) in $M$, $C_M^\circ(\phi|_I)$ is a torus. Moreover, $C_G^\circ(\phi|_I)$ is a reductive group and $C_P^\circ(\phi|_I)$ is a Borel subgroup in it.
\end{lemma}

\begin{proof}
    As $I$ is a normal subgroup of $W_F$, $\phi|_I$ is also semisimple.
    Thus we conclude that $C_G^\circ(\phi|_I)$ is a reductive group, $C_P^\circ(\phi|_I)$ is a parabolic subgroup in it and $C_M^\circ(\phi|_I)$ is a corresponding Levi subgroup in it. Hence it suffices to prove that $C_M^\circ(\phi|_I)$ is a torus.

    By Lemma \ref{pinning}, we may twist $\phi$ such that $\sigma=\ad({}^L\phi(\Fr))$ preserves a pinning of $C_M^\circ(\phi|_I)$. Moreover, the fixed points of $C_M^\circ(\phi|_I)$ under $\sigma$ is $C_M(\phi)$.
    As $\phi$ is elliptic in $M$, the group $C_M(\phi)/Z_M$ is finite. We may replace $M$ by its quotient $M/Z_M$. Then the centralizer $C_M(\phi)$ is finite.
    Suppose that $C_M^\circ(\phi|_I)$ is not a torus. We have that $\sigma$ permutes a pinning $\{X_\alpha\}$ of it. 
    Thus $\sigma$ fixes the vector $0\neq X := \sum_{\alpha\in O} X_\alpha$ for any orbit $O$ of simple roots of $C_M^\circ(\phi|_I)$ under the $\sigma$-action and it will fix the one-parameter subgroup $\exp(tX)$, a contradiction.
\end{proof}

From now on, we assume that $\phi$ satisfies that the adjoint action of ${}^L\phi(\Fr)$ on $C_G(\phi|_I)^\circ$ fixes a pinning of the Borel pair $(C_P(\phi|_I)^\circ, C_M(\phi|_I)^\circ)$, by replacing it by an unramified twist. We will denote the adjoint action of ${}^L\phi(\Fr)$ by $\sigma$. Note that these constructions are independent of the choice of $\Fr$.

\begin{proposition}
\label{reduction}
    The connected component $\Par_{G, \phi}$ is isomorphic to $\tilde{\BL}_{C_G^\circ (\phi|_I)\sigma}/C_G^1 (\phi|_I)$, where $C_G^1 (\phi|_I)$ is the open and closed subgroup of $C_G(\phi|_I)$ that stabilizes $C_G^\circ(\phi|_I)$ under the $\sigma$-twisted conjugation (equivalently saying, $\pi_0(C_G^1 (\phi|_I))=\pi_0(C_G (\phi|_I))^\sigma$).
    Similarly for $P$ and $M$. 
    Moreover we have $\Par_{G,\phi,0} \simeq C_G^\circ (\phi|_I)\sigma/C_G^1 (\phi|_I)$.
    
    More precisely, the isomorphism is induced by
    \[
    \tilde{\BL}_{C_G^\circ (\phi|_I)\sigma} \to \Par_{G,\phi},
    \quad (g\sigma, N) \mapsto (\chi_g\phi_0, N)
    \]
    where $\chi_g$ denotes the unipotent parameter $W_F\to 
    C_G^\circ (\phi|_I)\sigma$ sending $\Fr$ to $g$.
\end{proposition} 

\begin{proof}
    We view $\Par_G$ as the moduli space of Weil--Deligne L-parameters $\Par_G^{\text{WD}}$. 
    As $I_F$ is profinite, $(\phi_1, N_1)$, $(\phi_2, N_2)$ are in the same connected component only if $\phi_1|_I$ and $\phi_2|_I$ are $G$-conjugate. Thus we have
    \[
    \Par_{G,\phi} \subset G\times^{C_G(\phi|_I)}\{(\phi', N')\in \Par_G^{\text{WD}}\mid \phi'|_I=\phi|_I\}.
    \]
    We have an isomorphism of schemes (compare the arguments in \cite{dat2025}, Section 3.1)
    \[
    \{(\phi_0, N)\mid \phi_0\in Z^1_{\ad(\phi)}(W_F/I_F, C_G(\phi|_I)), 
    N \in \CN_{C_G^\circ(\phi|_I)}, \ad(\phi_0(w))N = q^{|w|}N\} \simeq
    \{(\phi', N')\in \Par_G^{\text{WD}}\mid \phi'|_I=\phi|_I\},
    \]
    via $(\phi_0, N) \mapsto (\phi_0\phi, N)$, where $\ad(\phi|_I)$ defines a $W_F/I_F$-action on $C_G(I_F)$. 

    Thus the connected component of $\{(\phi', N')\in \Par_G^{\text{WD}}\mid \phi'|_I=\phi|_I\}$ containing $\phi|_I$ is 
    \[
    \{(\phi_0, N)\mid \phi_0\in Z^1_{\ad(\phi)}(W_F/I_F, C_G^\circ(\phi|_I)), 
    N \in \CN_{C_G^\circ(\phi|_I)}, \ad(\phi_0(w))N = q^{|w|}N\},
    \]
    which is exactly $\tilde{\BL}_{C_G^\circ (\phi|_I)\sigma}$. Thus we conclude. The statement for $P$, $M$ and $\Par_{G,\phi,0}$, etc. follows similarly.
\end{proof}

From now on, we will denote by $M_\phi := C_M(\phi|_I)$,
$P_\phi:=C_P(\phi|_I)$, $P'_\phi:=C_{P'}(\phi|_I)$ and
$G_\phi:= C_G(\phi|_I)$ and similarly for the groups with superscripts $\circ$ and $1$. We have the following corollary.

\begin{corollary}
\label{to unipotent}
    The morphisms $\Par_{M,\phi}\leftarrow \Par_{P,\phi}\to \Par_{G,\phi}$ are isomorphic to 
    \[
    \tilde{\BL}_{M^\circ_\phi\sigma}/M_\phi^1 \leftarrow 
    \tilde{\BL}_{P^\circ_\phi\sigma}/P_\phi^1 \to
    \tilde{\BL}_{G^\circ_\phi\sigma}/G_\phi^1,
    \]
    where the morphisms are induced by the natural morphisms $M_\phi^1\leftarrow P^1_\phi\to G^1_\phi$ and morphisms between their connected components.

    Moreover, denote by
    \[
    \overline{q_{P_\phi^\circ}} \colon \tilde{\BL}_{P_\phi^\circ\sigma}/P^1_\phi\to \tilde{\BL}_{M_\phi^\circ\sigma}/M^1_\phi,
    \]
    \[
    \overline{p_{P_\phi^\circ}} \colon \tilde{\BL}_{P_\phi^\circ\sigma}/P^1_\phi\to \tilde{\BL}_{G_\phi^\circ\sigma}/M^1_\phi G_\phi^\circ,
    \]
    and the projection
    \[
    \pi_{\phi} \colon \tilde{\BL}_{G_\phi^\circ\sigma}/M^1_\phi G_\phi^\circ\to \tilde{\BL}_{G_\phi^\circ\sigma}/G^1_\phi.
    \]
    Then the Eisenstein operator $\Eis_P$
    \[
    \Eis_P = p_{P*}q^*_P \colon D\Coh(\Par_{M,\phi})\to D\Coh(\Par_{G, \phi})
    \]
    can be identified with $\pi_{\phi*}\circ\overline{\Eis_{P_\phi^\circ}}$,
    where
    $\overline{\Eis_{P_\phi^\circ}} := \overline{p_{P_\phi^\circ}} \circ \overline{q_{P_\phi^\circ}}$.
    Similarly, we have $\eis_P\simeq \pi_{\phi,0*}\circ\overline{\eis_{P_\phi^\circ}}$ where the operators are naturally defined.
\end{corollary}

As $M_\phi^\circ$ is a torus, $\tilde{\BL}_{M_\phi^\circ\sigma}=M_\phi^\circ\sigma$.
In particular, $\Par_{M,\phi,0}=\Par_{M,\phi}$ and the source of $\Eis_P$ and $\eis_P$ are identical.

Next we analyze the structure of the groups $M_\phi$, $P_\phi$ and $G_\phi$.

\begin{lemma}
    The maps between connected components satisfy that  $\pi_0(M^1_\phi)\simeq \pi_0(P^1_\phi)$ and $\pi_0(P^1_\phi)\to \pi_0(G^1_\phi)$ is injective.
\end{lemma}

\begin{proof}
    The first argument follows from the fact that $M$ is a deformation retract of $P$. For the second argument, it suffices to prove that
    $\pi_0(P_\phi)\to \pi_0(G_\phi)$ is injective.
    Note that $P_\phi$ normalizes $P_\phi$ in $G_\phi$. Thus we have that $P_\phi\cap G_\phi^\circ$ normalizes $P_\phi^\circ$ in $G_\phi^\circ$. As $P^\circ_\phi$ is a Borel subgroup of $G_\phi^\circ$, the normalizer of $P_\phi^\circ$ in $G_\phi^\circ$ is itself. Thus $P_\phi\cap G_\phi^\circ = P_\phi^\circ$ and we conclude.
\end{proof}

To analyze the structure of $G_\phi^1$, we define the following notion of ``normalizer''.

\begin{definition}[Normalizer of $(M, \phi)$]
\label{ngmphi}
    Let $N_G(M,\phi)$ be the subgroup of $G$ defined as 
    \[\{g \in N_G({}^LM)\mid \ad(g)\phi \in X_{M_\phi^\circ}\phi\}.\]
    Here $X_{M_\phi^\circ}\phi$, viewed as a subscheme of $Z^1(W_F,M)$ denotes the unramified twists of $\phi$ by characters $W_F/I\to M_\phi^\circ$.
\end{definition}

By definition we have 
$N_G(M, \phi)\cap M = M^1_\phi$ and $N_G(M, \phi)\subset G^1_\phi$.
In particular, $M_\phi^1$ is a normal subgroup of $N_G(M,\phi)$ of finite index.
We have another interpretation of the normalizer of $(M,\phi)$.

\begin{lemma}
\label{another ngmphi}
    The subgroup $N_G(M,\phi)$ in $G$ can be identified with the subgroup
    \[
    N_{G_\phi^1}(M_\phi^\circ\sigma)=\{g \in G\mid \ad(g)(X_{M_\phi^\circ}\phi)\subset X_{M_\phi^\circ}\phi\} \subset G.
    \]
    Note that these two groups are equal under the identification
    \[
    M_\phi^\circ\sigma \simto X_{M_\phi^\circ}\phi,
    \quad m\sigma \mapsto \chi_m\phi,
    \]
    where $\chi_m$ is the unipotent parameter $W_F\to M_\phi^\circ$ sending $\Fr$ to $m$. 
\end{lemma}

We remark that the condition $\ad(g)(X_{M_\phi^\circ}\phi)\subset X_{M_\phi^\circ}\phi$ is equivalent to $\ad(g)(X_{M_\phi^\circ}\phi)= X_{M_\phi^\circ}\phi$, as $\ad(g)$ defines an automorphism of $X_{G_\phi^\circ} \phi$.

\begin{proof}
    For any $g \in N_G(M,\phi)$ and any $\phi'\in X_{M_\phi^\circ}\phi$, the image of $\ad(g)(\phi')$ is in $M$. Moreover, $\ad(g)\phi|_I=\phi|_I$ and thus $\ad(g)(X_{M_\phi^\circ}\phi)\subset X_{M_\phi}\phi$. Thus $\ad(g)(X_{M_\phi^\circ}\phi)\subset X_{M_\phi^\circ}\phi$ by connectivity. Hence $N_G(M,\phi)\subset N_{G_\phi^1}(M_\phi^\circ\sigma)$.

    It remains to prove that $N_{G_\phi^1}(M_\phi^\circ\sigma)\subset N_G(M,\phi)$. For $g \in N_{G_\phi^1}(M_\phi^\circ\sigma)$, the requirement $\ad(g)\phi\in X_{M_\phi^\circ}\phi$ is trivial. 
    Then we prove the statement $g \in N_G({}^LM)$. By definition $g$ normalizes the subgroup
    \[
    M_\phi^\circ\im({}^L\phi) \subset {}^LG.
    \]
    This subgroup is generated by $\im({}^L(\phi|_I))$, $\sigma={}^L\phi(\Fr)$ and $M_\phi^\circ$. We claim that its center is $(M_\phi^\circ)^\sigma$. Indeed, we have an exact sequence
    \[
    1 \to M_\phi^\circ \to M_\phi^\circ\im({}^L\phi)\to W_F \to 1
    \]
    and the center of $W_F$ is trivial. Then the center of $M_\phi^\circ\im({}^L\phi)$ is contained in $M_\phi^\circ$ and then in $M_\phi^\circ$ as it centralizes ${}^L\phi(\Fr)$.
    
    Hence $g$ normalizes the connected component $(M_\phi^\sigma)^\circ$ of the center. As $\phi \in Z^1(W_F, M)$ is elliptic, we have
    \[
    (M_\phi^\sigma)^\circ = (Z(M)^{W_F})^\circ.
    \]
    Thus $g$ normalizes the centralizer
    \[
    C_{{}^LG}((Z(M)^{W_F})^\circ) \simeq {}^LM
    \]
    where this identity follows from \cite{Borel}, Lemma 3.5 and we conclude.
\end{proof}

\begin{lemma}[Structure of $G_\phi^1$]
\label{gphi1}
    The group $G_\phi^1$ can be written as the product $G_\phi^\circ N_G(M,\phi)$.
\end{lemma}

\begin{proof}
    Consider the semidirect product $G_\phi^0\rtimes \langle \sigma \rangle$.
    For an element $g\in G_\phi^1$, $\ad(g)$ induces an automorphism of $G_\phi^\circ\rtimes \langle \sigma \rangle$. Thus in this group it sends the minimal Levi subgroup $M_\phi^\circ\rtimes \langle \sigma \rangle$ to a minimal Levi subgroup. Here the Levi subgroups are in the sense of \cite{Borel}, 3.4. Hence there exists $g_0 \in G_\phi^\circ$ satisfying 
    \[
    \ad(g_0g)(M_\phi^\circ\rtimes \langle \sigma \rangle)\subset M_\phi^\circ\rtimes \langle \sigma \rangle 
    \]
    and then $g_0g\in N_G(M,\phi)$ by Lemma \ref{another ngmphi}.
\end{proof}

\begin{proposition}[Structure of $N_G(M,\phi)/M_\phi^1$]
\label{semidirect}
    $N_G(M,\phi)/M_\phi^\circ$ can be expressed as a semidirect product
    \[
    N_G(M,\phi)/M_\phi^\circ \simeq 
    \pi_0(G^1_\phi) \ltimes 
    N_{G_\phi^\circ}(M^\circ_\phi\sigma)/M_\phi^\circ,
    \]
    such that the elements in $\pi_0(G_\phi^1)$ stabilize $P_\phi^\circ$.
    Note that the second component is isomorphic to $W(G_\phi^\circ)^\sigma$, the $\sigma$-fixed points of the Weyl group of $G^\circ_\phi$.

    In particular, we have
    \[
    N_G(M,\phi)/M_\phi^1 \simeq \pi_0(G^1_\phi/M^1_\phi)\ltimes W(G_\phi^\circ)^\sigma ,
    \]
\end{proposition}

\begin{proof}
    As $N_G(M, \phi)= N_{G_\phi^1}(M_\phi^\circ\sigma)$ by Lemma \ref{another ngmphi}, we have
    \[
    (N_G(M, \phi)\cap G_\phi^\circ)/M_\phi^\circ\simeq N_{G_\phi^\circ}(M_\phi^\circ\sigma)/M_\phi^\circ \simeq W(G_\phi^\circ)^\sigma.
    \]
    By Lemma \ref{gphi1}, $N_G(M, \phi)/M_\phi^\circ$ surjects to $\pi_0(G_\phi^1)$. Thus it suffices to find a section as a group homomorphism of this surjection satisfying the condition about $P_\phi^\circ$.
    
    For $\gamma \in N_G(M,\phi)/M_\phi^\circ$, choose arbitrarily a preimage $\tilde{\gamma}\in N_G(M,\phi)$. Then the torus $M_\phi^\circ$ is stable under the $\sigma$-twisted $\tilde{\gamma}$-conjugation on $G_\phi^\circ$ and thus $\tilde{\gamma}(P^\circ_\phi)\sigma(\tilde{\gamma})^{-1}$ is a $\sigma$-stable standard Borel subgroup. This Borel subgroup is independent of the choice of $\tilde{\gamma}$ and we will denote it simply by $\gamma(P^\circ_\phi)\sigma(\gamma)^{-1}$.
    
    There exists a unique element $\gamma_0$ in $N_{G_\phi^\circ}(M^\circ_\phi\sigma)/M_\phi^\circ\simeq W(G_\phi^\circ)^\sigma$ such that $\gamma_0\gamma P^\circ_\phi\sigma(\gamma_0\gamma)^{-1} = P^\circ_\phi$. Thus $\gamma_1:=\gamma_0^{-1}\gamma$ is the unique element in the coset $\gamma W(G_\phi^\circ)^\sigma$ such that $P_\phi^\circ$ is stable under the $\sigma$-twisted $\gamma_1$-conjugation. Thus we obtain a section 
    \[
    \pi_0(G^1_\phi) 
    \to N_G(M,\phi)/M_\phi^\circ, \quad \gamma W(G_\phi^\circ)^\sigma \mapsto \gamma_1.\] 
    and the section is a group homomorphism by the uniqueness above.
\end{proof}

\begin{lemma}
\label{generical free}
    The morphism $\tilde{\BL}_{M_\phi^\circ\sigma}/M_\phi^1\to \tilde{\BL}_{M_\phi^\circ\sigma}\sslash M_\phi^1$ is a homeomorphism.
    Moreover, the $N_G(M,\phi)/M_\phi^1$-action on $\tilde{\BL}_{M_\phi^\circ\sigma}\sslash M_\phi^1$ is free over a dense open subset of it.
\end{lemma}

\begin{proof}
    The group $(M_\phi^\circ)^\sigma$ is a subgroup of finite index
    in the stabilizer of the $M_\phi^1$-action of any point $x \in M_\phi^\circ\sigma$.
    Thus the isotropy groups of all points in 
    $\tilde{\BL}_{M_\phi^\circ\sigma}/M_\phi^1$ have the same dimension. Thus all orbits in $\tilde{\BL}_{M_\phi^\circ\sigma}/M_\phi^1$ are closed and the morphism from $\tilde{\BL}_{M_\phi^\circ\sigma}/M_\phi^1$ to its coarse moduli space induces a homeomorphism.

    Consider the $N_G(M,\phi)$-action on $M_\phi^\circ\sigma \simeq X_{M_\phi^\circ}\phi$, the subscheme of unramified twists of $\phi$. Suppose that $g \in N_G(M,\phi)$ acts trivially on this scheme. As $g$ fixes $\phi$, $g \in C_G({}^L\phi)$. For $m \in M_\phi^\circ$, denote by $\chi_m$ the unipotent parameter $W_F\to M_\phi^\circ$ sending $\Fr$ to $m$. As $g$ fixes $\chi_m\phi$, $g \in C_G(m)$. Varying $m$, we obtain that $g$ centralizes $M_\phi^\circ$. As $(Z(M)^{W_F})^\circ \subset M_\phi^\circ$, we have
    \[
    C_{{}^LG}(M_\phi^\circ) \subset C_{{}^LG}((Z(M)^{W_F})^\circ) 
    = {}^LM
    \]
    where the identity is from \cite{Borel}, Lemma 3.5. Thus $g \in N_G(M,\phi)\cap M = M_\phi^1$.

    We conclude that the $N_G(M,\phi)/M_\phi^1$-action on $\tilde{\BL}_{M_\phi^\circ\sigma}\sslash M_\phi^1$ cannot factor through a proper quotient group. Thus the action is free over a dense open subset since $N_G(M,\phi)/M_\phi^1$ is a finite group.
\end{proof}

\subsection{Intertwining over the regular semisimple and generic locus}

We will define the regular semisimple locus and the generic locus in general and construct the normalized intertwining operator in this subsection. We retain the notations of the previous section.

\begin{definition}[L-function]
\label{L-function,general}
    For $s=0,1$ and $\psi \in \Par_{M, \phi}$, the L-function, as a rational function on $\Par_{M,\phi}$, is defined as
    \[
    L(s, \ad_{\fn_P})(\psi) := \det(1 - q^{-s} {}^L\psi(\Fr), \fn_P^{\phi(I)})^{-1}.
    \]
\end{definition}

As $\fn_{P_\phi^\circ} = \fn_P^{\phi(I)}$, we have $L(s, \ad_{\fn_P})(\chi_m\phi) = L(s, \ad_{\fn_{P_\phi^\circ}})(m)$ for any $m\in M_\phi^\circ\sigma$ (where $\chi_m$ is the unramified $1$-cocycle corresponding to $m$). In particular, the $M^1_\phi$-action on $M_\phi^\circ\sigma$ preserves the L-function $L(s, \ad_{\fn_{P_\phi^\circ}})$. Thus we may view the L-function as a rational function on $\tilde{\BL}_{M_\phi^\circ\sigma}/M_\phi^1$.

\begin{definition}[Regular semisimple locus and generic locus]
    \label{rs,g,intrinsic}
    We define the regular semisimple or generic loci as follows:
    \begin{itemize}
        \item For $\Par_{M,\phi}$, consider the cotangent complex $L\iota_{M,G}$ of the morphism $\iota_{M,G}\colon \Par_{M,\phi}\to \Par_{G,\phi}$ induced by $M \to G$. Then the generic locus $\Par_{M,\phi}$ is defined to be where $L^1\iota_{M,G}=0$ and the regular and semisimple locus is defined to be where $L^{-1}\iota_{M,G}=0$.
        \item For $\Par_{P,\phi}$, the regular semisimple or generic locus of it is defined to be the pullback of the corresponding locus of $\Par_{M,\phi}$ along $q_P$.
        \item For $\Par_{G,\phi}$, the regular semisimple or generic locus of it is defined to be the image of the corresponding locus of $\Par_{P,\phi}$ along $p_P$, as it is so in the unipotent case.
    \end{itemize}
    Similarly we define the loci for $\Par_{M,\phi,0}$, etc. for the parameters without monodromy.
\end{definition}

Note that the loci on $\Par_{M,\phi}$ depends on the ambient group $G$. We will always take $G$ as the ambient group in the definition of the loci. 

We remark that the regular semisimple and generic locus $\Par_{G,\phi}^{rs,g}\subset \Par_{G,\phi}$ here can be identified with the open substack of parameters of ``Langlands-Shahidi type'' in the sense of Zou (\cite{zou2025categoricallocallanglandsmathrmgln}) and also the open substack of ``generous'' L-parameters in the sense of Hansen \cite{hansen_beijing}.

\begin{lemma}
\label{rs,g,stable}
    The regular, semisimple or generic locus in $G_\phi^\circ\sigma$
    (resp. $P_\phi^\circ\sigma$, $M_\phi^\circ\sigma$) is invariant under the $G_\phi^1$ (resp. $P_\phi^1$, $M_\phi^1$)-action. 
    Similarly for $\tilde{\BL}_{G_\phi^\circ\sigma}$, etc.
\end{lemma}

\begin{proof}
    We prove the statement for $G_\phi^\circ\sigma$ and the statements for 
    $P_\phi^\circ\sigma$ and $M_\phi^\circ\sigma$
    follow similarly.

    For the statement about the regular locus, note that for $x\sigma\in G_\phi^\circ$, the stabilizers satisfy that $C_{G_\phi^\circ}(\ad(g)(x\sigma))= \ad(g)(C_{G_\phi^\circ}(x\sigma))$.
    Recall that the regularity of $x$ is defined by the dimension of its stabilizer.
    Thus the regular locus is stable under its action.

    For the statement about the semisimple locus, note that $x\sigma\in G_\phi^\circ\sigma$ is semisimple if and only if $\chi_x\phi$ is semisimple as an L-parameter, where $\chi_x$ denotes the $1$-cocycle on $W_F/I$ corresponding to $x$. 
    As any conjugation of a semisimple L-parameter is again semisimple, the semisimple locus is stable under conjugation.

    For the statement about the generic locus, 
    note that a semisimple element $x\sigma\in G_\phi^\circ\sigma$ is generic if and only if $(x\sigma,N)\in \BL_{G_\phi^\circ}$ implies $N = 0$. Thus the $G_\phi^1$-action preserves the semisimple and generic locus. 
    Moreover, recall that the genericity is defined via the twisted Chevalley morphism. Thus it is a condition on 
    $G_\phi^\circ\sigma\sslash G_\phi^\circ$ to which the $G_\phi^1$ action be descended. As any element in $G_\phi^\circ\sigma\sslash G_\phi^\circ$ has a semisimple representative, the generic locus is stable under conjugation.

    The statements about $\tilde{\BL}_{G_\phi^\circ\sigma}$ hold as the loci are defined by pulling back the corresponding loci along the projection $\tilde{\BL}_{G_\phi^\circ\sigma} \to G_\phi^\circ\sigma$.
    Similarly for $\tilde{\BL}_{P_\phi^\circ\sigma}$ and $\tilde{\BL}_{M_\phi^\circ\sigma}$.
\end{proof}

From the lemma above, we can identify these loci on $\Par_{G,\phi}$,
$\Par_{P,\phi}$ and $\Par_{M,\phi}$ with the corresponding loci in the unipotent case.

\begin{proposition}
    Under the identification in Proposition \ref{reduction}, the regular semisimple locus or the generic locus of $\Par_{M,\phi}$ 
    (resp. $\Par_{P,\phi}$, $\Par_{G,\phi}$) can be identified with 
    the corresponding locus of
    $\tilde{\BL}_{M_\phi^\circ\sigma}/M^1_\phi$ (resp. $\tilde{\BL}_{P_\phi^\circ\sigma}/P^1_\phi$, $\tilde{\BL}_{G_\phi^\circ\sigma}/G^1_\phi$).
    Similarly for the parameters without monodromy, such as $\Par_{G,\phi,0}$, we define the regular semisimple locus.
\end{proposition}

In particular, the loci of $\Par_{G,\phi}$ are open.

\begin{proof}
    The statement for $\Par_{M,\phi}$ follows from Proposition \ref{complex} about the cotangent complex in the unipotent case. The statements $\Par_{P,\phi}$ follows by definition. The statement for $\Par_{G,\phi}$ holds since $\BL_{G_\phi^\circ\sigma}^?$ is the image of $\BL_{M_\phi^\circ\sigma}^?$ and the ?-locus is invariant under conjugation, where ? is $rs$, $g$ or $rs, g$.  
\end{proof}

\begin{lemma}
\label{iota-G-general}
    The morphisms of quotient stacks
    \[
    (M_\phi^\circ\sigma)^{rs}/N_G(M,\phi)
    \to (G_\phi^\circ\sigma)^{rs}/G^1_\phi.
    \]
    and
    \[
    \tilde{\BL}_{M_\phi^\circ\sigma}^{rs,g}/N_G(M,\phi) \to
    \tilde{\BL}_{G_\phi^\circ\sigma}^{rs,g}/G_\phi^1 
    \]
    are isomorphisms. 
\end{lemma}

\begin{proof}
    We prove the statement without monodromy and the statement with monodromy follows directly.
    It suffices to prove the statement after pulling back along
    \[
    (G_\phi^\circ\sigma)^{rs} \to (G_\phi^\circ\sigma)^{rs}/G_\phi^1,
    \]
    under which the pullback of $(M_\phi^\circ\sigma)^{rs}/N_G(M,\phi)$ is
    $G_\phi^1 \times^{N_G(M,\phi)} (M_\phi^\circ\sigma)^{rs}$.
    By Lemma \ref{iota-G}, we have an isomorphism
    \[
    G_\phi^\circ \times^{N_{G_\phi^\circ}(M^\circ_\phi\sigma)} (M_\phi^\circ\sigma)^{rs}\simeq G_\phi^\circ.
    \]
    Moreover, by Lemma \ref{gphi1}, we have 
    \[
    G_\phi^\circ N_G(M,\phi) = G_\phi^1\quad \text{and}\quad G_\phi^\circ\cap N_G(M,\phi) = N_{G_\phi^\circ}(M^\circ_\phi\sigma)
    \]
    Thus the morphism
    \[
    G_\phi^\circ \times^{N_{G_\phi^\circ}(M^\circ_\phi\sigma)} (M_\phi^\circ\sigma)^{rs} \to 
    G_\phi^1 \times^{N_G(M,\phi)} (M_\phi^\circ\sigma)^{rs}
    \]
    is an isomorphism and we conclude.
\end{proof}

We can generalize the notion of regular semisimple and generic locus on the entire $\Par_M$, $\Par_P$ and $\Par_G$.

\begin{definition}[Regular semisimple and generic locus in general]
    Let $\psi \in \Par_M$ be a semisimple parameter. 
    Suppose that ${}^LM_0\subset {}^LM$ is a minimal Levi subgroup that contains $\psi$. Then $\psi$ is elliptic in $M_0$ and we define the regular semisimple and generic locus on $\Par_{M,\psi}^{rs,g}$ to be the image $\iota_{M_0,M}(\Par_{M_0,\psi}^{rs,g})$ under the morphism $\iota_{M_0,M} \colon \Par_{M_0,\psi}\to \Par_{M,\psi}$. 
    Note that for $\psi\in M$, all such $M_0$'s are $M$-conjugate. 
    Thus the locus $\Par_{M,\psi}^{rs,g}$ is independent of the choice of $M_0$.
    Moreover, by definition $\Par_{M,\psi}^{rs,g}\simeq \BL_{M_\psi^\circ\sigma}^{rs,g}/M_\psi^1$. Thus it is open. 
    Combining these loci for all $\psi$'s, we obtain an open substack $\Par_M^{rs,g}$ for $\Par_M$.

    By similar process, we define the locus $\Par_G^{rs,g}$ for $G$. For $P$, the locus $\Par_P^{rs,g}$ is defined to be the pullback $q_P^{-1}(\Par_M^{rs,g})$. 
\end{definition}

\begin{lemma}
\label{rs,g,locus,general}
    The morphisms $\iota_{M,P}$ (the morphism induced by the inclusion $M\subset P$) and $q_P$ define an isomorphism $\Par_M^{rs,g}\simeq \Par_P^{rs,g}$. Moreover, the regular semisimple and generic loci for $P$ and $G$ are stable under $p_P$ and $p_P^{-1}$. 
\end{lemma}

\begin{proof}
    For both statements, it suffices to consider each connected component $\Par_{M,\psi}$, for semisimple parameters $\psi$. Choose a minimal Levi subgroup of $M$ containing $\psi$. 
    
    For any $\psi'\in \Par_{M,\psi}^{rs,g}$, choose a preimage $\psi'_0$ of it in $\Par_{M_0,\psi}^{rs,g}$. By definition, $\iota_{M_0,P}$ and $\iota_{M_0,M}$ are \'etale at $\psi'_0$. Thus $\iota_{M,P}$ is \'etale at $\psi'$. Hence $\iota_{M,P}$ is \'etale on $\Par_M^{rs,g}$. Since $\iota_{M,P}$ is the right inverse of $q_P$ and $\Par_{P,\psi}$ is connected, $\iota_{M,P}$ and $q_P$ are mutually inverse over the regular semisimple and generic locus.

    As $\iota_{M,G}=p_P\circ \iota_{M,P}$, $p_P(\Par_{P,\psi}^{rs,g})=\Par_{G,\psi}^{rs,g}$. Moreover, for $\psi'\in \Par_{G,\psi}^{rs,g}$, it is semisimple and without monodromy. Thus for any $\tilde{\psi} \in q_P^{-1}(\Par_{P,\psi})$,
    it comes from $\Par_{M_0,\psi}$. As $\iota_{M_0,G}^{-1}(\Par_{G,\psi}^{rs,g})=\Par_{M_0,\psi}^{rs,g}$ as in the unipotent case, $\tilde{\psi} \in \Par_{P,\psi}^{rs,g}$. 
\end{proof}

Hence the loci are stable under $p_P$, $p_P^{-1}$, $q_P$ and $q_P^{-1}$.
as in the unipotent case. We will use the superscript $rs$ and $g$ to denote the regular semisimple or generic locus respectively.
Then we define
\[\Eis_P^?\colon D\Coh(\Par_{M,\phi}^?) \to D\Coh(\Par_{G,\phi}^?)\]
and 
\[
\Eis_P^{rs,g}\colon D\Coh(\Par_{M}^{rs,g}) \to D\Coh(\Par_{G}^{rs,g})
\]
by restricting $p_P$ and $q_P$ to the ?-loci, where ? is $rs$, $g$ or $rs, g$. Similarly for $\eis_P^?$.

Now we define the normalized intertwining operator $J^\circ = J^{rs,g}$.
Recall that we have $\iota_{M,G}=p_P\circ \iota_{M,P}$ 
and $\iota_{M,P}$ is the inverse of $q_P$ by Lemma \ref{rs,g,locus,general}.
We obtain isomorphisms
\[
\Eis_P^{rs,g}\simeq p_{P*} q_P^* \simeq p_{P*}(\iota_{M,P})_* \simeq (\iota_{M,G})_*.
\]

\begin{definition}[Intertwining over the regular semisimple and generic locus]
\label{rs, g, general}
    Let $Q$ be an another parabolic subgroup of $G$ stable under $W_F$ such that $M\subset Q$ is a Levi subgroup. Then we have Eisenstein operators
    \[
    \Eis_{P}^{rs,g}, \Eis_{Q}^{rs,g}\colon D\Coh(\Par_{M}^{rs,g})\to D\Coh(\Par_{G}^{rs,g}).
    \]
    The intertwining operator
    \[
    J_{Q|P}^{rs,g} \colon \Eis_P^{rs,g} \simto \Eis_Q^{rs,g}
    \]
    is defined as the composition
    \[
    \Eis_P^{rs,g} \simeq (\iota_{M,G})_*\simeq \Eis_Q^{rs,g}
    \]
    where the isomorphism $\Eis_P^{rs,g} \simeq (\iota_{M,G})_*$ (similarly for $Q$) is constructed above. We construct $J_{Q|P,0}^\circ\colon \eis_P^{rs}\simto \eis_Q^{rs}$ similarly.
\end{definition}

As in the unipotent case, the intertwining operator $J^\circ_{Q|P}$ generates all transformations in some sense.

\begin{proposition}
\label{rs,g,general,gen}
    The intertwining operator \[J^{rs,g}_{Q|P}=J_{Q|P}^\circ\colon \Eis_{P}^{rs,g}\simeq \Eis_{Q}^{rs,g}\] generates the free $\CZ(D\Coh(\Par_{M}^{rs,g}))$-module $\Hom(\Eis_P^{rs,g}, \Eis_{Q}^{rs,g})$. 
    
    Similarly, $J^{rs}_{Q|P,0}=J_{Q|P,0}^\circ$ generates the free $\CZ(D\Coh(\Par_{M,0}^{rs}))$-module $\Hom(\eis_P^{rs}, \eis_{Q}^{rs})$.
\end{proposition}

\begin{proof}
    We prove the statement for $\Eis$ and the statement for $\eis$ is similar. 
    As the intertwining operator $J^\circ_{Q|P}$ is defined as $\Eis_P^{rs,g} \simeq (\iota_{M,G})_*\simeq \Eis_Q^{rs,g}$,
    we have
    \[
    \Hom(\Eis_P^{rs,g}, \Eis_{Q}^{rs,g}) \simeq \End((\iota_{M,G})_*)
    \]
    and $J^\circ_{Q|P}$ corresponds to $1$ in the right hand side
    and it remains to show that
    \[
    \CZ(D\Coh(\Par_{M})) \simeq \End((\iota_{M,G})_*).
    \]
    We will then verify that $\iota_{M,G}$ satisfies the conditions in Proposition \ref{gen-quotient}. It suffices to prove for each connected component $\Par_{M,\psi}$, where $\psi$ is semisimple. Choose a minimal Levi subgroup ${}^LM_0\subset {}^LM$ containing $\psi$.
    By Lemma \ref{iota-G-general}, $\iota_{M,G}$
    can be identified with
    \[
    \tilde{\BL}_{M_{0,\psi}^\circ\sigma}^{rs,g}/N_M(M_0,\psi) \to \tilde{\BL}_{M_{0,\psi}^\circ\sigma}^{rs,g}/N_G(M_0,\psi).
    \]
    Thus $\iota_{M,G}$ is an \'etale covering between connected algebraic stacks. 

    It remains to find a point in $(M_{0,\psi}^\circ\sigma)^{rs,g}$ such that its stabilizers under the $M_{0,\psi}^1$-action and $N_G(M_0,\psi)$-action are the same (and hence the its stabilizer under the $N_M(M_0,\psi)$-action is also the same). In fact, by Lemma \ref{generical free}, 
    there is an open and dense subset of $M_{0,\psi}^\circ\sigma \sslash M_{0,\psi}^1$ such that $N_G(M_0,\psi)$ acts freely on it.
    Then any point in the inverse image of this subset in $(M_{0,\psi}^\circ\sigma)^{rs,g}/M_{0,\psi}^1$ satisfies this condition.
\end{proof}

From its construction, $J^\circ_{Q|P}$ satisfies the transitivity,
proving the first main theorem \ref{intertwining, rs,g}.

\subsection{Block decomposition of \texorpdfstring{$D\Coh(\Par_{M,\phi})$}{DCoh(Par M,phi)}}

Next we consider the block decomposition of $D\Coh(\Par_{M,\phi})$ as in the unipotent case. We will only consider elliptic components and always identify $\Par_{M,\phi}$ with $M_\phi^\circ\sigma/M_\phi^1$.

\begin{lemma}
\label{gerbe}
    The morphism $M_\phi^\circ\sigma/M_\phi^1 \to M_\phi^\circ\sigma \sslash M_\phi^1$ is a $C_M(\phi)=(M_\phi^1)^\sigma$-gerbe. 
    Moreover, the morphism $(M_\phi^\circ)_\sigma \to M_\phi^\circ\sigma \sslash M_\phi^1$ is a Galois covering and the gerbe is trivialized over this covering.
\end{lemma}

\begin{proof}
    By \cite{dat2025}, Theorem 6.10, the coarse moduli space of $\Par_{M,\phi}$ is isomorphic to the Haines variety of $\phi$, 
    which consists of unramified central twists of $\phi$.
    In particular, the morphism
    \[
    X_{(Z(M)^I)^\circ} \phi \to \Par_{M,\phi}
    \]
    (where the left hand side denotes the unramified twists of $\phi$ by an unramified cocycle in $(Z(M)^I)^\circ$)
    is surjective.
    Since the $Z(M)^I$-action by twisting and the $C_M(\phi)$-action by conjugation on $Z^1(W_F,M)_\phi$ (the scheme of $\psi$'s satisfying $\psi'|_I=\phi|_I$) commutes, 
    the isotropy groups of any point in $\Par_{M,\phi}$
    are identical.

    As 
    \[\Par_{M,\phi}\simeq M_\phi^\circ\sigma/M_\phi^1\simeq (M_\phi^\circ\sigma/M_\phi^\circ)/\pi_0(M_\phi^1),
    \]
    the isotropy group of any point in $\Par_{M,\phi}$ contains $(M_\phi^\sigma)^\circ$ as a subgroup and the corresponding quotient is the stabilizer of its image in $(M_\phi^\circ)_\sigma$ of the $\pi_0(M_\phi^1)$-action. 
    Since the stabilizers of all points in $(M_\phi^\circ)_\sigma$ are identical, there is a subgroup $H_\phi$ of $\pi_0(M_\phi^1)$ acting trivially on $(M_\phi^\circ)_\sigma$ such that $\pi_0(M_\phi^1)/H_\phi$ acts freely on $(M_\phi^\circ)_\sigma$. This proves $(M_\phi^\circ)_\sigma \to M_\phi^\circ\sigma \sslash M_\phi^1$ is a Galois covering.

    Then we prove the statements about the gerbe.
    The morphism $M_\phi^\circ\sigma/M_\phi^1 \to M_\phi^\circ\sigma \sslash M_\phi^1$ can be factored as
    \[
    ((M_\phi^\circ\sigma/M_\phi^\circ)/H_\phi)/(\pi_0(M_\phi^1)/H_\phi)
    \to ((M_\phi^\circ)_\sigma/H_\phi)/(\pi_0(M_\phi^1)/H_\phi)
    \to (M_\phi^\circ)_\sigma/(\pi_0(M_\phi^1)/H_\phi).
    \]
    Thus it is a gerbe as a composition of gerbes and it is a $C_M(\phi)$-gerbe as $C_M(\phi)$ is the isotropy group of $1\sigma$ in $M_\phi^\circ\sigma/M_\phi^1$. Moreover, recall that by Lemma \ref{splitting}, there is a section $(M_\phi^\circ)_\sigma \to M_\phi^\circ\sigma/M_\phi^\circ$. Thus we have a section 
    \[
    (M_\phi^\circ)_\sigma \to M_\phi^\circ\sigma/M_\phi^\circ
    \to M_\phi^\circ\sigma/M_\phi^1
    \]
    of the gerbe $M_\phi^\circ\sigma/M_\phi^1 \to M_\phi^\circ\sigma \sslash M_\phi^1$ after base change along the Galois covering $(M_\phi^\circ)_\sigma \to M_\phi^\circ\sigma \sslash M_\phi^1$.
\end{proof}

Recall that for a $H$-gerbe $\CX \to X$ 
where $X$ is a connected $k$-scheme and $H$ is a (not necessarily connected) reductive group, we have a block decomposition 
\[
D\Coh(\CX) = \bigoplus_{V \in \irr(H)} D\Coh(\CX)_V
\]
where the category with subscript $V$ consists of sheaves $\CF$ satisfying that for any point $x$ in $X$, the $H$-representation on the fiber $\CF_x$ is $V$-isotypic. 
Moreover, if the gerbe is trivial: $\CX \simeq BH \times X$, for any $\CF \in D\Coh(\CX)_V$, we have 
\[
\CF \simeq V \boxtimes \mathcal{H}om(V\boxtimes\CO_X,\CF)
\]
where the ``$\mathcal{H}om$'' is viewed as a sheaf on $X$.

Note that for any coherent sheaf $\CF$ on $X$, there is an $\CO(X)$-action on it.
This action induces a ring homomorphism from $\CO(X)$ to the Bernstein center of $D\Coh(\CX)_V$ for any $V \in \irr(H)$.

\begin{lemma}
\label{bernstein}
    Let $\CX \to X$ be as above. Suppose that $\CX$ and $X$ are smooth and the gerbe is \'etale-locally trivial. Then the ring homomorphism
    \[
    \CO(X) \to \CZ(D\Coh(\CX)_V)
    \]
    is an isomorphism, for any $V \in \irr(H)$.
\end{lemma}

\begin{proof}
    An element in $\CZ(D\Coh(\CX)_V)$ is given by a compatible family $\{a_\CF\}_{\CF\in D\Coh(\CX)_V}$, where $a_\CF \in \End_\CX(\CF)$. As $\CX$ is smooth and $k$ is of characteristic $0$, $D\Coh(\CX)_V$ is generated by locally free sheaves and thus $a_\CF$ is determined by those elements for locally free $\CF$. For locally free $\CF$, $\CO(X)\to \End_\CX(\CF)$ is injective. Thus the homomorphism $\CO(X) \to \CZ(D\Coh(\CX)_V)$ is injective.

    It remains to prove that the homomorphism is surjective.
    Fix an element $a \in \CZ(D\Coh(X)_V)$ and a locally free $\CF\in D\Coh(\CX)_V$. It suffices to prove that $a_\CF \in \CO(X)$. 
    Choose arbitrarily a closed point $x$ in $X$. Consider the completion $\CF_x^\wedge$ of $\CF$ along the fiber of $x$. As the gerbe is \'etale-locally trivial, $\CX_x^\wedge \simeq BH\times X_x^\wedge$ is a trivial gerbe and thus over $\CX_x^\wedge$, there exists $n\in \mathbb{N}$ such that
    \[
    \CF_x^\wedge \simeq (V \boxtimes \CO_{X,x}^\wedge)^n.
    \]
    As $\End_\CX(V \boxtimes \CO_{X,x}^\wedge)=\CO(X_x^\wedge)$, $a_{\CF_x^\wedge} \in \CO(X_x^\wedge)$ by compatibility. Here we view all those completions of sheaves as inverse systems in $D\Coh(\CX)$.

    We claim that the inverse image of $\CO(X_x^\wedge)$ along the homomorphism
    \[
    \End(\CF) \to \End(\CF_x^\wedge)
    \]
    is $\CO(X)$. We will denote by $(-)_Y$ for a $X$-scheme $Y$ the base change of $(-)$ along $Y \to X$. 
    As the constructions $\End(\CF_Y)$, $\End((\CF_x^\wedge)_Y)$, $\CO(Y)$ and $\CO(X_x^\wedge)_Y$ in the claim are sheafy over the \'etale site over $X$, it suffices to prove the claim after any surjective and \'etale base change.
    
    Choose an \'etale covering $X'\to X$ trivializing the gerbe. Then choose a Zariski covering $\tilde{X}\to X'$ trivializing the sheaf
    $\mathcal{H}om(V\boxtimes \CO_X,\CF)$. Then $\CF_{\tilde{X}} \simeq (V\boxtimes \CO_{\tilde{X}})^m$ for some $m\in \CN$ over $\CX_{\tilde{X}}$. Under this identification,
    \[
    \End_{\tilde{X}}(\CF_{\tilde{X}}) \simeq M_m(\CO(\tilde{X})).
    \]
    Thus the inverse image of $\CO(X_x^\wedge)_{\tilde{X}}$ is the scalar matrices, proving the claim. Hence $a_\CF \in \CO(X)$, proving the lemma.
\end{proof}

By Lemma \ref{gerbe}, the $C_M(\phi)$-gerbe $M_\phi^\circ\sigma/M_\phi^1 \to M_\phi^\circ\sigma\sslash M_\phi^1$ satisfies the conditions in Lemma \ref{bernstein}. Thus there is a block decomposition
\[
D\Coh(\Par_{M,\phi}) = \bigoplus_{V \in \irr(C_M(\phi))}D\Coh(\Par_{M,\phi})_V
\]
and 
\[\CZ(D\Coh(\Par_{M,\phi}))_V = \CO(\Par_{M,\phi})\] 
for any $V$.
Restricting the gerbe $M_\phi^\circ\sigma/M_\phi^1 \to M_\phi^\circ\sigma\sslash M_\phi^1$ to the ?-locus, where ? is $rs$, $g$ or $rs,g$, we have
\[
\CZ(D\Coh(\Par_{M,\phi}^?))_V = \CO(\Par_{M,\phi}^?).
\]
We will use the subscript $V$, such as $\Hom(\Eis^?_P,\Eis^?_Q)_V$ to denote the $\Hom$-module over the subcategory $D\Coh(\Par_{M,\phi})_V$.

Decomposing Proposition \ref{rs,g,general,gen} by block, we have the following corollary.

\begin{corollary}
    \label{rs,g,gen,V}
    Let $V$ be an irreducible representation of $C_M(\phi)$.
    The intertwining operator $J_{Q|P}^\circ\colon \Eis_{P}^{rs,g}\simeq \Eis_{Q}^{rs,g}$ generates the free $\CO(\Par_{M,\phi}^{rs,g})$-module $\Hom(\Eis_P^{rs,g}, \Eis_{Q}^{rs,g})_V$. Similarly, $J_{Q|P,0}^\circ$ generates the free $\CO(\Par_{M,\phi}^{rs})$-module $\Hom(\eis_P^{rs}, \eis_{Q}^{rs})_V$. 
\end{corollary}

\subsection{Integral intertwining operators}

Now we can prove the main theorem \ref{intertwining} about integral intertwining operators.

\begin{theorem}[Integral intertwining operator]
\label{integral, general}
    Let $V$ be an irreducible representation of $C_M(\phi)$ such that the $(M_\phi^\sigma)^\circ$-action on $V$ is invariant under the Weyl group $W(G_\phi^\circ)^\sigma$.
    After restricting the Eisenstein operators to the subcategory $D\Coh(\Par_{M,\phi})_V$ of $D\Coh(\Par_{M,\phi})$,
    there is a natural transformation $j_{Q|P}\colon \Eis_P\simeq \Eis_Q$, which is equal to $L(1,\ad_{\fn_{Q}/\fn_{P}})^{-1}J^\circ_{Q|P}$ after restricting to the regular semisimple and generic locus. Moreover, $j_{Q|P}$ is a generator of $\Hom(\Eis_P, \Eis_Q)_V$, as a free $\CO(\Par_{M, \phi})$-module of rank $1$.
\end{theorem}

\begin{proof}
    Recall that we have
    \[
    \Eis_P = (\pi_\phi)_* \circ \overline{\Eis_{P^\circ_\phi}}
    \] 
    for 
    \begin{gather*}
        \overline{\Eis_{P^\circ_\phi}}\colon D\Coh(\tilde{\BL}_{M_\phi^\circ\sigma}/M_\phi^1)
        \to D\Coh(\tilde{\BL}_{G_\phi^\circ\sigma}/G_\phi^\circ M_\phi^1),\\
        \pi_\phi \colon \tilde{\BL}_{G_\phi^\circ\sigma}/G_\phi^\circ M_\phi^1
        \to \tilde{\BL}_{G_\phi^\circ\sigma}/G_\phi^1.
    \end{gather*}
    We claim that $\pi_{\phi,*}$ induces an isomorphism
    \[
    \Hom(\overline{\Eis_{P^\circ_\phi}}, \overline{\Eis_{Q^\circ_\phi}})_V
    \simto \Hom(\Eis_P, \Eis_Q)_V.
    \]
    In fact, for locally free $\CF$, choose a point $x \in (M_\phi^\circ\sigma)^{rs,g}$ satisfying that the isotropy groups of it and its image in $\tilde{\BL}_{G_\phi^\circ\sigma}/G_\phi^1$ are isomorphic (by Lemma \ref{generical free} it is possible). Then by similar descent arguments as in the proof of Lemma \ref{gen-quotient} and Lemma \ref{bernstein}, 
    we have for any locally free sheaf $\CF \in D\Coh(\tilde{\BL}_{M_\phi^\circ\sigma}/M_\phi^1)_V$,
    \[
    \Hom(\Eis_P\CF, \Eis_Q\CF)\ \text{``} \cap\text{''}\ \Hom(\Eis_P\CF_x^\wedge, \Eis_Q\CF_x^\wedge) = \pi_{\phi,*}(\Hom(\overline{\Eis_{P^\circ_\phi}}\CF, \overline{\Eis_{Q^\circ_\phi}}\CF)),
    \]
    i.e., if there exists compatible $a_\CF \in \Hom(\Eis_P\CF, \Eis_Q\CF)$ and $a_{\CF_x^\wedge} \in \Hom(\Eis_P\CF_x^\wedge, \Eis_Q\CF_x^\wedge)$, then $a_\CF$ is in the image of $\pi_{\phi, *}$,
    proving the claim.

    Thus it suffices to prove similar statements for $\overline{\Eis_{P^\circ_\phi}}$ and $\overline{\Eis_{Q^\circ_\phi}}$.
    The L-function 
    \[
    L(1,\ad_{\fn_{Q}/\fn_{P}})^{-1} = L(1,\ad_{\fn_{Q^\circ_\phi}/\fn_{P_\phi^\circ}})^{-1}
    \]
    is invariant under the $\pi_0(M^1_\phi)$-action and $J^\circ_{Q^\circ_\phi|P^\circ_\phi}$ is the lifting of $J^\circ_{Q|P}$.
    For $\CF \in D\Coh(\Par_{M,\phi})_V$, 
    its lifting to $\BL_{M_\phi^\circ\sigma}$ satisfies that its $(M_\phi^\sigma)^\circ$-characters are invariant under the Weyl group $W(G_\phi^\circ)^\sigma$. Thus by Theorem \ref{all}, the operator 
    \[
    j_{P^\circ_\phi|Q^\circ_\phi} = L(1,\ad_{\fn_{Q^\circ_\phi}/\fn_{P^\circ_\phi}})^{-1}
    J^\circ_{Q^\circ_\phi|P^\circ_\phi}
    \]
    is well-defined. 
    Moreover, it is invariant by $\pi_0(M^1_\phi)$ and hence descends to an operator 
    $j_{Q|P}$ between $\overline{\Eis_{P^\circ_\phi}}$ and $\overline{\Eis_{Q^\circ_\phi}}$, which is $L(1,\ad_{\fn_{Q}/\fn_{P}})^{-1}J^\circ_{Q|P}$ after restricting to the regular semisimple and generic locus, as desired.

    Next we prove the statement about the structure of $\Hom(\overline{\Eis_{P^\circ_\phi}}, \overline{\Eis_{Q^\circ_\phi}})_V$. By Lemma \ref{inclusion} and descent, there is a natural inclusion
    \[
    \Hom(\overline{\Eis_{P^\circ_\phi}}, \overline{\Eis_{P^\circ_\phi}})_V \subset \Hom(\overline{\Eis_{P^\circ_\phi}^{rs,g}}, \overline{\Eis_{Q^\circ_\phi}^{rs,g}})_V
    \simeq \CO(\Par_{M,\phi}^{rs,g}) = 
    \CO(\BL_{M_\phi^\circ\sigma}^{rs,g})^{\pi_0(M_\phi^1)}
    \]
    Thus it suffices to determine for what function $f \in \CO(\Par_{M,\phi}^{rs,g})$, $fJ^\circ_{Q|P}$ can be extended to a well-defined operator $\overline{\Eis_{P^\circ_\phi}}\to \overline{\Eis_{Q^\circ_\phi}}$.
    Note that $f J^\circ_{Q|P}$ is well-defined if and only if its lifting $\tilde{f}J^\circ_{Q_\phi^\circ|P_\phi^\circ}\colon \Eis_{P_\phi^\circ}\to \Eis_{Q_\phi^\circ}$ is well-defined, where $\tilde{f}$ is its image in  
    By Theorem \ref{all}, the statement is equivalent to
    \[\tilde{f}\in \CO(\BL_{M_\phi^\circ\sigma})L(1, \ad_{\fn_{Q^\circ_\phi}/\fn_{P^\circ_\phi}})^{-1}\] and then $f \in \CO(\Par_{M,\phi})L(1,\ad_{\fn_Q/\fn_{P}})^{-1}$. Thus we conclude.
\end{proof}

\begin{remark}
    The generalization of Proposition \ref{g} about the generic locus holds by similar arguments as above.
    More precisely, let $V$ be an irreducible representation of $M_\phi^1$ such that the $(M_\phi^\sigma)^\circ$-action on $V$ is trivial.
    Then the intertwining operator 
    $J^\circ_{Q|P}\colon \Eis_{P}^{rs,g}\simeq \Eis_{Q}^{rs,g}$ can be extended to an isomorphism $J^g_{Q|P} \colon \Eis_{P}^{g}\simeq \Eis_{Q}^{g}$, which
    generates the free $\CO(\Par_{M,\phi}^{g})$-module $\Hom(\Eis_P^{g}, \Eis_{Q}^{g})_V$. 
    
    Similarly, $J^{rs}_{Q|P,0}=J_{Q|P,0}^\circ$ can be extended to an isomorphism
    $J_{Q|P,0} \colon \eis_{P}\simeq \eis_{Q}$, which
    generates the free $\CO(\Par_{M,\phi})
    $-module $\Hom(\eis_P, \eis_{Q})_V$.
\end{remark}

\subsection{Constant term and the unnormalized intertwining operator}

Next we define the constant term functor $\CT_P$ as the partial left adjoint functor of $\Eis_P$, and finally construct the unnormalized intertwining operator. Note that we will always consider adjoint functors in the $\infty$-categorical sense.

Denote by $D\Coh(\Par_{M,\phi})_{W^\sigma-\text{inv}}$ the direct sum of all $D\Coh(\Par_{M,\phi})_V$'s for irreducible representation $V$ of $C_M(\phi)$ such that $(M_\phi^\sigma)^\circ$ action is invariant under the Weyl group $W(G_\phi^\circ)^\sigma$.  
Denote by $D\Coh(\Par_{G,\phi})_{M,W^\sigma-\text{inv}}$ the thick subcategory generated by the image of 
\[\Eis_P \colon D\Coh(\Par_{M,\phi})_{W^\sigma-\text{inv}}\to D\Coh(\Par_{G,\phi}).\] 
Note that the category $D\Coh(\Par_{G,\phi})_{M,W^\sigma-\text{inv}}$ is independent of the choice of $P$: for $P'$, there exists $w \in W(G_\phi^\circ)^\circ$ such that 
$w(P_\phi'^\circ)=P_\phi^\circ$ and then $\Eis_{P'}=\Eis_P\circ w_*$.

The expected left adjoint of $\Eis_P$ is denoted by
\[\CT_P\colon D\Coh(\Par_{G,\phi})_{M,W^\sigma-\text{inv}} \to D\Coh(\Par_{M,\phi})_{W^\sigma-\text{inv}}.\]
We view $\CT_P\CG$ only formally as a functor $D\Coh(\Par_{M,\phi})_{W^\sigma-\text{inv}} \to \mathsf{Ani}$ for now. However, by Proposition \ref{ct-general} that we will prove later, $\CT_Q\circ \Eis_P \CF$ for any $\CF \in D\Coh(\Par_{M,\phi})_{W^\sigma-\text{inv}}$ lies in $D\Coh(\Par_{M,\phi})_{W^\sigma-\text{inv}}$ and thus
$\CT_P$ defines exactly a functor $D\Coh(\Par_{G,\phi})_{M,W^\sigma-\text{inv}} \to D\Coh(\Par_{M,\phi})_{W^\sigma-\text{inv}}$.

\begin{proposition}[Decomposition of $\CT\circ \Eis$]
\label{decomposition}
    There is a decomposition
    \[
    \CT_Q\circ\Eis_P \simeq \bigoplus_{\gamma\in \pi_0(G^1_\phi/M^1_\phi)} \overline{\CT_{Q_\phi^\circ}}\circ \overline{\Eis_{P_\phi^\circ}}\circ \gamma^*,
    \]
    where $\gamma$, being viewed as an element in $N_G(M,\phi)/M_\phi^1$ under the isomorphism in Proposition \ref{semidirect}, 
    means the $\gamma$-action on $\tilde{\BL}_{M_\phi^\circ\sigma}/M_\phi^1$.
    Here $\overline{\CT_{Q_\phi^\circ}}$ 
    is the formal adjoint functor of $\overline{\Eis_{Q_\phi^\circ}}$.
\end{proposition}

\begin{proof}
    Recall that $\Eis_P \simeq \pi_{\phi,*}\circ \overline{\Eis_{P_\phi^\circ}}$. We have $\CT_Q\simeq \overline{\CT_{Q_\phi^\circ}}\circ \pi_\phi^*$.
    As $\pi_\phi \colon \tilde{\BL}_{G_\phi^\circ\sigma}/G_\phi^\circ M_\phi^1 \to \tilde{\BL}_{G_\phi^\circ\sigma}/G_\phi^1$ is a Galois cover with Galois group $\pi_0(G_\phi^1/M_\phi^1)$, we have
    \[
    \CT_Q \circ \Eis_P \simeq \overline{\CT_{Q_\phi^\circ}}\circ 
    \pi_\phi^* \circ \pi_{\phi, *} \circ \overline{\Eis_{P_\phi^\circ}}
    \simeq \bigoplus_{\gamma\in \Gamma} \overline{\CT_{Q_\phi^\circ}}\circ 
    \gamma^* \circ \overline{\Eis_{P_\phi^\circ}},
    \]
    where $\gamma$ is the $\gamma$-action on $\tilde{\BL}_{G_\phi^\circ\sigma}/G_\phi^\circ M_\phi^1$.
    By Proposition \ref{semidirect}, the image of $\gamma$ under the section $\pi_0(G_\phi^1/M_\phi^1) \to N_G(M,\phi)/M_\phi^1$
    stabilizes $P_\phi^1$.
    Thus the $\gamma$-action restricts to actions on $\tilde{\BL}_{P_\phi^\circ\sigma}/P_\phi^1$ and $\tilde{\BL}_{M_\phi^\circ\sigma}/M_\phi^1$. 
    Hence 
    \[
    \gamma^*\circ \overline{\Eis_{P_\phi^\circ}}\simeq  \overline{\Eis_{P_\phi^\circ}}\circ \gamma^*
    \]
    and we conclude.
\end{proof}

\begin{proposition}
\label{ct-general}
    For any element $\CG \in D\Coh(\Par_{G,\phi})_{M,W^\sigma-\text{inv}}$, 
    the formal $\CT_Q(\CG)$ is representable by an object in $D\Coh(\Par_{M,\phi})_{W^\sigma-\text{inv}}$.
    
    In particular, after restricting the functor $\Eis_P$ to the block
    $D\Coh(\Par_{M,\phi})_V$ for $V$ as above, its left adjoint is obtained by composing $\CT_P$ with the functor of taking the $V$-isotypic component. The result will be denoted by $\CT_{P,V}$.

    Moreover, for locally free $\CF \in D\Coh(\Par_{M,\phi})_{W^\sigma-\text{inv}}$, $\CT_Q\circ\Eis_P\CF$ is locally free.
\end{proposition}

\begin{proof}
    By similar formal arguments as in the end of the proof of Proposition 
    \ref{ct_circ_eis}, it suffices to prove that for locally free 
    $\CF\in D\Coh(\Par_{M,\phi})_1$, $\CT_Q\circ\Eis_P\CF$ is representable and locally free.

    Fixing locally free $\CF \in D\Coh(\Par_{M,\phi})_1$, denote by $\tilde{\CF}$ its lifting on $\BL_{M_\phi^\circ}$. Then $\tilde{\CF}$ is a $\pi_0(M_\phi^1)$-equivariant sheaf. 
    The $(M_\phi)^\circ$-action at any fiber of $\tilde{\CF}$ is trivial by definition. Thus by Proposition \ref{CT_chi}, the object $\CT_{Q_\phi^\circ}\circ\Eis_{P_\phi^\circ}\tilde{\CF}$ exists and the $(M_\phi^\sigma)^\circ$-action on any fiber of it is trivial.
    As the construction $\CT_{Q_\phi^\circ}\circ\Eis_{P_\phi^\circ}$
    is functorial, we obtain a $\pi_0(M_\phi^1)$-equivariant structure on
    $\CT_{Q_\phi^\circ}\circ\Eis_{P_\phi^\circ}\tilde{\CF}$ and it can be descended to an object $\overline{\CT_{Q_\phi^\circ}\circ\Eis_{P_\phi^\circ}\tilde{\CF}}$ in $D\Coh(\Par_{M,\phi})_1$. This object satisfies that
    \[
    \Hom(\overline{\CT_{Q_\phi^\circ}\circ\Eis_{P_\phi^\circ}\tilde{\CF}},
    \CF') = \Hom(\CT_{Q_\phi^\circ}\circ\Eis_{P_\phi^\circ}\tilde{\CF}, \widetilde{\CF'})^{\pi_0(M_\phi^1)}
    = \Hom(\overline{(\Eis_{P_\phi^\circ}}\CF, \overline{\Eis_{Q_\phi^\circ}}\CF').
    \]
    Thus it is the desired object $\overline{\CT_{Q_\phi^\circ}}\circ \overline{\Eis_{P_\phi^\circ}}\CF$. By Proposition \ref{ct_circ_eis} and \ref{CT_chi}, $\CT_{Q_\phi^\circ}\circ\Eis_{P_\phi^\circ}\tilde{\CF}$ is locally free. Hence $\overline{\CT_{Q_\phi^\circ}}\circ \overline{\Eis_{P_\phi^\circ}}\CF$ is locally free.

    As $\gamma^*\CF \in D\Coh(\Par_{M,\phi})_1$ and is locally free, $\overline{\CT_{Q_\phi^\circ}}\circ \overline{\Eis_{P_\phi^\circ}}
    \circ \gamma^* \CF$ is representable for any $\gamma \in \pi_0(G_\phi^1/M_\phi^1)$. Thus $\CT_Q\Eis_P\CF$ is representable and locally free by the decomposition in Proposition \ref{decomposition}.
\end{proof}

\begin{remark}
    As mentioned in remark \ref{CT_B*}, over the category of ind-coherent sheaves, we may define a right adjoint $\CT_{P*}$ of $\Eis_P$ even for general components of $\Par_M$ and $\Par_G$. Moreover, the conjectural second adjunction implies that $\CT_{\bar{P}*}$ is the totally defined left adjoint of $\Eis_P$.
\end{remark}

Now we define a filtration of $\CT\circ \Eis$ indexed by the ``relative Weyl group'' $N_G(M, \phi)/M_\phi^1$. 

\begin{definition}[Order on $N_G(M, \phi)/M_\phi^1$]
    The partial order on $N_G(M, \phi)/M_\phi^1$ is defined as follows.
    For $w, w' \in N_G(M, \phi)/M_\phi^1$,
    \[
    w\leq w' \iff \overline{P w Q}\subset \overline{P w'Q} \subset G.
    \]
\end{definition}

As we have a morphism of stacks $P_\phi^1\backslash G_\phi^1/Q_\phi^1\to P\backslash G/Q$,
for $w, w' \in W(G_\phi^\circ)^\sigma$ and $\gamma\in \pi_0(G_\phi^1/M_\phi^1)\hookrightarrow N_G(M,\phi)/M_\phi^1$
as in Proposition \ref{semidirect},
$\overline{P_\phi^\circ w Q_\phi^\circ}\subset \overline{P_\phi^\circ w'Q_\phi^\circ}$ implies that $\gamma w\leq \gamma w'$. 
Thus the restriction of the order on $N_G(M, \phi)/M_\phi^1$ over each coset $W(G_\phi^\circ)^\sigma\gamma$ refines the Bruhat order of $W(G_\phi^\circ)^\sigma$ with respect to $(P_\phi^\circ, Q_\phi^\circ)$.

\begin{lemma}
\label{descend filtration}
    Let $\CF$ be an object in $D\Coh(\Par_{M,\phi})_{W^\sigma-\text{inv}}$ and denote by $\tilde{\CF}$ its lifting on $\BL_{M_\phi^\circ\sigma}$. 
    The filtration $\{F^{\geq w}, F^{>w}\}$ with respect to the order above of $\CT_{Q_\phi^\circ}\circ\Eis_{P_\phi^\circ}\widetilde{\gamma^*\CF}$
    indexed by $W(G_\phi^\circ)^\sigma$ defined in Definition   \ref{filtration} can be uniquely descended to a functorial filtration $\{\bar{F}^{\geq w}, \bar{F}^{>w}\}$ of $\overline{\CT_{P_\phi^\circ}}\circ \overline{\Eis_{P_\phi^\circ}}\gamma^*\CF$.

    Moreover, if $\CF$ is locally free, then all the subquotients on the filtration $\bar{F}$ are locally free.
\end{lemma}

As in the unipotent case,
here we use the term ``filtration'' for a compatible family of functors, insisting the fact that it defines a classical filtration in an abelian category for locally free sheaves.

\begin{proof}
    We will omit $\gamma$ in the notations, as only this summand is concerned.
    As $\BL_{M_\phi^\circ\sigma}\to \Par_{M,\phi}$ is a Galois covering, the uniqueness part is clear.
    It suffices to construct such a filtration for locally free $\CF$.
    Then the result extends by linearity. 
    Moreover, we can consider the case $P_\phi^\circ = Q_\phi^\circ$ only
    and the general case follows from multiplying the indices by some $w \in W(G_\phi^\circ)^\sigma$.
    Denote by $\CT^{rs,g}$ the left adjoint of $\Eis^{rs,g}$ and $j^{rs,g}$ the open immersion from the regular and semisimple locus. 
    Then $\overline{\CT_{P_\phi^\circ}^{rs,g}}\circ \overline{\Eis_{P_\phi}^{rs,g}}\CF$ exists and there is a natural inclusion
    \[
    \overline{\CT_{P_\phi^\circ}}\circ j^{rs,g}_*\overline{\Eis_{P_\phi^\circ}}\CF
    \subset
    \overline{\CT_{P_\phi^\circ}^{rs,g}}\circ \overline{\Eis_{P_\phi}^{rs,g}}j^{rs,g*}\CF
    \]
    as it is so for $\tilde{\CF}$ and descent.
    Similar to Definition \ref{rs, g, general}, 
    $\overline{\Eis^{rs,g}_{P_\phi^\circ}}$ can be canonically identified with
    $\overline{\iota_{G_\phi^\circ}}_*$, where
    \[
    \overline{\iota_{G_\phi^\circ}} \colon \tilde{\BL}^{rs,g}_{M_\phi^\circ\sigma}/M_\phi^1
    \to (\tilde{\BL}^{rs,g}_{M_\phi^\circ\sigma}/M_\phi^1)/W(G_\phi)^\sigma
    \simto  
    \tilde{\BL}^{rs,g}_{G_\phi^\circ\sigma}/G_\phi^\circ M_\phi^1.
    \]
    As $\overline{\iota_{G_\phi^\circ}}$ is a Galois covering with Galois group $W(G_\phi^\circ)^\sigma$, we have
    \[
    \overline{\CT_{P_\phi^\circ}^{rs,g}}\circ \overline{\Eis_{P_\phi}^{rs,g}}j^{rs,g*}\CF \simeq \overline{\iota_{G_\phi^\circ}}^*\overline{\iota_{G_\phi^\circ}}_*j^{rs,g*}\CF
    \simeq \bigoplus_{w\in W(G_\phi^\circ)^\sigma} w^*j^{rs,g*}\CF.
    \]
    Then we have a filtration $F^{\geq w}$ defined by
    \[
    \bar{F}^{\geq w} := \bigcap_{w'<w}
    \ker(
    \overline{\CT_{P_\phi^\circ}}\circ \overline{\Eis_{P_\phi^\circ}}\CF
    \hookrightarrow
    j^{rs,g}_*\overline{\CT_{P_\phi^\circ}^{rs,g}}\circ \overline{\Eis_{P_\phi}^{rs,g}}\CF \to j^{rs,g}_*w^*j^{rs,g*}\CF)
    \]
    and similarly for $F^{>w}$. By constructions in Definition \ref{filtration}, $\bar{F}$ lifts to the filtration $F$ of $\CT_{P_\phi^\circ}\circ\Eis_{P_\phi^\circ}\tilde{\CF}$ as desired.
    Moreover, the subquotients of the filtration is locally free, as they are so after lifting to $\BL_{M_\phi^\circ\sigma}$.
\end{proof}

\begin{definition}[Filtration of $\CT\circ \Eis$]
\label{filtration-general}
    The filtration $\{F^{\geq w}, F^{>w}\}$ of the functor $\CT_Q\circ \Eis_P$ on the category $D\Coh(\Par_{M,\phi})_{W^\sigma-\text{inv}}$ indexed by $N_G(M, \phi)/M_\phi^1$ with respect to the above ordering is
    defined as follows. Write $w=\gamma w_0$ as in Proposition \ref{semidirect}. By Lemma \ref{descend filtration}, there is a filtration $\bar{F}$ of $\overline{\CT_{Q_\phi^\circ}}\circ \overline{\Eis_{P_\phi^\circ}}$. Then for locally free $\CF$, as in the proof of Lemma \ref{descend filtration}, we define
    \[
    F^{\geq w} := \bigcap_{w'<w}
    \ker(
    \CT_Q\circ \Eis_P\CF
    \hookrightarrow
    j^{rs,g}_*\CT_Q\circ \Eis_P\CF \to j^{rs,g}_*w'^*j^{rs,g*}\CF)
    \]
    and similarly for $F^{>w}$.
    Note that the filtration is preserved under homomorphisms between locally free sheaves in $D\Coh(\Par_{M,\phi})_{W^\sigma-\text{inv}}$. We may extend by linearity to uniquely obtain a filtration on the functor $\CT_Q\circ\Eis_P$ on $D\Coh(\Par_{M,\phi})_{W^\sigma-\text{inv}}$. 

    Moreover, let $V$ be as above. By taking $V$-isotypic components, we obtain a filtration indexed by $N_G(M, \phi)/M_\phi^1$ of the functor $\CT_{Q,V}\circ \Eis_P$ on the category $D\Coh(\Par_{M,\phi})_V$. This filtration reduces to a $N_G(M,\phi,V)/M_\phi^1$-filtration, where the group $N_G(M,\phi,V)$ is the subgroup of $N_G(M,\phi)$ which containing those elements that preserve $D\Coh(\Par_{M,\phi})_V$ under their action on $\Par_{M,\phi}$.
\end{definition}

By definition, under the decomposition in Proposition \ref{decomposition}, if $w=\gamma w_0$ for $w_0\in W(G_\phi^\circ)^\sigma$,
$F$ restricts to the filtration $\bar{F}$ on $\overline{\CT_{P_\phi^\circ}}\circ \overline{\Eis_{P_\phi^\circ}}\gamma^*\CF$.

Now we can prove the third main theorem \ref{Adjunction}, about the unnormalized intertwining operator.

\begin{theorem}
\label{mul-general}
    Let $V$ be an irreducible representation of $M_\phi^1$ such that the $(M_\phi^\sigma)^\circ$-action on $V$ is invariant under the Weyl group $W(G_\phi^\circ)^\sigma$.
    The natural transformations between endofunctors on $D\Coh(\Par_{M,\phi})_V$, 
    viewed as an $\CO(\Par_{M,\phi}) = \CZ(D\Coh(\Par_{M,\phi}))_V$-module 
    \[
    \Hom(\mathrm{id}, \CT_{Q,V}\circ\Eis_P/F^{>1})_V 
    \]
    is free and of rank $1$. Moreover, there exists a generator of the module such that the composition
    \[
    \mathrm{id} \to \CT_{Q,V}\circ \Eis_{P}/F^{>1} \to \mathrm{id}
    \]
    is the multiplication by $L(0, \ad_{\fn_{Q}/\fn_{P}})^{-1}$,
    where the morphism 
    \[
    \CT_{Q,V}\circ \Eis_P \to \mathrm{id}
    \]
    is given by the adjunction of the intertwining operator $j_{Q|P} \colon \Eis_P \to \Eis_Q$ defined in \ref{integral, general}.
\end{theorem}

\begin{proof}
    We may replace the functor $\CT_{Q,V}$ by $\CT_Q$, as we have
    \[
    \Hom(\mathrm{id}, \CT_{Q,V}\circ\Eis_P/F^{>1})_V \simeq \Hom(\mathrm{id}, \CT_{Q}\circ\Eis_P/F^{>1})
    \]
    where in the latter $\Hom$, the terms are functors from $D\Coh(\Par_{M,\phi})_V$ to $D\Coh(\Par_{M,\phi})_{W^\sigma-\text{inv}}$ and $\mathrm{id}$ is viewed as the inclusion.
    
    Firstly, we prove that for $\gamma\neq 1 \in N_G(M,\phi)/M_\phi^1$, 
    \[
    \Hom(\mathrm{id}, \overline{\CT_{Q_\phi^\circ}}\circ \overline{\Eis_{P_\phi^\circ}}\circ \gamma/\bar{F}^{>1}) = 0.
    \]
    As $D\Coh(\Par_{M,\phi})$ and $D\Coh(\Par_{M,\phi})_V$ are generated by locally free sheaves and by Lemma \ref{descend filtration}, for locally free $\CF$, $\overline{\CT_{Q_\phi^\circ}}\circ\overline{\Eis_{P_\phi^\circ}}\circ \gamma^*\CF/F^{>1}$ is locally free. Thus we have an inclusion
    \[
    \Hom(\mathrm{id}, \overline{\CT_{Q_\phi^\circ}}\circ \overline{\Eis_{P_\phi^\circ}}\circ \gamma/\bar{F}^{>1}) 
    \subset
    \Hom(j^{rs,g,*}, \overline{\CT^{rs,g}_{Q_\phi^\circ}}\circ\overline{\Eis^{rs,g}_{P_\phi^\circ}}\gamma^*j^{rs,g*}/\bar{F}^{>1}),
    \]
    where $j^{rs,g}$ is the open immersion from the regular semisimple and generic locus.
    As in the proof of Lemma \ref{descend filtration}, we have
    \[
    \overline{\CT^{rs,g}_{Q_\phi^\circ}}\circ\overline{\Eis^{rs,g}_{P_\phi^\circ}}\gamma^*j^{rs,g*}\CF \simeq \bigoplus_{w\in W(G_\phi^\circ)^\sigma} w^*\gamma^*j^{rs,g*}\CF.
    \]
    By Lemma \ref{generical free}, there exists a point $x$ in the coarse moduli of $\Par_{M,\phi}$ such that $N_G(M,\phi)$ acts freely on it.
    Then for such $x$,
    \[
    \Hom(\CF_x^\wedge ,w^*\gamma^*\CF_x^\wedge) = 0
    \]
    and thus $\Hom(\mathrm{id}, \overline{\CT_{Q_\phi^\circ}}\circ \overline{\Eis_{P_\phi^\circ}}\circ \gamma/\bar{F}^{>1})=0$.

    Hence the inclusion map under the decomposition of $\CT_Q\circ\Eis_P$ in Proposition \ref{decomposition}
    \[
    \Hom(\mathrm{id}, \overline{\CT_{Q_\phi^\circ}}\circ \overline{\Eis_{P_\phi^\circ}}/\bar{F}^{>1}) \to
    \Hom(\mathrm{id}, \CT_Q\circ\Eis_P/F^{>1}) 
    \]
    is an isomorphism. Furthermore, again from the decomposition of $\overline{\CT^{rs,g}_{Q_\phi^\circ}}\circ\overline{\Eis^{rs,g}_{P_\phi^\circ}}$, we have 
    \[
    \Hom(\mathrm{id}, \overline{\CT_{Q_\phi^\circ}}\circ \overline{\Eis_{P_\phi^\circ}}/\bar{F}^{>1})_V 
    \hookrightarrow \Hom(j^{rs,g*}, \overline{\CT^{rs,g}_{Q_\phi^\circ}}\circ \overline{\Eis^{rs,g}_{P_\phi^\circ}}j^{rs,g*}/\bar{F}^{>1})_V 
    = \End(j^{rs,g*})_V = \CO(\Par_{M,\phi}^{rs,g})
    \]
    under which $1$ in the rightmost object corresponds in the second object to the canonical section of the adjoint of $J_{Q|P}^\circ$,
    denoted by $s$. Thus it suffices to prove that $L(1, \ad_{\fn_Q/\fn_P})L(0, \ad_{\fn_Q/\fn_P})^{-1}s$ lies in the first $\Hom$-module and generates it as an $\CO(\Par_{M,\phi})$-module.

    Consider the Galois covering $\BL_{M_\phi^\circ\sigma} \to \tilde{\BL}_{M_\phi^\circ\sigma}/M_\phi^1$ and denoted by a tilde
    the lifting of an object along the covering. 
    By definition, $\tilde{s}$ is the canonical section of the adjoint of $J_{Q_\phi^\circ|P_\phi^\circ}^\circ$. 
    Thus by Theorem \ref{mul-sc}, $L(1, \ad_{\fn_Q/\fn_P})L(0, \ad_{\fn_Q/\fn_P})^{-1}\tilde{s}$ is a well-defined transformation
    $\mathrm{id} \to \CT_{Q_\phi^\circ}\Eis_{P_\phi^\circ}/F^{>1}$. As the L-functions are invariant under the $\pi_0(M_\phi^1)$-action, 
    the transformation
    \[
    L(1, \ad_{\fn_Q/\fn_P})L(0, \ad_{\fn_Q/\fn_P})^{-1}s \in \Hom(\mathrm{id}, \overline{\CT_{Q_\phi^\circ}}\circ \overline{\Eis_{P_\phi^\circ}}/\bar{F}^{>1})_V
    \]
    is well-defined.

    Moreover, for $f \in \CO(\Par_{M,\phi}) = \CO(\BL_{M_\phi^\circ\sigma})^{\pi_0(M_\phi^1)}$, 
    choose arbitrarily a locally free sheaf $\CF\in D\Coh(\Par_{M,\phi})_V$. Suppose that the transformation $fs$ can be extended everywhere. Then 
    \[
    \widetilde{(fs)_\CF} \in \Hom_{\CO(\BL_{M_\phi^\circ\sigma})}(\tilde{\CF}, \CT_{Q_\phi^\circ}\Eis_{P_\phi^\circ}\tilde{\CF}/F^{>1}).
    \]
    We then prove the claim that the image of the module $\Hom(\mathrm{id},  \CT_{Q_\phi^\circ}\Eis_{P_\phi^\circ}/F^{>1})$ in the $\Hom$-module for $\tilde{\CF}$ is saturated. As $\CF$ is locally free, it is a direct summand of $\CO_\chi$'s, where $\chi$'s are characters of $(M_\phi^\circ)^\sigma$ such that they are trivial on $(M_\phi^\sigma)^\circ$. Hence it suffices to prove that the $\CO(\BL_{M_\phi^\circ\sigma})$-submodule
    \[
    \CO(\BL_{M_\phi^\circ\sigma}) \simeq \Hom(\mathrm{id},  \CT_{Q_\phi^\circ}\Eis_{P_\phi^\circ}/F^{>1})_\chi \subset
    \Hom(\CO_\chi,  \CT_{Q_\phi^\circ}\Eis_{P_\phi^\circ}\CO_\chi/F^{>1})
    \]
    is saturated. By Proposition \ref{CT_chi}, it suffices to prove the case for $\chi=1$. By Lemma \ref{free} and the proof of Theorem \ref{decomposition}, for $w \in W(G_\phi^\circ)^\sigma$ such that $Q_\phi^\circ=w(P_\phi^\circ)$, the $\CO(\BL_{M_\phi^\circ\sigma})$-module
    \[
    \Hom(\CO_1,  \CT_{Q_\phi^\circ}\Eis_{P_\phi^\circ}\CO_1/F^{>1}) =
    \Gamma(\CT_{Q_\phi^\circ}\Eis_{P_\phi^\circ}\CO_1/F^{>1})
    = \Gamma(\CT_{P_\phi^\circ}\Eis_{P_\phi^\circ}\CO_1/F^{>w})
    \]
    is free and generated by $\{(L\delta)^*_{w'}\}_{w'\leq w}$ and the submodule $\Hom(\mathrm{id},  \CT_{Q_\phi^\circ}\Eis_{P_\phi^\circ}/F^{>1})_1$ corresponds to the submodule generated by $(L\delta)_w^*$, proving the claim.

    Thus $\widetilde{(fs)_\CF}$ is an $\CO(\BL_{M_\phi^\circ\sigma})$-multiple of a generator of $\Hom(\mathrm{id}, \CT_{Q_\phi^\circ}\Eis_{P_\phi^\circ}/F^{>1})$,
    as it is a rational multiple. Hence $\tilde{f}$ is an $\CO(\BL_{M_\phi^\circ\sigma})$-multiple of $L(1, \ad_{\fn_Q/\fn_P})L(0, \ad_{\fn_Q/\fn_P})^{-1}$  as rational functions in $k(\BL_{M_\phi^\circ\sigma})$ and then $f$ is an $\CO(\Par_{M,\phi})$-multiple of $L(1, \ad_{\fn_Q/\fn_P})L(0, \ad_{\fn_Q/\fn_P})^{-1}$ as rational functions in $k(\Par_{M,\phi})$, 
    proving that $L(1, \ad_{\fn_Q/\fn_P})L(0, \ad_{\fn_Q/\fn_P})^{-1}s$ is a generator.
\end{proof}

From the proof, we have the following corollary as an analog of the Mackey's formula for classical parabolic induction and restrictions.
\begin{corollary}
    \label{stratum w}
    The descent of $L(1, \ad_{\fn_Q/\fn_P})L(0, \ad_{\fn_Q/\fn_P})^{-1}s$ induces an isomorphism
    \[
    \mathrm{id} \simeq F^{\geq 1}/F^{>1}(\CT_Q\circ \Eis_P)
    \]
    and by twisting we generalize this results to all strata:
    \[
    w^* \simeq F^{\geq w}/F^{>w}(\CT_Q\circ \Eis_P).
    \]
\end{corollary}

As in the unipotent case, we divide the integral intertwining operator to obtain the unnormalized intertwining operator.

\begin{definition}[Unnormalized intertwining operator]
\label{unnormalized}
    The unnormalized intertwining operator $J_{Q|P}$ is defined to be the rational operator $\Eis_P\to \Eis_Q$ by 
    \[
    J_{Q|P} := L(0,\mathrm{ad}_{\fn_Q/\fn_P}) j_{Q|P},
    \]
    as a transformation $\Eis_P^{rs} \to \Eis_Q^{rs}$.
\end{definition}

By Theorem \ref{mul-general}, $J_{Q|P}$ is the desired unnormalized intertwining operator in Theorem \ref{Adjunction}.

%% file: 4.tex
\section{Compatibility with classical intertwining operators}

In this section, we will state the compatibility between the classical intertwining operator and the spectral Eisenstein operator defined above,
under the conjectural categorical local Langlands correspondence.

\subsection{Review of categorical local Langlands correspondence}

We assume that $G$ is an extended pure inner form $(G^*)_b$ for a quasi-split p-adic reductive group $G^*$ and an element $b$ in the Kottwitz set $B(G^*)_{bas}$. 
To state the compatibility result, we assume the following version of the categorical local Langlands correspondence.
Recall the definition and properties of $\operatorname{Bun}_{G^*}$, the moduli space of principal $G^*$-bundles over the Fargues--Fontaine curve in \cite{fargues2024}, Chapter III and $D_{\text{lis}}(\operatorname{Bun}_{G^*})^\omega$, the category of lisse \'etale sheaves over $\operatorname{Bun}_{G^*}$ in \cite{fargues2024}, Chapter V. In particular, there is an open immersion $i_b \colon */G_b(F) \to \operatorname{Bun}_G$, which induces an inclusion $i_{b!} \colon D\mathsf{Rep}(G(F)) \to D_{\text{lis}}(\operatorname{Bun}_{G^*})$.

\begin{conjecture}[Categorical local Langlands correspondence]
\label{categorical}
    Let $G = (G^*)_b$ be as above.
    For a fixed Whittaker datum of $G^*$, there is an equivalence of categories
    \[
    \widetilde{\operatorname{LL}}_{G^*} \colon 
    D_{\text{lis}}(\operatorname{Bun}_{G^*})^\omega \simeq D\Coh^{\text{qc}}(\Par_{G^\vee}),
    \]
    where $\omega$ denotes the full subcategory of compact object and qc denotes the full subcategory of objects with compact supports,
    satisfying the following properties:
    \begin{enumerate}
        \item Under this functor, the parabolic induction operator corresponds to the Eisenstein operator. More precisely, let $P$ be a parabolic subgroup of $G$ with Levi subgroup $M$, denote by $\bar{P}$ the opposite parabolic subgroup of $P$. We have $M=(M^*)_b$ for the quasi-split inner form $M^*$ of $M$ and the Whittaker datum for $G$ induces a Whittaker datum for $M^*$. Then the diagram
        \[
        \begin{tikzcd}
	    D(\mathsf{Rep}(M(F))) & D_{\text{lis}}(\operatorname{Bun}_{M^*})^\omega & D\Coh^{\text{qc}}(\Par_{M^\vee}) \\
	      D(\mathsf{Rep}(G(F))) & D_{\text{lis}}
        (\operatorname{Bun}_{G^*})^\omega & D\Coh^{\text{qc}}(\Par_{G^\vee})
        \arrow["i_{b!}",from=1-1, to=1-2]
	    \arrow["\widetilde{\operatorname{LL}}_{M^*}",from=1-2, to=1-3]
        \arrow["i_{b!}",from=2-1, to=2-2]
	    \arrow["\widetilde{\operatorname{LL}}_{G^*}",from=2-2, to=2-3]
        \arrow["{i_{\bar{P}}^G}"', from=1-1, to=2-1]
	    \arrow["{\Eis^a_{\bar{P}!}}", from=1-2, to=2-2]
        \arrow["{\Eis_{P^\vee}}", from=1-3, to=2-3]
        \end{tikzcd}
        \]
        commutes up to isomorphisms, where $\Eis^a_{\bar{P}!}$ is the Eisenstein operator for $\operatorname{Bun}$'s, defined in \cite{hamann}, Definition 2.1.8. In this notation, ``$a$'' stands for ``automorphic side''.
        \item This functor is compatible with automorphisms: for an automorphism $\sigma$ of $G^*$ that preserves $b$ and a pinning, it induces an isomorphism $\sigma^\vee$ of $G^\vee$. Then the diagram
        \[
        \begin{tikzcd}
	    D(\mathsf{Rep}(G(F))) & D_{\text{lis}}(\operatorname{Bun}_{G^*})^\omega & D\Coh^{\text{qc}}(\Par_{G^\vee}) \\
	      D(\mathsf{Rep}(G(F))) & D_{\text{lis}}
        (\operatorname{Bun}_{G^*})^\omega & D\Coh^{\text{qc}}(\Par_{G^\vee})
        \arrow["i_{b!}",from=1-1, to=1-2]
	    \arrow["\widetilde{\operatorname{LL}}_{G^*}",from=1-2, to=1-3]
        \arrow["i_{b!}",from=2-1, to=2-2]
	    \arrow["\widetilde{\operatorname{LL}}_{G^*}",from=2-2, to=2-3]
        \arrow["\sigma"', from=1-1, to=2-1]
	    \arrow["\sigma", from=1-2, to=2-2]
        \arrow["\sigma^\vee", from=1-3, to=2-3]
        \end{tikzcd}
        \]
        commutes up to isomorphisms.
        Such a functor is called categorical local Langlands correspondence.
        \item The correspondence is compatible with the Kottwitz isomorphism. More precisely, for any $V \in \mathsf{Rep}(G(F))$,
        The $Z(G^\vee)^{W_F}$-character of the fiber of $\widetilde{\operatorname{LL}}_{G^*} \circ i_{b!}(V)$ at any point in $\Par_{G^\vee}$ is $\kappa(b)$ under the Kottwitz isomorphism $\kappa\colon B(G^*)_{bas} \simeq X^*(Z(G^\vee)^{W_F})$.
    \end{enumerate}    
\end{conjecture}

In particular, we obtain a fully faithful functor
\[
    \operatorname{LL}_G = \widetilde{\operatorname{LL}}_{G^*} \circ i_{b!} \colon D(\mathsf{Rep}(G(F))) \to D\Coh(\Par_{G^\vee})
\]
Moreover, since $\CO(\Par_{G^\vee})$ is in the Bernstein center of the right hand side, it induces a homomorphism of rings
\[
    \operatorname{LL}_G^*\colon \CO(\Par_{G^\vee}) \to \CZ(D\mathsf{Rep}(G(F))),
\]
called the local Langlands correspondence for functions, which will satisfy similar compatibility conditions. The correspondence for functions induces for an irreducible representation $\pi$ of $G(F)$ a semisimple cocycle $\phi_\pi\in Z^1(W_F, G^\vee)$ up to inner conjugation, called its semisimple L-parameter. 

One possible way to construct such a categorical Langlands correspondence is in the article of Fargues--Scholze \cite{fargues2024}. They defined the functor $\widetilde{\operatorname{LL}}_{G^*}$ as the conjecture inverse of the functor in \cite{fargues2024}, Conjecture X.1.4. 
They also construct a correspondence for functions
\[
\CO(\Par_{G^\vee}) \to \CZ(D\mathsf{Rep}(G(F)))
\]
in \cite{fargues2024}, Corollary IX.0.3 (for general $G$). By their constructions, the categorical correspondence and the correspondence for functions are compatible. The assumption (1) holds for functions by \cite{fargues2024}, Corollary IX.7.3 and it is conjectured to be true categorically in \cite{hansen_cllc}, Conjecture 5.4.1. The assumption (2) holds for functions by \cite{fargues2024}, Theorem IX.6.1 and it is natural to hope that is it true categorically. The assumption (3) is conjecture to be true in \cite{hansen_beijing}, Conjecture 1.4.5 (it is also expected to be true classically in \cite{kaletha2016local}, Conjecture F).

We present some corollaries of the conjectural categorical local Langlands correspondence that will be used later.  

The correspondence induces an equivalence of blocks. 

\begin{proposition}
   \label{equivalence-blocks}
    Let $\pi$ be an irreducible, supercuspidal representation of $G(F)$
    such that its corresponding semisimple L-parameter $\phi_\pi \in \Par_{G^\vee}$ is elliptic. Denote by $\mathsf{Rep}(G(F))_\pi$ the block of the representation $\pi$ (in which the irreducible objects are the unramified twists of $\pi$). Then there exists a unique irreducible representation $V_\pi$ of $C_{G^\vee}(\phi_\pi)$ such that the categorical local Langlands correspondence $\operatorname{LL}_G$ restricts to an equivalence of categories
    \[
    D(\mathsf{Rep}(G(F))_\pi) \simeq D\Coh(\Par_{G^\vee,\phi_\pi})_{V_\pi}.
    \] 
\end{proposition}

In particular, this equivalence induces an isomorphism of Bernstein centers
\[
\CO(\pi) := \CZ(D(\mathsf{Rep}(G(F))_\pi)) \simeq \CZ(D\Coh(\Par_{G^\vee,\phi_\pi})_{V_\pi}) \simeq \CO(\Par_{G^\vee,\phi_\pi}),
\]
where the last identity is from Lemma \ref{gerbe} and \ref{bernstein}.
Moreover, by assumption (3) of Conjecture \ref{categorical}, the $Z(G^\vee)^{W_F}$-character of $V_\pi$ is $\kappa(b)$.

\begin{proof}
    The correspondence for functions induces a morphism of algebraic varieties, from the variety of unramified twists of $\pi$ to the coarse moduli of $\Par_{G^\vee}$.
    By connectedness, the image of this morphism lies in the coarse moduli of the connected component $\Par_{G^\vee,\phi_\pi}$. Thus for any $V \in D(\mathsf{Rep}(G(F))_\pi)$, the support of $\operatorname{LL}_G(V)$ is contained in $\Par_{G^\vee, \phi}$.
    As $\phi_\pi$ is elliptic, $\Par_{G^\vee, \phi}$ is disjoint with the image of $\Par_{P^\vee}$ for any proper parabolic subgroup $P$ of $G$. Thus the image of $\operatorname{LL}_G(V)$ under the (putative) right adjoint of $\Eis_{P^\vee}$ is $0$ for any $V \in D(\mathsf{Rep}(G(F))_\pi)$. 

    Hence for any proper parabolic subgroup $P$ of $G$, $\CT^a_{P*}i_{b!}(V)=0$, where $\CT_*$ is the constant term operator defined in \cite{hamann}, Definition 2.1.8, which it is the right adjoint of $\Eis^a_{P!}$. Thus 
    \[
    i_{b!}(D(\mathsf{Rep}(G(F))_\pi)) \subset D(\operatorname{Bun}_G)^{\mathrm{cusp}}
    \]
    where the latter is defined by vanishing of all such $\CT_{P*}^a$'s and proved to be a direct summand of $D(\operatorname{Bun}_G)^{\mathrm{cusp}}$ in \cite{hamann}, Theorem 1.3.2. Again by \cite{hamann}, Theorem 1.3.2, 
    \[
    i_{b!}(D(\mathsf{Rep}(G(F))) \cap D(\operatorname{Bun}_G)^{\mathrm{cusp}} \subset D(\operatorname{Bun}_G)^{\mathrm{cusp}}
    \]
    is a direct summand. As $D(\mathsf{Rep}(G(F))_\pi)$ is a block of $D(\mathsf{Rep}(G(F)))$, $i_{b!}(D(\mathsf{Rep}(G(F))_\pi))$ is a block of $D(\operatorname{Bun}_G)$ and it corresponds to a block of $D\Coh(\Par_{G^\vee,\phi_\pi})$. By Lemma \ref{gerbe}, $\Par_{G^\vee, \phi_\pi}$ is a $C_{G^\vee}(\phi_\pi)$-gerbe and thus there exists such a $V_\pi$.
\end{proof}

Moreover, we can prove that there is an isomorphism between the relative Weyl groups.

\begin{proposition}
\label{correspondence-Weyl}
    Let $\pi$ be an irreducible, supercuspidal representation of $M(F)$
    such that its corresponding semisimple L-parameter $\phi_\pi \in \Par_{M^\vee}$ is elliptic. Denote by $V_\pi$ the representation of $C_{M^\vee}(\phi_\pi)$ as in Conjecture \ref{equivalence-blocks}.
    Then under the isomorphism 
    \[N_{G(F)}(M(F))/M(F)\simeq N_{{}^LG}({}^LM)/{}^LM\]
    of Weyl groups,
    there is an isomorphism of relative Weyl groups
    \[
    N_{G^\vee}(M^\vee, \phi, V_\pi)M^\vee/M^\vee \simeq N_G(M, \pi)/M
    \]
    where the group $N_{G^\vee}(M^\vee, \phi_\pi, V_\pi)$ is defined in Definition \ref{filtration-general} and the group $N_G(M, \pi)$ is
    \[
    \{g\in N_{G(F)}(M(F))\mid \pi\circ g \simeq \chi \pi, \chi \in X^*_{\text{ur}}(M(F))\},
    \]
    such that their action on the isomorphic rings $\CO(\Par_{M,\phi})\simeq \CO(\pi)$ are compatible.
\end{proposition}

\begin{proof}
    For an element $w$ in these isomorphic Weyl groups, $w$ defines an automorphism of $M(F)$ on the automorphic side and its corresponding automorphism of ${}^LM$ on the spectral side. Hence by the compatibility of the categorical correspondence with isomorphisms, $\phi_{w(\pi)}=w(\phi_\pi)$ and $V_{w(\pi)} = w(V_\pi)$.

    By the equivalence of blocks \ref{equivalence-blocks}, these two representations $\pi, w(\pi)$ are in the same block of $\mathsf{Rep}(M(F))$ if and only if their corresponding L-parameters $\phi_\pi$ and $\phi_{w(\pi)}$ are in the same connected component and $V_\pi=V_{w(\pi)}$. Hence the relative Weyl groups are isomorphic.

    By the compatibility of the geometric correspondence with isomorphisms, the action of the relative Weyl group on the isomorphic rings $\CO(\Par_{M^\vee,\phi})\simeq \CO(\pi)$ are compatible.
\end{proof}

\subsection{Proof of the compatibility result}

We first recall the definition of the classical intertwining operator in \cite{Dat2005NuTempered}, (2.11).

\begin{definition}[Classical intertwining operator]
    Let $G$ be a reductive group over a $p$-adic local field $F$, $P$, $Q$ are two parabolic subgroups of $G$ such that their Levi subgroups are the same, denoted $M$.
    Then there are (normalized) parabolic induction functors
    \[
    i_P^G, i_Q^G \colon \mathsf{Rep}(M(F)) \to \mathsf{Rep}(G(F))
    \]
    and parabolic restriction functors 
    \[
    r_P^G, r^G_Q \colon \mathsf{Rep}(G(F)) \to \mathsf{Rep}(M(F)).
    \]

    By Mackey's formula, there exists a filtration $F^{\leq w}$ of $r_Q^G \circ i_P^G$ 
    indexed by $w \in N_G(M)/M$ and $F^{\leq 1}/F^{<1}\simeq \mathrm{id}$.
    It is proved that the inclusion
    \[
    \mathrm{id} \to r_Q^G \circ i_P^G/F^{<1}
    \]
    has a unique \textbf{rational} section $r_Q^G \circ i_P^G/F^{<1} \to \mathrm{id}$.
    The adjunction of 
    \[
    r_Q^G \circ i_P^G \to r_Q^G \circ i_P^G/F^{<1} \to \mathrm{id},
    \]
    $J_{Q|P} \colon i_P^G \to i_Q^G$ is called the unnormalized intertwining operator.
\end{definition}

Now we can state the compatibility theorem between classical and spectral intertwining operators. Let $G = G^*_b$ be as in the previous subsection.
Let $\pi$ be an irreducible, supercuspidal representation of $M(F)$ such that its corresponding semisimple L-parameter $\phi_\pi \in \Par_{M^\vee}$ is elliptic. By Proposition \ref{equivalence-blocks}, it corresponds to an irreducible representation $V_\pi$ of $C_{M^\vee}(\phi_\pi)$. Denote by $M^*$ the quasisplit inner form of $M$. Note that for $b_M \in B(M^*)_{bas}\simeq X^*(Z(M^\vee)^{W_F})$ corresponding to $M$, it comes from $b_G \in B(G^*)_{bas}\simeq X^*(Z(G^\vee)^{W_F})$. Thus $\kappa(b_M)$ is invariant under the Weyl group action. Moreover, since $\phi_\pi$ is elliptic, $(Z(M^\vee)^{W_F})^\circ = (C_{M^\vee}(\phi_\pi))^\circ$. Thus the representation $V_\pi$ satisfies the condition in section 3.4 and 3.5 and we can define the intertwining operators over the category $D\Coh(\Par_{M^\vee,\phi_\pi})_{V_\pi}$.

\begin{theorem}[Compatibility]
\label{compatibility}
    Under the commutative diagram
    \[\begin{tikzcd}
	\mathsf{Rep}(M(F))_\pi & D\Coh(\Par_{M^\vee, \phi_\pi})_{V_\pi} \\
	\mathsf{Rep}(G(F))_\pi & D\Coh(\Par_{G^\vee, \phi_\pi}),
	\arrow["\operatorname{LL}_M", from=1-1, to=1-2]
	\arrow["{i_{\bar{P}}^G, i_{\bar{Q}}^G}", from=1-1, to=2-1]
	\arrow["{\Eis_P, \Eis_Q}", from=1-2, to=2-2]
	\arrow["\operatorname{LL}_G",from=2-1, to=2-2]
    \end{tikzcd}\]
    the intertwining operator satisfies the following compatibility:
    for any $\pi_0 \in \mathsf{Rep}(M(F))_\pi$,
    \[
    \operatorname{LL}_G(J_{\bar{Q}|\bar{P}}\pi_0) = J_{Q^\vee|P^\vee}(\operatorname{LL}_M\pi_0)
    \]
    up to a unit in $\CO(\pi)\simeq \CO(\Par_{M,\phi})$.
\end{theorem}

\begin{proof}
    On the automorphic side, the intertwining operator is defined as the adjunction of the rational splitting of the inclusion of functors from $\mathsf{Rep}(M(F))_\pi$,
    \[
    \mathrm{id} \to r_{\bar{Q}}^G \circ i_{\bar{P}}^G/F^{<1},
    \]
    and by Theorem \ref{mul-general}, on the spectral side, the intertwining operator is the adjunction of the rational splitting of the following transformation of functors from $D\Coh(\Par_{M^\vee,\phi_\pi})_{V_\pi}$, 
    \[
    \mathrm{id} \to \CT_{Q^\vee} \circ \Eis_{P^\vee}/F^{<1}.
    \]
    By fully-faithfulness of the functor $\operatorname{LL}_M$ and $\operatorname{LL}_G$, $\CT_{Q^\vee}$ corresponds to $r_{\bar{Q}}^G$ and thus
    \[
    \operatorname{LL}_M(r_{\bar{Q}}^G \circ i_{\bar{P}}^G(\pi_0)) \simeq \CT_{Q^\vee} \circ \Eis_{P^\vee}(\operatorname{LL}_{M}\pi_0).
    \]
    As the inclusion in Theorem \ref{mul-general} is a generator of the $\CO(\Par_{M^\vee,\phi})$-module of rank $1$, it is unique up to a unit in $\CO(\Par_{M^\vee,\phi})$.

    Moreover, after taking the direct summand of $r_{\bar{Q}}^G \circ i_{\bar{P}}^G$ in the block $D(\mathsf{Rep}(M(F)))_\pi$ and taking the direct summand of $\CT_{Q^\vee} \circ \Eis_{P^\vee}$ in the block $D\Coh(\Par_{M^\vee,\phi_\pi})_{V_\pi}$, their filtrations reduce to 
    $N_{G^\vee}(M^\vee, \phi, V_\pi)M^\vee/M^\vee \simeq N_G(M, \pi)/M$-filtrations.
    Thus it suffices to prove that the reduced filtrations correspond to each other under the isomorphic relative Weyl groups in Proposition \ref{correspondence-Weyl} and it suffices to prove this statement for those $\pi_0$ with locally free $\operatorname{LL}_M\pi_0$, as locally free sheaves in $D\Coh(\Par_{M^\vee,\phi_\pi})_{V_\pi}$ generate this category.
    For those $\pi_0$, the filtrations are ``usual'' filtrations in abelian categories.

    Note that the order of the filtration are the same: $\overline{PwQ}\subset\overline{Pw'Q}$ if and only if 
    $\overline{P^\vee wQ^\vee}\subset\overline{P^\vee w'Q^\vee}$. Moreover, the filtrations satisfy that the subquotients are compatible:
    \begin{itemize}
        \item The subquotient of $r_{\bar{Q}}^G \circ i_{\bar{P}}^G$ at the stratum $w$ satisfies $F^{\geq w}/F^{>w} \simeq w^*$ by Mackey's formula.
        \item The subquotient of $\CT_{Q^\vee} \circ \Eis_{P^\vee}$, at the stratum $w$ satisfies $F^{\geq w}/F^{>w} \simeq w^*$ by Corollary \ref{stratum w}.
    \end{itemize}
    Hence by compatibility of the action of Weyl group, the filtrations on both sides of the formula
    \[
    \operatorname{LL}_M((r_{\bar{Q}}^G \circ i_{\bar{P}}^G \pi_0)_\pi)\simeq (\CT_{Q^\vee} \circ \Eis_{P^\vee}(\operatorname{LL}_M(\pi_0)))_{V_\pi}
    \]  
    (where the subscripts denotes the corresponding direct summands)
    correspond to each other, as such a filtration is unique.
\end{proof}

By definition of the unnormalized intertwining operator on the spectral side (Definition \ref{unnormalized}), we verify the Langlands conjecture \ref{conjecture-Langlands} under the conjectural local Langlands correspondence, after changing the convention of the Frobenius as in Remark \ref{convention}. 